\documentclass[11pt]{article}
\usepackage[margin=1.2in]{geometry}
\usepackage[utf8]{inputenc}

\usepackage{mathrsfs}
\usepackage{amsmath}
\usepackage{hyperref}
\usepackage{amsthm}
\usepackage{amssymb}
\usepackage{parskip}
\usepackage{xcolor}
\usepackage{subcaption}
\usepackage{tikz}
\usetikzlibrary{arrows}
\usetikzlibrary{arrows.meta}
\definecolor{wwhhii}{rgb}{1.,1.,1.}
\definecolor{rreedd}{rgb}{1.,0.,0.}
\definecolor{uuuuuu}{rgb}{0.26666666666666666,0.26666666666666666,0.26666666666666666}

\title{The directed landscape from Brownian motion}
\author{Duncan Dauvergne \and B\'alint Vir\'ag}

\newcommand{\Z}{\mathbb{Z}}
\newcommand{\Q}{\mathbb{Q}}
\newcommand{\N}{\mathbb{N}}
\newcommand{\R}{\mathbb{R}}
\newcommand{\p}{\mathbb{P}}
\newcommand{\E}{\mathbb{E}}

\newcommand{\indic}{\mathbf{1}}

\newcommand{\fl}[1]{{\left\lfloor #1 \right\rfloor}}
\newcommand{\cl}[1]{{\left\lceil #1 \right\rceil}}
\newcommand{\sset}{\subset}
\newcommand{\lf}{\left}
\newcommand{\rg}{\right}

\newcommand{\ga}{\gamma}
\newcommand{\Ga}{\Gamma}
\newcommand{\ep}{\epsilon}

\newcommand{\ka}{\kappa}
\newcommand{\de}{\delta}

\newcommand{\sig}{\sigma}

\newcommand{\la}{\lambda}
\newcommand{\al}{\alpha}
\newcommand{\Om}{\Omega}

\newcommand{\Rd}{\mathbb{R}^4_\uparrow}

\newcommand{\cB}{\mathcal{B}}
\newcommand{\cC}{\mathcal{C}}
\newcommand{\cD}{\mathcal{D}}

\newcommand{\cF}{\mathcal{F}}

\newcommand{\cI}{\mathcal{I}}
\newcommand{\cJ}{\mathcal{J}}
\newcommand{\cK}{\mathcal{K}}
\newcommand{\cL}{\mathcal{L}}

\newcommand{\cN}{\mathcal{N}}

\newcommand{\cP}{\mathcal{P}}
\newcommand{\cQ}{\mathcal{Q}}

\newcommand{\cS}{\mathcal{S}}

\newcommand{\cW}{\mathcal{W}}

\newcommand{\fB}{\mathfrak{B}}

\usepackage{MnSymbol,enumerate}

\DeclareMathOperator*{\argmax}{arg\,max}

\newcommand{\eqd}{\stackrel{d}{=}}

\newcommand{\cvgd}{\stackrel{d}{\to}}
\newcommand{\cvgp}{\stackrel{\mathbb{P}}{\to}}

\newcommand{\X}{\times}

\newcommand{\cvgdown}{\downarrow}

\newcommand{\id}{\text{id}}

\newcommand{\smin}{\setminus}

\newcommand{\bw}{\mathbf{w}}
\newcommand{\ba}{\mathbf{a}}
\newcommand{\bb}{\mathbf{b}}
\newcommand{\bx}{\mathbf{x}}
\newcommand{\bn}{\mathbf{n}}
\newcommand{\bm}{\mathbf{m}}
\newcommand{\by}{\mathbf{y}}
\newcommand{\bz}{\mathbf{z}}
\newcommand{\bp}{\mathbf{p}}
\newcommand{\bq}{\mathbf{q}}

\renewcommand{\P}{\mathbb{P}}

    \newtheorem{theorem}{Theorem}[section]
    \newtheorem{lemma}[theorem]{Lemma}
    \newtheorem{proposition}[theorem]{Proposition}
    \newtheorem{corollary}[theorem]{Corollary}
    
    \newtheorem{conjecture}[theorem]{Conjecture}

\theoremstyle{definition} % For roman text in the body
    
    \newtheorem{remark}[theorem]{Remark}
    \newtheorem{example}[theorem]{Example}
       
    \newtheorem{problem}[theorem]{Problem}

\theoremstyle{remark} % For an italic header, more subtle than definition style

\newcommand\ee{\end{equation}}

\newcommand{\comment}[1]{}

\newcommand{\sm}{{\raise0.3ex\hbox{$\scriptstyle \setminus$}}}

\newcommand{\btau}{{\boldsymbol \tau}}
\newcommand{\bbtau}{\bar {\boldsymbol \tau}}

\usepackage{hyperref}
\usepackage[nottoc,numbib]{tocbibind}
\hypersetup{
	colorlinks=true,
	linkcolor=blue,
	citecolor=blue,
	filecolor=magenta,      
	urlcolor=cyan,
	linktoc=page
}
\title{The directed landscape from Brownian motion}
\author{Duncan Dauvergne and B\'alint Vir\'ag}
\begin{document}
	\maketitle

\begin{abstract}
We construct an almost sure bijection that recovers the directed landscape on the half-plane from a sequence of independent Brownian motions. This map is the natural scaling limit of the Robinson--Schensted--Knuth (RSK) correspondence. The Brownian motions arise as the marginals of the multi-path stationary horizon associated with the directed landscape. The inverse map is fully explicit and yields a natural coupling in which Brownian last-passage percolation converges in probability to the directed landscape.

As an application, we prove that the directed landscape restricted to a strip can be reconstructed from the parabolic Airy line ensemble, resolving a conjecture of the first author and Zhang.

Along the way we develop two new versions of RSK in the semi-discrete setting, introduce a general theory of sorting via Pitman operators that generates a faithful action of the biHecke monoid, and establish key identities for the multi-path stationary horizon for both the directed landscape and Brownian last-passage percolation.
\end{abstract}
	\tableofcontents

\section{Introduction}

Since Robinson's 1938 paper \cite{Robinson1938}, versions of the Robinson--Schensted--Knuth (RSK) correspondence have been central to representation theory, algebraic combinatorics, and probability. The classical Robinson-Schensted correspondence gives a bijection between permutations and pairs of standard Young tableaux of the same shape. 

Via Greene's theorem this bijection admits a last-passage percolation interpretation on the permutation matrix: the shape records the lengths of the longest disjoint increasing subsequences, realized as maximal path sums in the matrix between a corner and two boundaries. This last-passage viewpoint extends naturally to general environments and forms the foundation for the continuum RSK theory, the main topic of this paper.

Baik, Deift, and Johansson showed that the normalized length of the longest increasing subsequence under the uniform measure on permutations converges to the GUE Tracy-Widom distribution \cite{baik1999distribution}. More generally, the full directed metric given by last-passage values on \([n]^2\) itself converges in the scaling limit to the directed landscape \(\mathcal{L}\) on the plane,  \cite{dauvergne2021scaling}.

By the KPZ universality conjecture, models with homogeneous locally independent planar noise are expected to scale to \(\mathcal{L}\). This has been established for several exactly solvable models, many of which possess an underlying RSK structure. We therefore expect pre-limiting RSK correspondences to converge to a limiting RSK theory for the directed landscape. 

The main technical achievement of this paper is to establish this theory and to prove that the limiting map remains a bijection. Since classical inverses lose all meaning in the limit, we construct the inverse using a suite of new geometric, probabilistic, and algebraic results.

In the rest of the introduction we highlight the main results; more precise statements appear in Section \ref{S:results}.

\subsection{RSK in the directed landscape on a bounded time interval: existing and new results}

The original construction of the directed landscape relies on a version of RSK restricted to the time interval \([0,1]\). Here the analogue of the pair of Young tableaux is the parabolic Airy line ensemble \(\mathfrak{A}\). It has long been known that \(\mathcal{L}(0,0;x,1)=\mathfrak{A}_1(x)\). More generally, \cite{dauvergne2021disjoint} established the multi-path formula
\[
\sum_{i=1}^k \mathfrak{A}_i(x)=\mathcal{L}(0,0^k;x^k,1),
\]
where the right-hand side denotes the maximal $\mathcal L$-length of $k$ disjoint paths between $(0,0)$ and $(x,1)$.
This is the direct analogue of the RSK map on the strip. Dauvergne and Zhang have the following natural conjecture.
\begin{conjecture}[\cite{dauvergne2021disjoint}, Conjecture 1.10]\label{c:1}
The map $\mathcal L|_{t\in[0,1]}\mapsto \mathfrak A$ is invertible. 
That is, for \(s<t\in[0,1]\) and \(x,y\in\mathbb{R}\), \(\mathcal{L}(x,s;y,t)\) is a measurable function of \(\mathfrak{A}\).
\end{conjecture}

The original construction of $\mathcal L$ gives hope: \(\mathcal{L}(x,0;y,1)\) can be recovered as a last-passage limit in the environment  \(\mathfrak{A}\). Full bijectivity on the strip has remained open. As a corollary (Corollary \ref{C:strip-reconstruction}) to our main result, we get the following.
\begin{theorem}[RSK on the strip]
\label{T:Airy-RSK-intro}
RSK on the strip is an almost sure bijection: Conjecture \ref{c:1} holds.  \end{theorem}

In the strip case the inverse is not fully explicit, leaving open the following problem. 
\begin{problem}
Express the inverse map in terms of last passage on \(\mathfrak{A}\).
\end{problem}

The RSK correspondence simplifies dramatically, and the inverse becomes fully explicit, when the directed landscape is restricted to a half-plane \(\{t\le 0\}\) (the case \(\{t\ge 0\}\) is symmetric).

\subsection{The main result:  RSK for the directed landscape on \texorpdfstring{$\{t\le 0\}$}{tle0}}

For the lower half-plane the RSK map may be viewed as a normalized limit of the strip map. It sends \(\mathcal{L}|_{\{t\le 0\}}\) to a sequence of paths \(B_i\) defined by the {\bf multi-path Busemann function}
\begin{equation}
\label{e:Bdef-intro}
\sum_{i=1}^k B_i(x) = \lim_{t\to-\infty} \mathcal{L}(0^k,t;x^k,0) - \mathcal{L}(0^k,t;0^k,0).
\end{equation}
Busemann techniques together with the Brownian Gibbs property yield the following (part of Theorem \ref{T:landscape-recovery-intro}).

\begin{proposition}
The limits in \eqref{e:Bdef-intro} exist almost surely and the  \(B_i\) are independent two-sided Brownian motions (of variance 2, a convention in KPZ and here).
\end{proposition}

Inverting the map \(\mathcal{L}|_{\{t\le 0\}}\mapsto B\) therefore gives a natural construction of the directed landscape from independent randomness. This is the main result of the paper.

To state the inverse we work in the last-passage percolation setting. Let \(f\) be a sequence of continuous functions. The last-passage value \(f[(x,n)\to(y,m)]\) (and its multi-path version) is defined in the usual way, see \eqref{E:lpp-defn-intro}. The fundamental convergence result of \cite{DOV} is:

\begin{theorem}[Theorem 1.5, \cite{DOV}]
\label{T:DL-cvg-intro}
For every \(a\in\mathbb{R}\), let \(B^a\) be an environment of independent two-sided Brownian motions of variance 2 and common drift \(a\). With the scaling \((x,s)_a=(x-as/4,\lfloor sa^3/8\rfloor+1)\), we have
\[
\mathcal{L}^a(x,s;y,t):=B^a[(x,s)_a\to(y,t)_a]-\frac{a^2}{4}(t-s)\to\mathcal{L}(x,s;y,t)
\]
in distribution as \(a\to-\infty\).
\end{theorem}

In our coupling the environments \(B^a\) are constructed directly from a fixed zero-drift Brownian environment \(B\) via multi-path Busemann functions. For \(a\ge 0\), define the {\bf Busemann shear}
\[
\sum_{i=1}^k B^a_i(x)=\lim_{t\to-\infty}B[(t,\lfloor\theta |t|\rfloor)^k\to(x,1)^k]-B[(t,\lfloor\theta t\rfloor)^k\to(0,1)^k],\qquad\theta=a|a|/2,
\]
with the case \(a<0\) defined analogously by reversal, see  \eqref{e:reversed-LP} and Figure \ref{fig:BrownianRSK}.

\begin{theorem}[Main]
\label{T:RSK-intro}\label{t:main-intro}
RSK for the directed landscape on \(\{t\le 0\}\) admits an explicit inverse. With \(B^a\) defined as above, the convergence of Theorem \ref{T:DL-cvg-intro} holds in probability on \(\{s<t\le 0\}\). Moreover, \(B\) and \(\mathcal{L}|_{\{t\le 0\}}\) are coupled through the Busemann relation \eqref{e:Bdef-intro}.
\end{theorem}

For \(a\neq 0\) the Busemann shear \(B\mapsto B^a\) itself forms an RSK correspondence on the Brownian environment and satisfies a simple group law.

\begin{theorem}
For \(a,b\in\mathbb{R}\) we have \((B^a)^b=B^{a+b}\) almost surely. In particular, \((B^a)^{-a}=B\).
\end{theorem}

\subsection{The multi-path stationary horizon}

The map sending the zero-drift environment \(B\) to the first \(k\) lines of \(B^a\) is closely related to the {\bf  stationary horizon}. In the directed landscape setting, the stationary horizon was defined in \cite{busani2024stationary} as the joint law of Busemann functions in several directions at a fixed time.

We develop this theory by viewing the independent Brownian motions appearing in the horizon construction as differences of multi-path Busemann functions in the directed landscape. The following result highlights the structure (a special case of Theorems \ref{T:Busemann-law-intro} and \ref{T:landscape-recovery-intro}). For $\theta \in \mathbb R^k$ increasing, let $\mathfrak B^\theta$ denote the {\bf multi-path stationary horizon},  the $k$-path Busemann function in direction $\theta$ defined through the natural generalization \eqref{E:cL-busemann-intro} of \eqref{e:Bdef-intro}.

\begin{theorem}
\label{t:bus-intro} 
\begin{equation}
\label{e:ind-intro}
\mathfrak B^{\theta_1}, \mathfrak B^{(\theta_1, \theta_2)}-\mathfrak B^{\theta_1}, \ldots, \mathfrak B^{\theta}-\mathfrak B^{(\theta_1,\ldots,\theta_{k-1})}
\end{equation}
 are independent Brownian motions with drifts $\theta_1/2,\ldots \theta_k/2$ respectively.
For the $2^k-1$ nonempty subsequences $\eta$ of $\theta$, $\mathfrak B^\eta$ can be recovered from \eqref{e:ind-intro} via an action of the biHecke monoid $\mathfrak D_k$.
\end{theorem}

Up to parametrization, the same result holds for Brownian last passage percolation, see Theorem \ref{T:Busemann-law-intro}.
Our main algebraic contribution is the introduction of the action of the biHecke monoid $\mathfrak D_k$ on last passage percolation. This is explained in detail in the next subsection.

\subsection{Proof strategy}
\label{SS:strategy}

The proof of Theorem \ref{T:RSK-intro} has four main steps and requires a combination of new algebraic, probabilistic and geometric ideas.

The first step is devoted to developing an algebraic language built around the biHecke monoid. Underlying the algebra is the observation that RSK is fundamentally a sorting operation, see \cite{biane2005littelmann}. The Pitman transform acts as a conditional swap for neighboring disordered lines. Iterating these operators generates a faithful action of the {\bf 0-Hecke monoid} \(\mathfrak{M}_n\) introduced in \cite{Norton1979}. Adjoining the reverse operators yields the {\bf biHecke monoid} \(\mathfrak{D}_n\), first introduced in \cite{HST13}. We prove that the Pitman and reverse Pitman operators give an action of \(\mathfrak{D}_n\) in the last-passage setting. This algebraic structure is essential for establishing isometries and identifying inverses.

The second step proves generalized isometries relating multi-path Busemann values in the landscape and in the Brownian environment.  A specific case of this isometry states that the drifted Brownian motions $B^a$, defined as multi-path Busemann values in the environment $B$ in Theorem \ref{t:main-intro} are in fact multi-path Busemann values in $\mathcal L$:
\begin{theorem}[Isometry between $\mathcal L|_{t\le 0}$ and $B$, part 1]
\begin{equation}
\sum_{i=1}^k B_i^a(x) = \lim_{t \to -\infty} \cL((at/2,t)^k; (x, 0)^k) - \cL((at/2,t)^k; (0, 0)^k)
\end{equation}
\end{theorem}

The third step supplies the core geometric argument of the paper. For large \(k\) the \(k\)-path optimizers occupy a remarkably rigid, almost deterministic ``truss'' shape. By carefully sending starting points to \(-\infty\) at a controlled rate and driving directions also to \(-\infty\), these trusses converge in the \(k\to\infty\) limit to the horizontal line \(\ell = [-\infty,0] \times \{-1\}\). \usetikzlibrary{arrows.meta, decorations.pathmorphing, positioning, calc}

\begin{figure}[htbp]
\centering
\definecolor{forestgreen}{rgb}{0.13, 0.55, 0.13}
\begin{tikzpicture}[
    scale=0.97,
    thick,
    >=Stealth,
    every node/.style={font=\small}
]

% Horizontal time lines
\draw[thick] (-4.8,4.6) -- (4.8,4.6) node[right] {\(t=0\)};
\draw[dashed, gray!70] (-4.8,2.1) -- (4.8,2.1) node[right] {\(t=-1\)};

% Time direction
\draw[gray!60, ->] (-5.3,4.7) -- (-5.3,1.0);
\node[left, gray!70] at (-5.3,0.7) {\(t \to -\infty\)};

% ==================== MAIN BLUE BUNDLE ====================
% All paths start at the exact same point: (-4.5,4.45)

% Top path (highest / widest arch)
\draw[blue!80, thick] 
    (-4.5,4.45) 
    .. controls (-3.0,4.35) and (-1.2,3.95) .. (1.9,3.05)   % TWEAK HERE: first control y (4.35) and second x (-1.2)
    .. controls (2.7,2.55) and (1.7,1.75) .. (-0.65,1.43)    % TWEAK HERE: max right reach (1.9 and 2.7)
    .. controls (-2.3,1.03) and (-3.6,0.73) .. (-4.6,0.33);

% Second path
\draw[blue!75, thick] 
    (-4.5,4.45) 
    .. controls (-3.0,4.25) and (-1.3,3.80) .. (1.6,2.95)   % TWEAK HERE
    .. controls (2.5,2.45) and (1.55,1.85) .. (-0.75,1.48)  % TWEAK HERE
    .. controls (-2.4,1.08) and (-3.65,0.78) .. (-4.6,0.38);

% Third path
\draw[blue!70, thick] 
    (-4.5,4.45) 
    .. controls (-3.0,4.15) and (-1.4,3.65) .. (1.3,2.85)   % TWEAK HERE
    .. controls (2.3,2.35) and (1.4,1.95) .. (-0.85,1.53)   % TWEAK HERE
    .. controls (-2.5,1.12) and (-3.7,0.83) .. (-4.6,0.43);

% Bottom blue path (lowest / narrowest arch)
\draw[blue!65, thick] 
    (-4.5,4.45) 
    .. controls (-3.0,4.05) and (-1.5,3.50) .. (1.0,2.75)   % TWEAK HERE
    .. controls (2.1,2.25) and (1.25,2.05) .. (-0.95,1.58)  % TWEAK HERE
    .. controls (-2.6,1.18) and (-3.75,0.88) .. (-4.6,0.48);

% ==================== EXTRA RED PATH (strictly below all blue) ====================
\draw[forestgreen, very thick] 
    (2.3,4.6) .. controls (2.4,4.0) and (2.6,3.4) .. (2.6,3.15)
    .. controls (2.6,2.55) and (1.85,1.70) .. (-0.55,1.35)
    .. controls (-2.2,0.92) and (-3.55,0.62) .. (-4.6,0.25);

    % (-4.5,4.45) 
    % .. controls (-3.0,4.35) and (-1.2,3.95) .. (1.9,3.05)   % TWEAK HERE: first control y (4.35) and second x (-1.2)
    % .. controls (2.7,2.55) and (1.7,1.75) .. (-0.65,1.43)    % TWEAK HERE: max right reach (1.9 and 2.7)
    % .. controls (-2.3,1.03) and (-3.6,0.73) .. (-4.6,0.33);

% Point and label y
\filldraw[black] (2.3,4.6) circle (1.6pt);
\node[above=11pt] at (2.3,4.6) {\(y\)};

% Red bidirectional arrow
\draw[->, thick, forestgreen] (1.1,5.05) -- (3.5,5.05);
\draw[<-, thick, forestgreen] (1.1,5.05) -- (3.5,5.05);

% Labels
%\node[blue!75, left, font=\footnotesize] at (-4.9, 3.9) {main paths};
\node[forestgreen, right, font=\small] at (-4.35, 0.05) {extra optimizer from \(y\)};
\node[blue!75, above, font=\small] at (-0.9, 2.35) {truss};

\end{tikzpicture}

% \caption{Truss formed by large-\(k\) disjoint optimizers in the directed landscape.}
% \label{fig:truss}
% \end{figure}

% \end{document}
\caption{Truss formed by large-\(k\) disjoint optimizers in the directed landscape.}
\label{f:truss}
\end{figure}

Introducing a new optimizer started from a bounded point \(y\) then reveals that the incremental change in the \((k+1)\)-path optimizer is essentially identical to the change in \(\mathcal{L}(0,-1;y,0)\), because only the portion after time \(-1\) varies significantly while the earlier portion is forced to hug the truss, see Figure \ref{f:truss}. This rigidity allows full reconstruction of \(\mathcal{L}(x,s;y,0)\) for all \(x,y\) and \(s<0\). A two-slit variant extends the reconstruction to the complete \(\mathcal{L}(x,s;y,t)\) for \(s<t\le 0\).

The final step employs a novel abstract measure-theoretic argument to upgrade convergence in law to convergence in probability under the explicit coupling.

\subsection{Future research directions}

We hope this work opens new avenues both for understanding the directed landscape and for proving convergence results. Several directions appear particularly promising with current techniques.

\begin{problem}
Construct the directed landscape directly from Brownian last passage percolation without any appeal to determinantal formulas.
\end{problem}

\begin{problem}
Upgrade the convergence in Theorem \ref{T:RSK-intro} to quantitative rates: for compact \(K\subset\mathbb{H}^2_\uparrow\) show the existence of \(\alpha>0\) such that
\[
\mathbb{E}\sup_{u\in K}|\mathcal{L}^a(u)-\mathcal{L}(u)|\le |a|^{-\alpha+o(1)}
\]
as \(a\to-\infty\), and optimize \(\alpha\).
\end{problem}

Discrete models whose multi-type stationary distributions admit queuing-theoretic descriptions (colored TASEP, q-PushTASEP and relatives) are natural candidates for an RSK-based convergence proof to the landscape. 
\begin{problem}
\label{P:problem-three}
Prove that discrete particle systems whose multi-type stationary distributions have queuing-theoretic descriptions converge to the directed landscape.
\end{problem}

The techniques developed here also extend naturally to ``zigzag'' Brownian last-passage models with alternating drifts. 

\begin{problem}
\label{P:problem-four}
Let $B = (B_i, i \in \Z)$ be an environment of independent Brownian motions of drift $1$ if $i$ is even and drift $-1$ if $i$ is odd. For points $(x, n), (y, m) \in \R \X \Z$ with $n > m$ define
$$
B[(x, n) \rightsquigarrow (y, m)] = \sup_{z_{n+1} = x, \dots, z_m = y} \sum_{i=m}^n B_i(z_i) - B_i(z_{i+1}),
$$
where in the supremum, we require $z_{i+1} \le z_i$ if $i$ is odd, and $z_{i+1} \ge z_i$ if $i$ is even. This requirement along with the drift condition on the Brownian paths means the above supremum is always finite.
Prove that under some scaling, this model converges to the directed landscape.
\end{problem}
Establishing convergence for such models would provide a clean test case for the present ideas.

\subsection{Background and related work}
\label{S:background}

We give a brief review of the literature on KPZ, RSK, and the directed landscape, focusing on papers closely related to the present work. For a gentle introduction to KPZ suitable for a newcomer to the area, see \cite{romik2015surprising} or the introductory articles \cite{corwin2016kardar, ganguly2021random}. Review articles and books that go into more depth include \cite{ferrari2010random, quastel2011introduction,weiss2017reflected, zygouras2018some}.

The importance of RSK to problems in the KPZ universality class goes back at least to \cite{logan1977variational, vershik1977asymptotics}, who identified the law of large numbers for the longest increasing subsequence in a uniform permutation. The Baik-Deift-Johansson theorem \cite{baik1999distribution} went further, analyzing exact formulas for the Young tableaux arising from RSK to identify the one-point distribution $\cL(0,0; 0, 1)$ as a GUE Tracy-Widom random variable, see also \cite{johansson2000shape}. Pr\"ahofer and Spohn expanded on these methods \cite{prahofer2002scale} to prove convergence to the parabolic Airy line ensemble for the PNG droplet, see also \cite{johansson2003discrete}. The Pr\"ahofer-Spohn theorem can be viewed as taking the KPZ scaling limit of the Young tableaux side of the RSK correspondence. More recent approaches to studying KPZ limits have found exact formulas for multi-time distributions, see \cite{johansson2019multi, baik2017multi} and \cite{matetski2016kpz} for the construction of the KPZ fixed point.

A series of papers by O'Connell and coauthors \cite{o2001brownian, o2002representation, o2003path, biane2005littelmann} developed RSK in the semi-discrete setting and started to understand the correspondence in terms of Pitman transforms. One of the upshots of this work was a version of the RSK isometry which had also been observed in \cite{noumi2002tropical}. This isometry was used in \cite{DOV} as the key tool for constructing the directed landscape, and of course drives much of the present paper.

The relevance of Busemann functions to studying KPZ models goes back to Hoffman \cite{hoffman2005coexistence, hoffman2008geodesics} in first passage percolation. See \cite[Section 5]{auffinger201750} for discussion surrounding open problems on Busemann functions in that setting. In solvable models of last passage percolation, Busemann functions have been used to study competition and coexistence, to construct stationary solutions and establish uniqueness of these solutions, and to establish results about semi-infinite geodesics, see \cite{cator2012busemann, cator2013busemann, bakhtin2014space, georgiou2015ratios, georgiou2017geodesics, georgiou2017stationary} for a sample of recent work in this direction. 
The last-passage/queuing-theoretic description for the law of the stationary horizon in \eqref{E:W-B-intro} has its roots in the work of Ferrari and Martin \cite{ferrari2007stationary, ferrari2005multiclass} on coloured TASEP and the Hammersley process. A similar description in exponential LPP was given by Sepp\"al\"ainen and Fan \cite{fan2020joint} before being adapted to Brownian LPP \cite{seppalainen2023busemann}.  Busemann functions in the directed landscape were first studied in \cite{rahman2021infinite}. The law of the whole stationary horizon found in \cite{busani2021diffusive, seppalainen2023global}, connected to the directed landscape in \cite{busani2024stationary}, and shown to be the scaling limit for stationary measures in discrete models in \cite{busani2022scaling, busani2023scaling}. 

The use of multi-path last passage values and disjoint optimizers to extract information about underlying geometry in KPZ is a more recent phenomenon. Hammond \cite{hammond2016brownian, hammond2020exponents} first connected exact statistics about disjoint optimizers from a common point in Brownian LPP to the rarity of disjoint geodesics.  A precise combinatorial connection between disjoint optimizers and the RSK isometry was exploited in \cite{DOV}, and used there as a key tool in the original construction of the directed landscape. The theory of disjoint optimizers at the level of the directed landscape was developed in \cite{dauvergne2021disjoint}. This was used to help classify geodesic networks in the directed landscape in \cite{dauvergne202327}, and to prove structural results about the spacetime difference profile in \cite{ganguly2022fractal}. The phenomenon of optimizer rigidity for large $k$ that is so important in the proof of Theorem \ref{T:landscape-recovery-intro} was anticipated in the watermelon scaling exponents uncovered in \cite{basu2022interlacing}.

\subsection*{Acknowledgements} The authors would like to thank Timo Sepp\"al\"ainen and Evan Sorensen for several useful discussions about the paper, and Nikos Zygouras for pointing out a connection with the Sch\"utzenberger involution (see Corollary \ref{C:RSK}).

We devote the next section to a detailed and more precise description of our main results. The structure of Section \ref{S:results} follows the three consecutive RSK limits used in the proof, most  requiring substantial new ideas beyond simply passing to the limit.

\section{Main results in detail: the three limits of RSK}
\label{S:results}
 To introduce the RSK correspondence, we need to first define multi-path LPP across a functional environment. Consider a continuous function $f:\R \to \R^I$, where $I \subset \Z$ is an interval. For a nonincreasing right-continuous function $\pi:[x,y]\to \{m, \dots, n\}$ (henceforth a \textbf{path} from $(x, n)$ to $(y, m)$) define its \textbf{length} with respect to $f$ by
\begin{equation}
	\label{E:pi-f-intro}
	\|\pi\|_f = \sum_{i=m}^n f_i(t_{i-1}) - f_i(t_i)
\end{equation}
where $t_i = \inf \{t \in [x, y] : \pi(t) \le i\}$ and $t_{m-1} = y$. We then define the \textbf{last passage value} 
\begin{equation}
	\label{E:lpp-defn-intro}
	f[(x, n) \to (y, m)] = \max_{\pi:(x, n) \to (y, m)} \|\pi\|_f,
\end{equation}
where the maximum is over all paths from $(x, n)$ to $(y, m)$. Last passage percolation (LPP) is best thought of as a kind of directed metric on the set $\R \X I$ with distances distorted by the environment $f$, and where we maximize rather than minimize path length. With this perspective in mind, paths that achieve the maximum in \eqref{E:lpp-defn-intro} are called \textbf{geodesics}. We also need to define multi-path last passage values. First, for a linearly ordered set $X$, we write 
$$
X^k_\le = \{(x_1, \dots, x_k)\in X^k: x_1 \le x_2 \le \cdots \le x_k\},
$$ 
and similarly define $X^k_<, X^k_\ge, X^k_>$. Consider $\bp = (\bx, \bn), \bq = (\by, \bm) \in \R^k_\le \times I^k_\le$. We say that $\pi = (\pi_1, \dots, \pi_k)$ is a \textbf{disjoint $k$-tuple} from $\bp$ to $\bq$ if each $\pi_i$ is a path from $(x_i, n_i)$ to $(y_i, m_i)$ and if $\pi_i(t) < \pi_{i+1}(t)$ for $t \in (x_i, y_i) \cap (x_{i+1}, y_{i+1})$. If there is at least one disjoint $k$-tuple from $\bp$ to $\bq$, we can define the multi-path last passage value
$$
f[\bp \to \bq] = \sup_{\pi: \bp \to \bq} \sum_{i=1}^k \|\pi_i\|_f,
$$
where the supremum is over all disjoint $k$-tuples from $\bp$ to $\bq$. We call a $k$-tuple achieving the above supremum a \textbf{(disjoint) optimizer}. To simplify the notation for multi-path last passage values in the situations that occur most, we write $x^k := (x, \dots, x) \in \R^k$. Typically, we will consider multi-path last passage when all the points in $\bp$ lie on a common vertical or horizontal line (and similarly for $\bq$). In this case, we will write $\bp = (\bx, n) := (\bx, n^k)$ and $\bp = (x, \bn) := (x^k, \bn)$. See Figure \ref{fig:dis-opt} for illustrations of this definition.

\begin{figure}
	\centering
	\begin{tikzpicture}[line cap=round,line join=round,>=triangle 45,x=4cm,y=5cm]
		\clip(-0.15,-0.15) rectangle (2.15,1.15);
		
		% \fill[line width=0.pt,color=green,fill=green,fill opacity=0.35]
		% (0.15,0.25) -- (0.25,0.15) -- (0.05,-0.05) -- (-0.05,0.05) -- cycle;
		% \fill[line width=0.pt,color=green,fill=green,fill opacity=0.35]
		% (0.9,1.2) -- (1.,1.1) -- (0.8,0.9) -- (0.7,1.) -- cycle;

		\draw (0.,0.1) node[anchor=east]{$5$};
		\draw (0.,0.3) node[anchor=east]{$4$};
		\draw (0.,0.5) node[anchor=east]{$3$};
		\draw (0.,0.7) node[anchor=east]{$2$};
		\draw (0.,0.9) node[anchor=east]{$1$};

		\draw plot coordinates {(0.1,0.1) (0.2,0.) (0.3,0.11) (0.4,0.16) (0.5,0.04) (0.6,0.08) (0.7,0.06) (0.8,0.01) (0.9,0.02) (1.,0.14) (1.1,0.12) (1.2,0.03) (1.3,0.06) (1.4,0.) (1.5,0.04) (1.6,0.12) (1.7,0.06) (1.8,0.04) (1.9,0.05) (2.,0.07) };
		\draw plot coordinates {(0.1,0.3) (0.2,0.38) (0.3,0.21) (0.4,0.29) (0.5,0.34) (0.6,0.28) (0.7,0.26) (0.8,0.31) (0.9,0.32) (1.,0.23) (1.1,0.22) (1.2,0.33) (1.3,0.26) (1.4,0.2) (1.5,0.36) (1.6,0.33) (1.7,0.26) (1.8,0.28) (1.9,0.37) (2.,0.33) };
		\draw plot coordinates {(0.1,0.5) (0.2,0.54) (0.3,0.59) (0.4,0.46) (0.5,0.44) (0.6,0.48) (0.7,0.46) (0.8,0.51) (0.9,0.58) (1.,0.49) (1.1,0.45) (1.2,0.53) (1.3,0.56) (1.4,0.49) (1.5,0.54) (1.6,0.52) (1.7,0.56) (1.8,0.47) (1.9,0.45) (2.,0.46) };
		\draw plot coordinates {(0.1,0.7) (0.2,0.74) (0.3,0.63) (0.4,0.73) (0.5,0.64) (0.6,0.78) (0.7,0.66) (0.8,0.71) (0.9,0.65) (1.,0.63) (1.1,0.71) (1.2,0.63) (1.3,0.66) (1.4,0.77) (1.5,0.71) (1.6,0.78) (1.7,0.63) (1.8,0.64) (1.9,0.75) (2.,0.72) };
		\draw plot coordinates {(0.1,0.9) (0.2,0.98) (0.3,0.81) (0.4,0.87) (0.5,0.84) (0.6,0.91) (0.7,0.86) (0.8,0.85) (0.9,0.82) (1.,0.93) (1.1,0.82) (1.2,0.96) (1.3,0.96) (1.4,0.91) (1.5,0.88) (1.6,0.91) (1.7,0.87) (1.8,0.94) (1.9,0.95) (2.,0.93) };

		\draw [thick] [red] plot coordinates {(0.4,0.16) (0.5,0.04) (0.6,0.08) (0.7,0.06) (0.8,0.01) (0.9,0.02) (1.,0.14)};
		\draw [dashed] [red] plot coordinates {(1.,0.14) (1.,0.23)};
		\draw [thick] [red] plot coordinates {(1.,0.23) (1.1,0.22) (1.2,0.33) (1.3,0.26) (1.4,0.2) (1.5,0.36)};
		\draw [dashed] [red] plot coordinates {(1.5,0.36) (1.5,0.54)};
		\draw [thick] [red] plot coordinates {(1.5,0.54) (1.6,0.52) (1.7,0.56)};
		\draw [dashed] [red] plot coordinates {(1.7,0.56) (1.7,0.63)};
		\draw [thick] [red] plot coordinates {(1.7,0.63) (1.8,0.64) (1.9,0.75)};
		\draw [dashed] [red] plot coordinates {(1.9,0.75) (1.9,0.95)};
		\draw [thick] [red] plot coordinates {(1.9,0.95) (2.,0.93)};
		
		\draw [thick] [blue] plot coordinates {(0.1,0.1) (0.2,0.) (0.3,0.11)};
		\draw [dashed] [blue] plot coordinates {(0.3,0.11) (0.3,0.21)};
		\draw [thick] [blue] plot coordinates {(0.3,0.21) (0.4,0.29)};
		\draw [dashed] [blue] plot coordinates {(0.4,0.29) (0.4,0.46)};
		\draw [thick] [blue] plot coordinates {(0.4,0.46) (0.5,0.44) (0.6,0.48) (0.7,0.46) (0.8,0.51) (0.9,0.58)};
		\draw [dashed] [blue] plot coordinates {(0.9,0.58) (0.9,0.65)};
		\draw [thick] [blue] plot coordinates {(0.9,0.65) (1.,0.63) (1.1,0.71) (1.2,0.63) (1.3,0.66) (1.4,0.77) (1.5,0.71) (1.6,0.78)};
		\draw [dashed] [blue] plot coordinates {(1.6,0.78) (1.6,0.91)};
		\draw [thick] [blue] plot coordinates {(1.6,0.91) (1.7,0.87) (1.8,0.94)};
		
		\draw [thick] [green] plot coordinates {(0.1,0.3) (0.2,0.38)};
		\draw [dashed] [green] plot coordinates {(0.2,0.38) (0.2,0.54)};
		\draw [thick] [green] plot coordinates {(0.2,0.54) (0.3,0.59)};
		\draw [dashed] [green] plot coordinates {(0.3,0.59) (0.3,0.63)};
		\draw [thick] [green] plot coordinates {(0.3,0.63) (0.4,0.73) (0.5,0.64) (0.6,0.78)};
		\draw [dashed] [green] plot coordinates {(0.6,0.78) (0.6,0.91)};
		\draw [thick] [green] plot coordinates {(0.6,0.91) (0.7,0.86) (0.8,0.85) (0.9,0.82) (1.,0.93) (1.1,0.82) (1.2,0.96) (1.3,0.96) (1.4,0.91)};
		
	\end{tikzpicture}
	\caption{A disjoint optimizer for $k=3$ from $((0,0,0.2),5)$ to $((0.7,0.9,1),1)$. Note that in our coordinate system, we view paths as moving up and to the right across the page.}   \label{fig:dis-opt}
\end{figure}
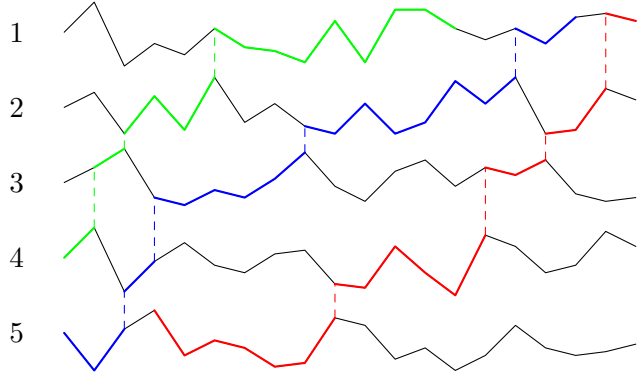

Via Greene's theorem \cite{greene1974extension}, we can describe many classical versions of RSK in terms of multi-path LPP across a continuous environment $f:[a, b] \to \R^n$. Essentially, the RSK map takes $f$ to a pair of environments $W^U f, W^R f$ which record multi-path last passage values from the bottom corner $(a, n)$ to the upper (U) boundary $[a, b] \X \{1\}$ and the right (R) boundary $\{b\} \X \{1, \dots, n\}$ of $[a, b] \X \{1, \dots, n\}$, respectively. More precisely, letting $T_n := \{(i, j) \in \{1, \dots, n\}^2 : i \le n - j + 1\}$, define two functions $W^Uf:[a, b] \to \R^n$ and $W^R f: T_n \to \R, (i, j) \mapsto (W^R f)_i(j)$ as follows:
\begin{alignat}{2}
	\label{E:WiUj}
\sum_{i=1}^k (W^Uf)_i(x) &= f[(a^k, n) \to (x^k, 1)], \qquad &&1 \le k \le n, \;\; x \in [a, b]. \\
\nonumber
	\sum_{i=1}^k (W^R f)_i(j) &= f[(a^k, n) \to (b^k, j)], \qquad &&(k, j) \in T_n.
\end{alignat}
The map $f \mapsto Wf := (W^U f, W^R f)$ has many remarkable properties. We highlight three that are particularly important for this paper:
\begin{description}
	\item[Invertibility:] The RSK map $W$ is invertible on its image with an \textit{explicit inverse}. Moreover, its image is easy to describe: a pair $(h, g)$ equals $Wf$ for some continuous $f$ if and only if $h:[a, b] \to \R^n$ is continuous function into the ordered set $\R^n_\le$ with $h(a) = 0$, $g$ satisfies the Gelfand-Tsetlin inequalities $g_i(j) \ge g_i(j+1) \ge g_{i+1}(j)$, and $h(b) = g_\cdot(1)$. The inequalities in this description of the image imply that $h, g$ can be viewed as path representations of Young tableaux, and the condition $h(b) = g_\cdot(1)$ says that these tableaux have the same shape. By restricting to certain classes of piecewise linear functions $f$ in our framework we can recover classical matrix versions of RSK (i.e. the Robinson-Schensted correspondence for permutations, usual RSK, dual RSK).
	\item[Measure Preservation:] The bijectivity of RSK implies that it maps nice measures to nice measures. For example, if $a, b \in \Z$ and $f$ is an environment of independent Bernoulli walks, then $W^U f$ consists of $n$ independent Bernoulli walks conditioned to stay ordered. Similarly, if $B:[a, b] \to \R^n$ is an environment of independent Brownian motions, then $W^U B$ consists of $n$ non-intersecting Brownian motions. There are similar interpretations for $W^R$.
	\item[Isometry:] The map $W^U$ preserves bottom-to-top last passage values. More precisely, we have the \emph{boundary isometry}
	\begin{equation}
		\label{E:isom-intro}
		f[(\bx, n) \to (\by, 1)] = W^U f[(\bx, n) \to (\by, 1)]
	\end{equation}
	for all $f, \bx, \by$. Similarly, $W^R$ preserves left-to-right passage values.
\end{description}
We refer the reader to \cite{dauvergne2022rsk} for precise versions of the statements above and a presentation of RSK in this framework, and \cite{fulton1997young, stanley1999enumerative} for more classical presentations.

In this paper, we will take three successive limits of RSK so that at every stage we have useful and interesting notions of invertibility, measure preservation, and isometry. Our final limit of the RSK correspondence will make explicit the function in Theorem \ref{T:RSK-intro}.

\subsection{The first limit: RSK for functions \texorpdfstring{$f:\R \to \R^n$}{f: R to Rn}}
\label{SS:RSK-1}
In Section \ref{S:sorting-monoids}, we define an RSK map on the domain
$$
\cC^n(\R) := \{f:\R \to \R^n: \la(f) := \lim_{x \to \pm \infty} \frac{f(x)}{x} \text{ exists}\}.
$$
This map is the limit of RSK on functions $f:[a, b] \X \{1, \dots, n\}$ where we take $a \to -\infty, b \to \infty$. We restrict to functions with an asymptotic slope for simplicity; the map could be defined on a larger space, but we will not need this here. The motivation for taking this limit is that when we let $[a, b]$ tend to $(-\infty, \infty)$, the RSK correspondence becomes simpler and picks up additional symmetries. In particular, the inverse map in this setting becomes a straightforward conjugation of the forward map.

 Since our right boundary $\{b\} \times \{1, \dots, n\}$ is tending to infinity, the information from $W^R f$ should disappear in this limit, and the correspondence should now produce just a single function $W f = (W f_1, \dots, Wf_n)$, which takes the place of $W^U f$. 
The corner $(a, n)$ that played the key anchoring role in \eqref{E:WiUj} also converges to $(-\infty, n)$ in this limit. For this reason, we need to define last passage from $-\infty$: for $\bm \in \{1, \dots, n\}^k_<$, let
\begin{equation}
	\label{E:finftyM}
f[(-\infty, \bm) \to (\bx, 1)] := \lim_{t \to -\infty} f[(t, \bm) \to (\bx, 1)] + \sum_{i=1}^k f_{m_i}(t).
\end{equation}
This limit does not necessarily exist, but if we ask that $f \in \cC^n_<(\R) := \{f \in \cC^n(\R) : \la(f) \in \R^n_<\}$ then it always will. With the above setup in place, we can now take the limit of the RSK formula \eqref{E:WiUj} to describe $Wf$ when $f \in \cC^n_<(\R)$:
\begin{equation}
\label{E:Wff}
\sum_{i=1}^k Wf_i(x) = f[(-\infty, (n, \dots, n - k + 1)) \to (x^k, 1)].
\end{equation}
With this definition, we have versions of invertibility, measure preservation, and isometry.
 
\begin{theorem}
\label{T:RSK-on-CnR-intro}
The map $W$ in \eqref{E:Wff} has the following properties:
\begin{enumerate}
	\item (Invertibility) $W$ maps $\cC^n_<(\R)$ bijectively onto $\cC^n_>(\R) := \{f \in \cC^n(\R) : \la(f) \in \R^n_>\}$. Moreover, the inverse map is given by $RWR:\cC^n_>(\R) \to \cC^n_<(\R)$, where $R:\cC^n(\R) \to \cC^n(\R)$ is the reflection map $Rf(x) = f(-x)$.
	\item (Measure preservation) Let $\mu_\la$ denote the law on $\cC^n(\R)$ of $n$ independent Brownian motions with drift vector $\la$ and variance $2$. Then if $\la \in \R^n_<$ and $B \sim \mu_\la$, then $WB - WB(0) \sim \mu_{\la^*}$ where $\la^* = (\la_n, \dots, \la_1)$.
	\item (Isometry) The isometry \eqref{E:isom-intro} holds with $Wf$ in place of $W^U f$.
\end{enumerate}
\end{theorem}

Parts of Theorem \ref{T:RSK-on-CnR-intro} are not completely new. Indeed, the isometry is essentially immediate from the isometry in the finite setting, and the fact that $RWR$ inverts $W$ is related to a similar description of the inverse RSK map on functions $f:[0, \infty) \to \R^n$ given in \cite{dauvergne2022rsk}.
The bijection $W$ turns out to also be the same as a certain queuing map described in Sorensen's Ph.D. thesis \cite{sorensen2023stationary}, see Section 2.3.3 and Lemma 2.3.18 therein. The description in terms of multi-path LPP and the connection with RSK is not discussed there.

When $n = 2$, the measure preservation property in Theorem \ref{T:RSK-on-CnR-intro} is well-known: this is one aspect of the so-called \textit{Brownian Burke theorem}, first observed in \cite{harrison1990quasireversibility} and applied to LPP in \cite{o2001brownian}. To move from measure preservation for $n=2$ to general $n$, we will build the $n$-line map by repeatedly applying the $2$-line map to pairs of adjacent lines. This approach to RSK was first developed in the closely related context of functions $f:[0, \infty) \to \R^n$ by Biane, Bougerol, and O'Connell \cite{biane2005littelmann}, building on work of \cite{o2001brownian, o2002representation, o2003path}. This approach was further studied in \cite{DOV, dauvergne2022rsk}. We expand on these ideas here, and along the way describe a whole family of useful and interesting maps on the space of functions $\cC^n(\R)$ which generalize the maps $W$ and $RWR$. This broader family will be crucial for understanding the next limit transition. Each of these maps has nice invertibility, measure-preservation, and isometry properties and are useful for studying different aspects and models of LPP.

To describe things more precisely, define the \textbf{Pitman transform} $\cP:\cC^2(\R) \to \cC^2(\R)$ by letting $\cP|_{\cC^2_<(\R)}$ be the $2$-line map $W$ described above, and setting $\cP|_{\cC^2_\ge(\R)} = \id$. Next, for $i = 1, \dots, n - 1$, define \textbf{sorting operators} $\tau_i, \bar \tau_i:\R^n \to \R^n$ by
$$
\tau_i(x) = \begin{cases}
x, \qquad &x_i \ge x_{i+1}, \\
(x_1, \dots, x_{i+1}, x_i, \dots, x_n), \qquad &x_i < x_{i+1}
\end{cases}
$$
and
$$
\bar \tau_i(x) = \begin{cases}
	x, \qquad &x_i \le x_{i+1}, \\
	(x_1, \dots, x_{i+1}, x_i, \dots, x_n), \qquad &x_i > x_{i+1}.
\end{cases}
$$
Equivalently, $\bar \tau_i = R \tau_i R$, where $R x = - x$. Let $\mathfrak D_n$ denote the monoid generated by the $\tau_i, \bar \tau_i, i = 1, \dots, n - 1$. We then define operators $\cP_{\tau_i}, \cP_{\bar \tau_i}:\cC^n(\R) \to \cC^n(\R)$ by setting 
$$
\cP_{\tau_i}(f_1, \dots, f_n) = (f_1, \dots, f_{i-1}, \cP(f_i, f_{i+1}), f_{i+2}, \dots, f_n),
$$
and letting $\cP_{\bar \tau_i} = R \cP_{\tau_i} R$. It turns out that these definitions extend to an action of the whole monoid $\mathfrak D_n$. For an arbitrary element $\sig = \pi_1 \cdots \pi_k \in \mathfrak D_n$ where $\pi_i \in \{\tau_i, \bar \tau_i, i = 1, \dots, n - 1\}$, we can define $\cP_\sig = \cP_{\pi_1} \cdots \cP_{\pi_k}$; this definition is independent of the choice of the decomposition $\pi_1 \cdots \pi_k$ of $\sig$, and we have the following theory for this monoid action, which can be viewed as a generalization of Theorem \ref{T:RSK-on-CnR-intro}.

\begin{theorem}
	\label{T:RSK-on-CnR-2-intro}
	The maps $\cP_\sig, \sig \in \cD_n$ have the following properties.
	\begin{enumerate}
		\item (Invertibility via last passage) Let $f \in \cC^n_\le(\R)$. Let $O(f)$ denote the orbit of $f$ under the action of $\cP_\sig, \sig \in \mathfrak D_n$ and let $O(\la(f))$ denote the orbit of $\la(f)$ under the action of $\mathfrak D_n$. Then the map $g \mapsto \la(g)$ is a bijection from $O(f)$ to $O(\la(f))$ and $\la \circ \cP_\sig = \sig \circ \la$ as functions from $O(f) \to O(\la(f))$. Moreover, if $g \in O(f)$, then for all $x \in \R, k \in \{1, \dots, n\}$ we have
		$$
		\sum_{i=1}^k g_i(x) = f[(-\infty, \bm^{k, g}) \to (x^k, 1)],
		$$
		where $\bm^{k, g} \in \{1, \dots, n\}^k_<$ is the minimal vector such that as multi-sets, 
		$$
		\{\la(g)_1, \dots, \la(g)_k\} = \{\la(f)_{\bm^{k, g}_1}, \dots, \la(f)_{\bm^{k, g}_k}\}.
		$$
		\item (Measure preservation) If $\la \in \R^n$ and $B \sim \mu_\la$, then $\cP_\sig B - \cP_\sig B(0) \sim \mu_{\sig(\la)}$.
		\item (Isometry) The isometry \eqref{E:isom-intro} holds with $\cP_\sig f$ in place of $W^U f$.
	\end{enumerate}
\end{theorem}

Theorem \ref{T:RSK-on-CnR-2-intro} can be interpreted as saying that given an environment $f$, we can build a family of boundary isometric environment $\cP_\sig f$ which permute the slopes of the lines of $f$ in different ways. We can access different information about last passage values across $f$ by studying different elements in this family. 
Theorem \ref{T:RSK-on-CnR-2-intro} is proven in the text as a combination of Corollary \ref{C:Pitman-iso}, Lemma \ref{L:orbits}, and Proposition \ref{P:orbits}. The special case of Theorem \ref{T:RSK-on-CnR-intro} is given in Corollary \ref{C:RSK}.

\subsection{The second limit: RSK for functions \texorpdfstring{$f \in \cC^\N(\R)$}{}}
\label{SS:RSK-2}

Next, we take a limit of RSK on $\cC^n(\R)$ as $n \to \infty$. This is the goal of Section \ref{S:Pitman-Brownian}. This limit naturally results in \textbf{Busemann functions} for Brownian LPP, foreshadowing the description of the maps in Theorem \ref{T:RSK-intro}. Let $B = (B_i, i \in \N) \in \cC^\N(\R)$ be an environment of independent two-sided Brownian motions with no drift. For $\theta \in [0, \infty)^k_\le$ and $\bx \in \R^k_\le$ define the multi-path Busemann function
\begin{equation}
	\label{E:Buse-multi-intro}
	\cB^\theta(\bx; B) = \lim_{t \to -\infty} B[(t, \fl{\theta |t|} +k) \to (\bx, 1)] - N_\theta(t)
\end{equation}
Here the notation $\fl{\theta t}$ should be understood coordinatewise: $\fl{\theta t} = (\fl{\theta_1 t}, \dots, \fl{\theta_k t})$. We use similar coordinatewise notation throughout the paper.
We have added $+k$ above to ensure that the Busemann function is well-defined even if $\theta = 0^k$. The function $N_\theta(t)$ is a normalization term needed in order to ensure that the limit \eqref{E:Buse-multi-intro} exists. There are a couple of different choices for this normalization term; we choose the normalization that retains the most information. Letting $\hat \theta \in [0, \infty)^m_<$ be uniquely chosen so that $\theta = (\hat \theta_1^{\ell_1}, \dots, \hat \theta_m^{\ell_m})$, set
$$
N_\theta(t) := \sum_{i=1}^{m} B[(t, \lfloor \hat \theta_i^{\ell_i} |t| \rfloor  +k) \to (0^{\ell_i}, 1)],
$$
where the operation $\hat \theta_i^{\ell_i} \mapsto \lfloor \hat \theta_i^{\ell_i} |t| \rfloor  +k$ should again be understood coordinatewise. Busemann functions are closely linked to the study of \textbf{semi-infinite optimizers}. A semi-infinite optimizer in direction $\theta$ ending at $(\bx, 1)$ is a $k$-tuple $\pi = (\pi_1, \dots, \pi_k)$ such that $\pi_i:(-\infty, x_i] \to \N$, $\pi(t)/|t| \to \theta$ as $t \to -\infty$ and such that $\pi|_{[y, x_k]}$ is an optimizer ending at $(\bx, 1)$ for any $y < x_1$. A semi-infinite optimizer with one path is called a \textbf{semi-infinite geodesic}.

 The single-path Busemann process and semi-infinite geodesics in Brownian LPP have been extensively studied in recent years. The state-of-the-art results are due to Sepp\"al\"ainen and Sorensen \cite{seppalainen2023global}, building on \cite{alberts2020busemann, seppalainen2023busemann}. They showed that there is a random countable set $\Theta \subset [0, \infty)$ such that for $\theta \notin \Theta$, the limit \eqref{E:Buse-multi-intro} exists for all $x \in \R$ and the process $(\theta, x) \mapsto \cB^\theta(x ; B)$ is continuous in both $x$ and $\theta$ on $\R \X [0, \infty) \smin \Theta$. Moreover, semi-infinite geodesics exist ending at all points and in all directions $\theta$, and for $\theta \notin \Theta$, all semi-infinite geodesics in direction $\theta$ eventually coalesce. We use these results to help build an analogous toolkit for working with multi-path Busemann functions for Brownian LPP, see Section \ref{S:Busemann-A} for details.
 
 Now let us return to the problem of defining a limit for the RSK correspondence on $\cC^n(\R)$ as $n \to \infty$. Consider the Brownian environment $B$ as above, and let $B' \in \cC^\N(\R)$ be an independent Brownian environment, where each $B'_i$ has drift $a > 0$. For each $n \in \N$, let $W^n = \cP_{\sig_n}(B_1, \dots, B_n, B_1', \dots, B_n')$, where $\sig_n$ is chosen so that the slopes are reversed, i.e. $\la(W^n) = (a^n, 0^n)$. By Theorem \ref{T:RSK-on-CnR-2-intro}.2, $W^n - W^n(0) \sim \mu_{(a^n, 0^n)}$. In particular, $W^n - W^n(0)$ converges in law as $n \to \infty$ to an environment $B^a \in \cC^\N(\R)$ of infinitely many independent Brownian motions with slope $a$. It turns out that this convergence also happens almost surely, and that $B^a$ is described purely in terms of Busemann functions for the unsloped environment $B$; the dependence on $B'$ disappears in the limit. We have now defined a collection of maps indexed by $a > 0$ that play the role of RSK on $\cC^\N(\R)$. It turns out they are invertible, measure-preserving, and naturally satisfy a \textit{Busemann isometry}.
 
 To make things more precise, define
 $
 \cW^\la(\bx; B) := \cB^{\la |\la|/2}(\bx; B),
 $
 which parametrizes Busemann functions by the asymptotic slope of the function $\bx \mapsto \cW^\la(\bx; B)$, rather than the Busemann direction. For $a \ge 0$ we can alternately define $B^a\in \cC^\N(\R)$ by the formula
 \begin{equation}
 	\label{E:Ba-def-1}
 	\sum_{i=1}^k B^a_i(x) = \cW^{a^k}(x^k ; B), \qquad k \in \N.
 \end{equation}
 We also build environments for $a < 0$. First, for an environment $f$, define $Rf(x) = f(-x)$, and for $\bx \in \R^k_\le$, let $-\bx = (-x_k, \dots, -x_1) \in \R^k_\le$. For $\theta \in (-\infty, 0]^k_\le, \bx \in \R^k_\le$, define
 \begin{equation}
 \label{e:reversed-LP}
 \cW^\theta(\bx ; B) = \cB^{\theta|\theta|/2}(\bx ; B) := \cB^{-\theta|\theta|/2}(-\bx ; RB), \qquad       
 \end{equation}
 See Figure \ref{fig:BrownianRSK} for an illustration of this definition when $\theta \in (-\infty, 0)$. Using this definition, we can understand formula \eqref{E:Ba-def-1} for arbitrary $a \in \R$. The full suite of RSK transformations $B \mapsto B^a, a \in \R$  exhibits a group structure, and acts on Busemann functions via a simple parameter translation. 
 \begin{theorem}
 	\label{T:busemann-shear-intro}
 	Let $B = (B_i, i \in \N)$ be an environment of independent two-sided Brownian motions with zero drift.
 	Then:
 	\begin{enumerate}
 		\item (Invertibility and group structure) For any $a, b \in \R$, almost surely:
 		\begin{equation*}
 			B^0 = B, \qquad (B^a)^{-a} = B, \qquad (B^a)^b = B^{a + b}.
 		\end{equation*}
 		\item (Measure preservation) For any $a \in \R$, $B^a$ is a sequence of independent Brownian motions of drift $a$. 
 		\item (Busemann isometry) Fix $a \in \R$ and $\la \in (-\infty, 0]^k_\le \cup [0, \infty)^k_\le$. Suppose also that $\la + a \in (-\infty, 0]^k_\le \cup [0, \infty)^k_\le$. Then a.s.\ for all $\bx \in \R^k_\le$ we have
 		\begin{equation*}
 			\cW^{\la + a}(\bx; B) = \cW^\la(\bx; B^a).
 		\end{equation*}
 	\end{enumerate}
 \end{theorem}
Theorem \ref{T:busemann-shear-intro}.3 also allows us to define the function $\cW^\la(\bx ; B)$ when $\la \in \R^k_\le$ but possibly has coordinates of both signs: we simply set $\cW^\la(\bx ; B) = \cW^{\la - \la_1}(\bx; B^{-\la_1})$. Theorem \ref{T:busemann-shear-intro} is proven in the text as Theorem \ref{T:Busemann-shear}.
 
 We can also take different limits of the RSK correspondence on $\cC^n(\R)$ that result in new representations for the multi-path Busemann process for $B$. More precisely, let $\theta \in [0, \infty)^k_\le$ and let $B' \sim \mu_\theta$. Let the transform $\cP_{\sig_n}, \sig_n \in \mathfrak{D}_{n+k}$ be chosen so that $W^n := \cP_{\sig_n}(B_1, \dots, B_n, B')$ satisfies $\la(W^n) = (\theta, 0^n)$. Again, the maps $W^n - W^n(0)$ have an almost sure limit $W$ as $n \to \infty$ which comes from the Busemann process for $B$. The next theorem describes this limit. 
 \begin{theorem}
 	\label{T:Busemann-law-intro}
 	Let $\theta \in [0, \infty)^k_\le$, and define $W = (W_1, \dots, W_k):\R\to \R^k$ by the formula:
 	$$
 	\sum_{i=1}^\ell W_i(x) = \cW^{(\theta_1, \dots, \theta_\ell)}(x^\ell ; B), \qquad 1 \le \ell \le k.
 	$$
 	Then $W - W(0) \sim \mu_{\theta}$. Moreover, consider any  $\bn \in \{1, \dots, k\}^\ell_<$ for some $\ell \le k$ and suppose that $\bn \le \bm$ for any $\bm \in \{1, \dots, k\}^\ell_<$ with $\theta_{n_i} = \theta_{m_i}$ for all $i$. Then almost surely, for all $\bx \in \R^\ell_\le$ we have
 	\begin{equation}
 		\label{E:N-x}
 		\cB^{(\theta_{n_1}, \dots, \theta_{n_\ell})}(\bx; B) = W[(-\infty, \bn) \to (\bx, 1)].
 	\end{equation}
 \end{theorem}
The identity \eqref{E:N-x} may look familiar to the reader in the special case when $\bn$ is a singleton. Indeed, as part of their comprehensive study of Brownian Busemann functions, Sepp\"al\"ainen and Sorensen \cite{seppalainen2023global} showed that given $\theta \in [0, \infty)^k_<$, there exists an environment $\tilde W \sim \mu_\theta$ such that the following identity holds in distribution, as functions of $(x, i) \in \R \X \{1, \dots, k\}$.
\begin{equation}
	\label{E:W-B-intro}
	\tilde W[(-\infty, i) \to (x, 1)] - \tilde W[(-\infty, i) \to (0, 1)] \eqd \cW^{\theta_i}(x; B)
\end{equation}
One upshot of Theorem \ref{T:Busemann-law-intro} is that we realize this as an \emph{almost sure identity}, naturally constructing $\tilde W$ in terms of $B$.
A multi-path almost sure approach to the identity \eqref{E:W-B-intro} was previously suggested in \cite[Remark 1.3.4]{dauvergne2023last} and developed there in the context of the Airy sheet. Theorem \ref{T:Busemann-law-intro} is proven in the text as Theorem \ref{T:brownian-law}.
\begin{figure}[t]
	\centering
	\includegraphics[width=10cm]{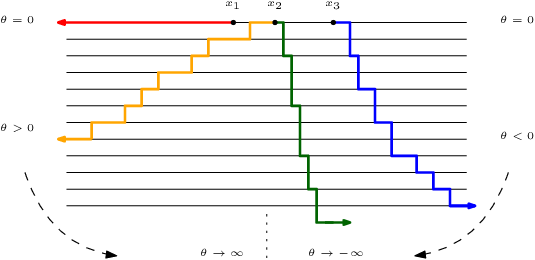}
	\caption{Semi-infinite geodesics in $B$, indexed by the direction $\theta \in \R$. To make sense of the direction when $\theta < 0$, we look along down-right paths, or equivalently, paths in the reversed environment. }
	\label{fig:BrownianRSK}
\end{figure}

\begin{figure}[ht]
	\centering
	\includegraphics[width=10cm]{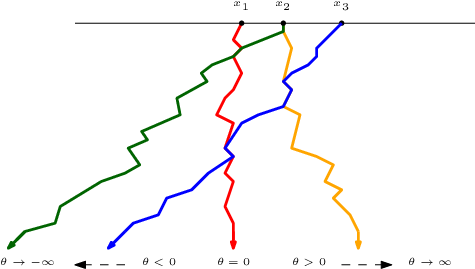}
	\caption{Semi-infinite geodesics in $\cL$. The colouring of paths in Figures \ref{fig:BrownianRSK} and \ref{fig:LSRK} is chosen so that the start points and directions match in the two pictures under the version of the RSK correspondence in Theorem \ref{T:landscape-recovery-intro}.}
	\label{fig:LSRK}
\end{figure}

\subsection{The final limit: RSK for the directed landscape}
\label{SS:RSK-3}

At this point, we have constructed analogues of the RSK correspondence on an infinite line environment on $\R\X\N$ indexed by $a \in \R$. We can make one final limiting transition. This transition can be viewed either by taking $a \to \pm \infty$ in Theorem \ref{T:busemann-shear-intro}, or equivalently by decreasing the spacing between lines to $0$, in a manner that takes Brownian LPP to the directed landscape. This is the goal of Section \ref{S:reconstruction}. 

Recall the main theorem of \cite{DOV}, cited above as Theorem \ref{T:DL-cvg-intro}, which defines the directed landscape as the scaling limit of Brownian LPP. In Theorem \ref{T:DL-cvg-intro}, the convergence in distribution is in the compact topology on $\cC(\Rd)$.
Just as with LPP, the directed landscape is best thought of as assigning distances to pairs of points $p = (x, s), q = (y, t)$. The value $\cL(p; q)$ is best thought of as a distance between two points $p$ and $q$ in the space-time plane, and indeed, it satisfies a reverse triangle inequality
\begin{equation}
	\label{E:landscape-triangle}
	\cL(x, s; y, t) \ge \cL(x, s; z, r) + \cL(z, r; y, t),
\end{equation}
just like LPP. Unlike a usual metric, $\cL$ is not symmetric, does not assign distances to every pair of points in the plane, and may take negative values. 

We define the analogue of RSK for $\cL$ by modifying the Busemann definition \eqref{E:Ba-def-1} for the limiting context. To set things up properly we need to define path lengths, geodesics, and optimizers in the directed landscape.
For a continuous function $\pi:[s, t]\mapsto \R$,  referred to as a \textbf{path}, let $\bar \pi(r) = (\pi(r), r)$. Following \cite[Section 12]{DOV}, define the \textbf{length} of $\pi$ by
\begin{equation}
	\label{E:length}
	\|\pi\|_\cL=\inf_{k\in \N}\inf_{s=t_0<t_1<\ldots<t_k=t}\sum_{i=1}^k\cL(\bar \pi(t_{i-1});\bar \pi(t_i))\,.
\end{equation}
Next, for $s < t$ and $\bx, \by \in \R^k_\le$, define the \textbf{extended landscape value}
\begin{equation}
	\label{E:extended-land-value}
	\cL(\bx, s; \by, t) = \max_{\pi = (\pi_1, \dots, \pi_k)} \sum_{i=1}^k \|\pi_i\|_\cL,
\end{equation}
where the maximum is over all \textbf{disjoint $k$-tuples} of paths $\pi = (\pi_1, \dots, \pi_k)$ with $\pi(r) \in \R^k_<$ for $r \in (s, t)$, and such that $\pi(s) = \bx, \pi(t) = \by$. In \cite{dauvergne2021disjoint}, the authors show that the maximum in \eqref{E:extended-land-value} is almost surely achieved for all $s < t$ and $\bx, \by \in \R^k_\le$. As before, we call a maximizer a \textbf{disjoint optimizer}, or \textbf{geodesic} if $k = 1$. Moreover, the extended landscape is continuous, and is the limit of multi-path Brownian LPP. Using this structure, we can define (multi-path) Busemann functions in $\cL$ as in \eqref{E:Buse-multi-intro}. For $\theta, \bx \in \R^k_\le$, let
\begin{equation}
	\label{E:cL-busemann-intro}
	\fB^\theta(\bx; \cL) := \lim_{t \to -\infty} \cL(\theta |t|, t; \bx, 0) - \sum_{i=1}^{m} \cL( \hat \theta^{\ell_i} |t|, t; 0^{\ell_i}, 0).
\end{equation}
Putting aside the question of Busemann existence for one moment, by analogy with \eqref{E:Ba-def-1}, for every $a \in \R$ we can define an RSK map for $\cL$. For $a \in \R$, define $B^{\cL,a}(\cdot) \in \cC^\N(\R)$ by:
\begin{equation}
\label{E:B-a-L}
\sum_{i=1}^k B^{\cL,a}_i(x) = \fB^{(a/2)^k}(x; \cL).
\end{equation}
We can now ask for the usual properties of these maps: invertibility, measure-preservation, isometry. The next theorem gives analogues of these properties here. 

\begin{theorem}
	\label{T:landscape-recovery-intro}
Let $\cL$ denote the directed landscape, and let $\mathbb H^2_\uparrow = \{(x, s; y, t) \in \R^4 : s < t \le 0\}$. For every $a \in \R$ define an environment $B^{\cL,a} \in \cC^\N(\R)$ as in \eqref{E:B-a-L} from $\cL|_{\mathbb H^2_\uparrow}$.
Then:
\begin{enumerate}
	\item (Measure preservation) For any $a \in \R$, $B^{\cL,a}$ is a sequence of independent Brownian motions of drift $a$. Moreover, the joint law of $B^{\cL,a}, a \in C$ is the same as the joint law of the environments $B^a, a \in C$ in Theorem \ref{T:busemann-shear-intro} for any countable set $C \subset \R$ (we restrict to countable $C$ to avoid topological issues). In particular, almost surely $B^{\cL,a + b} = (B^{\cL,a})^b$ for any fixed $a, b\in \R$. 
	\item (Busemann isometry) For any $a \in \R$ and $\la \in \R^k_\le$, a.s.\ for all $\bx \in \R^k_\le$ we have
	$$
	\cW^{\la + a}(\bx; B^{\cL,a}) = \fB^{(\la + a)/2}(\bx; \cL).
	$$
	\item (Invertibility) If we let the environment $\cL^a|_{\mathbb H^2_\uparrow}$ be defined as in Theorem \ref{T:DL-cvg-intro} where $B^a = B^{\cL,a}$, then
	$$
	\cL^a|_{\mathbb H^2_\uparrow} \cvgp \cL|_{\mathbb H^2_\uparrow}
	$$
	as $a \to -\infty$, where the convergence in probability is in the compact topology on functions from $\mathbb H^2_\uparrow \to \R$. Since any two of the environments $B^{\cL, a}, B^{\cL, b}$ are a.s.\ measurable functions of each other by part $1$, this implies that for every $a_0 \in \R$ there is a measurable map $f:\cC^\N(\R) \to \cC(\mathbb H^2_\uparrow)$ such that $\cL|_{\mathbb H^2_\uparrow} = f(B^{\cL,a_0})$ a.s.
\end{enumerate}
\end{theorem}
Theorem \ref{T:landscape-recovery-intro} contains Theorem \ref{T:RSK-intro}, and is proven in the text as part of the stronger Theorem \ref{T:landscape-recovery}.
To understand where Theorem \ref{T:landscape-recovery-intro} comes from, we should examine what happens to the Busemann functions for $B^a$ under the scaling in Theorem \ref{T:DL-cvg-intro}. If we momentarily assume that we can exchange the order of the $t \to -\infty$ limit in \eqref{E:Buse-multi-intro} and the $a \to -\infty$ limit in Theorem \ref{T:DL-cvg-intro}, we can use the coupling of the environments $B^a$ in Theorem \ref{T:busemann-shear-intro} to see that:
\begin{equation}
	\label{E:land-cpu}
	\begin{split}
	\fB^\theta(\bx; \cL) &\eqd \lim_{a \to -\infty} \cB^{a^3/(2a + 8 \theta)}(\bx; B^a) = \lim_{a \to -\infty} \cW^{|a| \sqrt{1/(1 + 4\theta/a)}}(\bx; B^a) \\
	&= \lim_{a \to -\infty} \cW^{|a| \sqrt{1/(1 + 4\theta/a)} + a}(\bx; B) = \cW^{2 \theta}(\bx ; B).
	\end{split}
\end{equation}
Putting this together gives a candidate for the law of the extended Busemann process $\fB^\theta(\bx; \cL)$, which is verified by Theorem \ref{T:landscape-recovery-intro}.1,2. Comparing Figures \ref{fig:BrownianRSK} and \ref{fig:LSRK} illustrates this part of the theorem. The law we have identified is consistent with known results about the single-path Busemann process in the directed landscape \cite{rahman2021infinite, busani2021diffusive, busani2024stationary}. In \cite{busani2021diffusive}, the process $\fB^\theta(x, \cL)$ for $\theta, x \in \R$ was termed the stationary horizon, so naturally, we call the extended process $(\theta, \bx) \mapsto \fB^\theta(\bx, \cL)$ for arbitrary $(\theta, \bx) \in \bigcup_{k=1}^\infty \R^k_\le \X \R^k_\le$ the \textbf{extended stationary horizon}.

Now, given Theorem \ref{T:landscape-recovery-intro}.1,2, the coupling of the environments $\cL^a$ through the environments $B^{\cL, a}$ gives a sequence of maps converging to $\cL$ in law, and whose multi-path Busemann functions converge in probability. One may optimistically expect that in such a coupling, we also have convergence in probability of the directed landscape on the whole lower half-plane. Indeed, this is the case, and is verified by Theorem \ref{T:landscape-recovery-intro}.3, using the truss-slit framework described in Section \ref{SS:strategy}. See Sections \ref{SS:slit}, \ref{SS:multi-path-to-restricted}, \ref{SS:optimizers-rays} for details.

% Theorem \ref{T:landscape-recovery-intro} also allows us to prove a version of the same theorem restricted to a strip, resolving a conjecture from \cite[Conjecture 1.10]{dauvergne2021disjoint}. For this theorem and throughout the paper, for a topological space $X$ we let $\cC(X)$ denote the space of continuous functions $f:X \to \R$ with the compact topology.

% \begin{theorem}
% 	\label{T:Airy-RSK-intro}
% 	Let $\cL$ denote the directed landscape restricted to the strip 
% 	$
% 	\mathbb S^2_\uparrow = \{(x, s; y, t) \in ([0, 1] \X \R)^2: s < t\},
% 	$
% 	and define $\fA \in \cC(\R \X \N), (i, x) \mapsto \fA_i(x)$ by the formula
% 	$$
% 	\sum_{i=1}^k \fA_i(x) = \cL(0^k, 0; x^k, 1).
% 	$$
% 	Then $\fA \in \cC(\R \X \N)$ is a parabolic Airy line ensemble and there is a measurable function $f:\cC(\R \X \N) \to \cC(\mathbb S^2_\uparrow)$ with $f(\fA) = \cL$ almost surely.
% \end{theorem}

% Unlike in Theorem \ref{T:landscape-recovery-intro}, we do not have an explicit description of the above function $f$. However, its existence is almost immediate given Theorem \ref{T:landscape-recovery-intro} and a multi-path RSK isometry proven in \cite{dauvergne2021disjoint}, see Corollary \ref{C:strip-reconstruction} for details.

\section{Preliminaries}
\label{S:prelim}
In this section, we gather a few basic results about LPP and the directed landscape. For basic definitions, we refer the reader back to the introduction.

\subsection{Basics of last passage percolation}
\label{SS:basics-lpp}
Recall the definition of LPP, multi-path LPP, geodesics, and semi-infinite optimizers in Section \ref{SS:RSK-1}. We use the following basic facts about these objects throughout the paper. We state all results here for $k$-path optimizers and multi-path last passage values. When $k = 1$, these results specialize to the simpler setting of geodesics and last passage values. Throughout Section \ref{SS:basics-lpp}, we fix $f \in \cC^\Z(\R)$. All results will also apply for $f \in \cC^I(\R)$ for integer intervals $I$, when the relevant objects are well-defined. For $\bp = (\bx, \bn), \bq = (\by, \bm) \in \R^k_\le \times \Z^k_\le$, we call $(\bp, \bq)$ an \textbf{endpoint pair} if there is at least one disjoint $k$-tuple from $\bp$ to $\bq$.

\begin{lemma}[Optimizer existence: Lemma 2.2, \cite{dauvergne2021disjoint}]
\label{L:optimizer-existence-monotoncity}
Let $(\bp, \bq)$ be an endpoint pair. Then there exists an optimizer $\pi^L$ from $\bp$ to $\bq$ such that $\pi^L \le \pi$ for any optimizer $\pi$ from $\bp$ to $\bq$. We call $\pi^L$ the \textbf{leftmost} optimizer. Similarly, there exists a \textbf{rightmost} optimizer $\pi^R$ from $\bp$ to $\bq$ such that $\pi^R \ge \pi$ for any optimizer $\pi$ from $\bp$ to $\bq$.
\end{lemma}

\begin{lemma}[Optimizer monotonicity: Lemma 2.3, \cite{dauvergne2021disjoint}]
	\label{L:mono-tree-multi-path}
	Let $(\bp, \bq)$ and $(\bp', \bq')$ be two endpoint pairs of sizes $k$. Suppose that:
	\begin{enumerate}
		\item $\bp = (\bx, n)$ and $\bp' = (\bx', n)$ for some $\bx \le \bx' \in \R^k_\le$, or else $\bp = (x, \bn)$ and $\bp' = (x, \bn')$ for some $\bn \le \bn' \in \Z^k_\le$.
		\item $\bq = (\by, n)$ and $\bq' = (\by', n)$ for some $\by \le \by' \in \R^k_\le$, or else $\bq = (y, \bm)$ and $\bq' = (y, \bm')$ for some $\bm \le \bm' \in \Z^k_\le$.
	\end{enumerate} 
	Then letting $\pi$ be the rightmost optimizer from $\bp$ to $\bq$, and $\pi'$ be the rightmost optimizer from $\bp'$ to $\bq'$ we have that $\pi \le \pi'$.
\end{lemma}
Technically, \cite[Lemma 2.2]{dauvergne2021disjoint} only applies when $\bp = (\bx, n), \bp' = (\bx', n)$ and $\bq = (\by, n), \bq' = (\by', n)$. However, the proof goes through verbatim is the slightly more general setting above.

\begin{lemma}[Metric composition law]
	\label{L:split-path}
	Let $(\bp, \bq) = (\bx, \bn; \by, \bm)$ be an endpoint pair of size $k$ and let $\ell \in \{m_k + 1, \dots, n_1\}$. Then
	$$
	f[\bp \to \bq] = \max_\bz f[\bp \to (\bz, \ell)] +  f[(\bz, \ell - 1) \to \bq],
	$$
	where the maximum is taken over $\bz \in \R^k_\le$ such that both $(\bp; \bz, \ell)$ and $(\bz, \ell - 1; \bq)$ are endpoint pairs. Similarly, if $z \in [x_k, y_1]$, then
	$$
	f[\bp \to \bq] = \max_{\mathbf{l} \in \Z^k_<} f[\bp \to (z, \mathbf{l})] +  f[(z, \mathbf{l}) \to \bq],
	$$
	where the maximum is taken over $\mathbf{l} \in \Z^k_<$ such that both $(\bp; z, \mathbf{l})$ and $(z, \mathbf{l}; \bq)$ are endpoint pairs.
\end{lemma}

\begin{lemma}[Quadrangle inequality, Lemma 2.4, \cite{dauvergne2021disjoint}]
	\label{L:quadrangle}
	Let $(\bp, \bq) = (\bx, \bn; \by, \bm)$, $(\bp', \bq') = (\bx', \bn'; \by', \bm')$ be endpoint pairs satisfying the conditions of Lemma \ref{L:mono-tree-multi-path}. Suppose that $(\bp, \bq')$ and $(\bp', \bq)$ are also endpoint pairs. Then:
	$$
	f[\bp \to \bq'] + f[\bp' \to \bq] \le f[\bp \to \bq] + f[\bp' \to \bq']
	$$
	with equality if there exists a point $z \in [x_k', y_1]$ and optimizers $\pi$ from $\bp$ to $\bq$, $\pi'$ from $\bp'$ to $\bq'$ such that $\pi(z) = \pi'(z)$.
\end{lemma}

The `equality if' claim in Lemma \ref{L:quadrangle} is not contained in \cite{dauvergne2021disjoint}, but it is easy to see. In this case, the paths $\pi|_{(-\infty, z)} \oplus \pi'|_{[z, \infty)}$ and $\pi'|_{(-\infty, z)} \oplus \pi|_{[z, \infty)}$ (here $\oplus$ denotes paths concatenation), are disjoint $k$-tuples from $\bp$ to $\bq'$ and $\bp'$ to $\bq$ whose lengths sum to $f[\bp \to \bq] + f[\bp' \to \bq']$, proving the reverse inequality in Lemma \ref{L:quadrangle}. We end with a straightforward statement regarding last passage across lines and common shifts of the environment, whose proof we leave to the reader. 

\begin{lemma}[Last passage commutes with shifts]
	\label{L:shift-commute}
	Let $g:\R \to \R$ be any continuous function, and let $f + g \in \cC^\Z(\R)$ be given by $(f_i + g, i \in \Z)$. Then for any endpoint pairs $\bp = (\bx, \bn), \bq = (\by, \bm) \in \R^k_\le \X \Z^k_\le$, we have:
	$$
	(f+g)[\bp \to \bq] = f[\bp \to \bq] + \sum_{i=1}^k [g(y_i) - g(x_i)].
	$$
\end{lemma}
We will typically use Lemma \ref{L:shift-commute} when $g$ is a constant or linear function. We will typically refer to these five results without reference, as they should be viewed as part of the basic language for working with geodesics and optimizers.

\subsection{Shape and continuity bounds for Brownian LPP}

We will need two complementary limit shape theorems for Brownian LPP. For the first proposition, for $n \in \N$ and $x, \al, b, w > 0$ we define
\[
\cN_{b,w}(n,x,a)=
\sqrt{8 nx} + \sqrt{x} n^{-1/6}(a + b\log^{2/3}(n^{1/3} |\log(x/w)| + 1)).
\]
Note that for any $\alpha>0$, we have
\begin{equation} \label{eq:rescale-cN}
	\cN_{b,\alpha w}(n,\alpha x,a) = \sqrt{\alpha}\cN_{b,w}(n,x,a).
\end{equation}
\begin{proposition}[Proposition 4.3, \cite{dauvergne2021bulk}]
	\label{P:cross-prob}
	There exist positive constants $b, c$ and $d$ such that for all $w,a > 0$ and $n \ge 1$, the probability that
	\[
	B[(-x, n) \to (0, 1)] \le \cN_{b,w}(n,x,a), \;\; \forall x \in (0, \infty)
	\]
	is greater than or equal to $1- c e^{-d a^{3/2}}$.
\end{proposition}

The bound in Proposition \ref{P:cross-prob} is chosen to minimize the error term at $x = w$. Note that in \cite{dauvergne2021bulk}, Proposition \ref{P:cross-prob} is stated when $w = 1$. The general case follows by Brownian scaling. We will typically use Proposition \ref{P:cross-prob} when $w = n/\theta$ for some $\theta > 0$, in which case the mean term $\sqrt{8 n x} \sim n \sqrt{8/\theta}$ when $x \sim w$  and the error term is $\sqrt{x n^{-1/6}} = O(n^{1/3})$, as is expected for KPZ models. Note that Proposition \ref{P:cross-prob} also provides a uniform upper bound on multi-path last passage values, by virtue of the bound
$$
B[(\bx, n) \to (\by, 1)] \le \sum_{i=1}^k B[(x_i, n) \to (y_i, 1)].
$$
A corresponding lower bound also holds. We only state this at the level of the one-point bound. 
\begin{proposition}[Lemma A.4, \cite{dauvergne2021disjoint}]
	\label{P:top-bd}
	For every $k \in \N$, there exist positive constants $c_k, d_k, k \in \N$ such that the following holds. For all $m > 0$ and all endpoint pairs $(\bx, n), (\by, 1)$ with $\bx, \by \in \R^k_\le$ we have
	\[
	\mathbb P\left (|B[(\bx, n) \to (\by, 1)] - \sum_{i=1}^k \sqrt{8 n (x_i - y_i)}| \ge m \sqrt{x} n^{-1/6} \right ) \le c_k e^{-d_k m^{3/2}}.
	\]
\end{proposition}
Note that in the statement of Lemma A.4, \cite{dauvergne2021bulk}, there are restrictions on $\bx, \by$. However, these restrictions are not used in the proof of that lemma. A typical application of Propositions \ref{P:cross-prob} and \ref{P:top-bd} will be to bound the location of the argmax in the metric composition law from Brownian LPP. This is why we need a uniform upper bound but only a pointwise lower bound.

\subsection{The directed landscape}

In this final preliminary section, we collect basic results about the directed landscape. We start with the axiomatic description of $\cL$ in terms of the marginal $\cS(x,y)=\cL(x,0;y,1)$, known as the \textbf{Airy sheet}. We have the following uniqueness theorem, see Definition 10.1 and Theorem 10.9 of \cite{DOV}.

\begin{theorem}
	\label{T:L-unique}	
	The directed landscape $\cL:\Rd \to \R$ is the unique random continuous function satisfying:
	\begin{enumerate}
		\item (Airy sheet marginals) For any $t\in \R$ and $s>0$ we have 
		$$
		\mathcal{\cL}(x, t; y,t+s^3) \eqd s \cS(x/s^2, y/s^2) 
		$$
		jointly in all $x, y$. That is, the increment over time interval $[t,t+s^3)$ is an \textbf{Airy sheet of scale $s$}.
		\item (Independent increments) For disjoint time intervals $\{[t_i, s_i] : i \in \{1, \dots k\}\}$, the random functions
		$
		\{\cL(\cdot, t_i ; \cdot, s_i) : i \in \{1, \dots, k \}\}
		$
		are independent.
		\item (Metric composition law) Almost surely, for any $r<s<t$ and $x, y \in \R$ we have that
		$$
		\cL(x,r;y,t)=\max_{z \in \mathbb R} [\cL(x,r;z,s)+\cL(z,s;y,t)].
		$$
	\end{enumerate}
\end{theorem} 

Note that Definition 10.1 in \cite{DOV} states the independent increment property for disjoint open intervals, rather than closed intervals. The two are equivalent by continuity. The directed landscape has invariance properties which we use throughout the paper. 

\begin{lemma} [Lemma 10.2, \cite{DOV}]
	\label{L:invariance} We have the following equalities in distribution as random functions in $\cC(\R^4_\uparrow)$. Here $r, c \in \R$, and $q > 0$.
	\begin{enumerate}
		\item (Time stationarity)
		$$
		\displaystyle
		\cL(x, t ; y, t + s) \eqd \cL(x, t + r ; y, t + s + r).
		$$
		\item (Spatial stationarity)
		$$
		\cL(x, t ; y, t + s) \eqd \cL(x + c, t; y + c, t + s).
		$$
		\item (Flip symmetry)
		$$
		\cL(x, t ; y, t + s) \eqd \cL(-y, -s-t; -x, -t).
		$$
		\item (Shear stationarity)
		$$
		\cL(x, t ; y, t + s) \eqd \cL(x + ct, t; y + ct + sc, t + s) + s^{-1}[(x - y)^2 - (x - y - sc)^2].
		$$
		\item ($1:2:3$ rescaling)
		$$
		\cL(x, t ; y, t + s) \eqd  q \cL(q^{-2} x, q^{-3}t; q^{-2} y, q^{-3}(t + s)).
		$$
	\end{enumerate}
\end{lemma}

As discussed in the introduction (see \eqref{E:length} and surrounding discussion), we can define path length and geodesics in the directed landscape. We record one strong convergence lemma for geodesics that will be used in Section \ref{SS:extended-multi-horizon}.
\begin{lemma}
	\label{L:overlap-cvg}
	For two $\cL$-geodesics $\pi:[s, t] \to \R, \ga:[s', t'] \to \R$, define the overlap $O(\pi, \ga)$ to be the closure of the set $\{r \in (s, t) \cap (s', t') : \pi(r) =\ga(r)\}$. Also let $\Ga(S)$ denote the set of geodesics with endpoints in a set $S \sset \Rd$.
	Then the following claims hold almost surely:
	\begin{enumerate}
		\item (Theorem 1, \cite{bhatia2023duality} or Lemma 3.3.2, \cite{dauvergne202327})  For any $\cL$-geodesics $\pi, \ga$, $O(\pi, \ga)$ is always a closed interval.
		\item (Part of Proposition 3.5, \cite{dauvergne202327}) For two geodesics, $\ga:I \to \R, \pi:I' \to \R$, define 
		$$
		d_o(\ga, \pi) = \la(I) + \la(I') - 2 \la (O(\pi, \ga)),
		$$
		For any compact set $K \sset \Rd$, this is a metric on $\Ga(K)$ that turns $\Ga(K)$ into a compact Polish space.
	\end{enumerate}
\end{lemma}

Lemma \ref{L:overlap-cvg}.2 ensures that any pointwise limit of a sequence of geodesics is also a geodesic, and moreover, is also a limit in the stronger sense of overlap. The overlap structure of landscape geodesics allows us to treat them more like their discrete counterparts in LPP.

We move on to bounds on the extended directed landscape, see \eqref{E:extended-land-value} for the definition. We start with a version of Theorem \ref{T:DL-cvg-intro} for the full extended landscape. For this theorem, recall the scaling $(\bx,s)_a =(\bx - as/4, \fl{s a^3/8})$. In the introduction, we used this notation only for singletons $\bx$. Here we extend its use to vectors in $\R^k_\le$. We also let $\mathfrak X_\uparrow$ be the space of all points $(\bx, s; \by, t)$, where $s < t \in \R$ and $\bx, \by$ lie in the same space $\R^k_\le$ for some $k \in \N$.

\begin{theorem}[Theorem 1.5/1.6, \cite{dauvergne2021disjoint}]
	\label{T:DL-cvg-extended}
	For every $a \in \R$, let $B^a = (B^a_i:\R \to \R, i \in \Z)$ be an environment of independent two-sided Brownian motions of common drift $a$, and for a point $(\bx, s; \by, t) \in \mathfrak X_\uparrow$ with $\bx, \by \in \R^k_\le$, define
	$$
	\cL^a(\bx, s; \by, t) = B^a[(\bx,s)_a \to (\by,t)_a] - \frac{ka^2}{4}(t-s),
	$$
	when the right-hand side exists, and simply set $\cL^a(\bx, s; \by, t) = -\infty$ otherwise. 
	Then as $a \to - \infty$,
	$$
	\cL^a(\bx, s; \by, t) \cvgd \cL(\bx, s; \by, t),
	$$
	where the underlying topology is compact convergence of functions on $\mathfrak{X}_\uparrow$.
\end{theorem}

Theorem \ref{T:DL-cvg-extended} is stated with drift-free Brownian motions in \cite{dauvergne2021disjoint}. The translation between the two theorems is immediate since last passage commutes with common shifts of the environment (Lemma \ref{L:shift-commute}).

 Analogues of Lemmas \ref{L:optimizer-existence-monotoncity}, \ref{L:mono-tree-multi-path}, \ref{L:split-path}, and \ref{L:quadrangle} hold in the directed landscape, and we record here the results we need in that context.

\begin{lemma}[Optimizer existence in $\cL$: Theorem 1.7, \cite{dauvergne2021disjoint}]
	\label{L:optimizer-existence-monotoncity-landscape}
Almost surely, for every choice of $(\bx, s, \by, t)$
there is a disjoint $k$-tuple $\pi = (\pi_1, \dots, \pi_k)$ with 
$$
\sum_{i=1}^k \|\pi\|_\cL = \cL(\bx, s; \by, t).
$$
That is, the maximum \eqref{E:extended-land-value} is attained for every $\bx, s, \by, t$. For any fixed $(\bx, s, \by, t)$, almost surely this \textbf{disjoint optimizer} is unique.
\end{lemma}

Unlike in the semi-discrete setting, Lemma \ref{L:optimizer-existence-monotoncity} is actually quite a difficult result. Its proof takes up a large part of the paper \cite{dauvergne2021disjoint}. In the $k = 1$ case of geodesics, the result is easier, and was proven in \cite[Theorem 1.7, Lemma 13.2]{DOV}, where the existence of leftmost and rightmost geodesics is also shown. Because Lemma \ref{L:optimizer-existence-monotoncity-landscape} does not give the existence of leftmost and rightmost optimizers, the analogue of optimizer monotonicity for $\cL$ is slightly more subtle.

\begin{lemma}[Optimizer monotonicity in $\cL$: see Lemma 7.7, \cite{dauvergne2021disjoint}]
	\label{L:mono-tree-multi-path-L}
	Almost surely the following holds in $\cL$.
	Let $\bx \le \bx' \in \R^k_\le$, $\by \le \by' \in \R^k_\le$, $s < t$, and let $\pi, \pi'$ be optimizers in $\cL$ from $(\bx, s)$ to $(\by, t)$ and $(\bx', s)$ to $(\by', t)$, respectively. Then if either $\pi$ or $\pi'$ is the \textbf{unique} optimizer between its endpoints, we have $\pi \le \pi'$.
\end{lemma}

Observe that in the setting of Lemma \ref{L:mono-tree-multi-path}, we can also conclude that $\pi \le \pi'$ if there exists $\bx'', \by''$ with $\bx \le \bx'' \le \bx'$ and $\by \le \by'' \le \by'$ such that there is a unique optimizer from $(\bx'', s)$ to $(\by'', t)$. Such $\bx'', \by''$ will exist almost surely for all $\bx, \by, \bx', \by'$ for which $x_i < x_i', y_i < y_i'$ for all $i$.

\begin{lemma}[Metric composition law, Proposition 6.9, \cite{dauvergne2021disjoint}]
	\label{L:split-path-L}
	\label{P:mc-everywhere}
Almost surely, for every $r < s < t$ and $\bx, \by \in \R^k_\le$ for some $k \in \N$ we have
$$
\cL(\bx, r; \by, t) = \max_{z \in \R^k_<} \cL(\bx, r; \bz, s) + 	\cL(\bz, s; \by, t).
$$
\end{lemma}

\begin{lemma}[Quadrangle inequality, Lemma 5.7, \cite{dauvergne2021disjoint}]
	\label{L:quadrangle-landscape}
	Almost surely, for all $\bx \le \bx'$ and $\by \le \by'$ in $\R^k_\le$ and $s < t$, we have
$$
\cL(\bx, s; \by', t) + \cL(\bx', s; \by, t)  \le \cL(\bx, s; \by, t)  + \cL(\bx', s; \by', t).
$$
\end{lemma}

We end with a shape bound on the extended directed landscape.

\begin{lemma}[Lemma 6.7, \cite{dauvergne2021disjoint}]  \label{L:shape-land}
	For any $\eta>0$ and $k\in\N$, there is a random constant $R>1$, such that for any $\bx,\by \in\R^k_\le$ and $s<t$, we have
	$$
	\lf|\cL(\bx,s;\by,t) + \sum_{i=1}^k \frac{(x_i - y_i)^2}{t-s}\rg| < RG(\bx,\by,s,t)^{\eta} (t-s)^{1/3}
	$$
	where
	$$
	G(\bx,\by,s,t) = \lf(1+\frac{\|\bx\|_1+\|\by\|_1}{(t-s)^{2/3}}\rg)\lf(1+\frac{|s|}{t-s}\rg)(1+|\log(t-s)|).
	$$
	Also $\p(R>a) < ce^{-da}$ for any $a>0$, where $c, d > 0$ are constants depending on $k,\eta$.
\end{lemma}

\section{The biHecke monoid and RSK on \texorpdfstring{$\cC^n(\R)$}{}}
\label{S:sorting-monoids}

In this section we build up the theory of Pitman transforms and the biHecke monoid. One aspect of this theory is a version of the RSK correspondence for functions $f \in \cC^n(\R)$. However, the theory as a whole offers a much richer picture.

\subsection{An abstract theory of sorting}

Let $S_n$ denote the symmetric group with $n$ elements, and let $\pi_i = (i, i + 1)$ denote the adjacent transposition in $S_n$ swapping $i, i + 1$. The adjacent transpositions $\pi_1, \dots, \pi_{n-1}$ generate the symmetric group, together with the relations
\begin{align}
\label{E:braid}
\pi_i \pi_{i+1} \pi_i &= \pi_{i+1} \pi_i \pi_{i+1}, \quad &&i = 1, \dots, n-2 \qquad &&&\text{(Braid relation)}\\
\label{E:commutation}
\pi_i \pi_j &= \pi_j \pi_i, \quad &&|i - j| \ge 2 \qquad &&&\text{(Commutation relation)}\\
\label{E:involution}
\pi_i^2 &= \id, \quad &&i = 1, \dots, n-1 \qquad &&&\text{(Involution relation)}
\end{align}
Now, let 
$
G(S_n)
$
be the monoid of all functions $f:S_n \to S_n$ (with the operation of function composition). For $i = 1, \dots, n-1$ define the \textbf{adjacent sorting operator} $\tau_i \in G(S_n)$ by:
\begin{align*}
\tau_i(\chi) = \begin{cases}
\chi, \qquad &\chi^{-1}(i) > \chi^{-1}(i+1), \\
\pi_i \chi, \qquad &\chi^{-1}(i) < \chi^{-1}(i+1).
\end{cases}
\end{align*}
Define the \textbf{$0$-Hecke monoid} $\mathfrak{M}_n$ as the submonoid of $G(S_n)$ generated by the elements $\tau_i, i = 1, \dots, n-1$, see \cite{Norton1979}.
The following description of the operators $\tau_i$ may also be enlightening. Recall that the inversion number of a permutation $\pi$ is given by:
$$
\operatorname{Inv}(\pi) = \#\{(i, j) \in  \{1, \dots, n\}^2 : i < j, \pi(i) > \pi(j)\}.
$$
Then the operator $\tau_i$ takes the permutation $\chi$ and applies the adjacent transposition $\pi_i$ if and only if doing so increases the inversion number. The effect of this is that composition by the element $\tau_i$ pushes $\chi$ further away from the identity permutation and closer to the reverse permutation $n \cdots 1$. From this point of view, we can think of repeated applications of $\tau_i$ as sorting into reverse order.

Abstractly the $0$-Hecke monoid can be given by a set of relations that is quite similar to the relations above for the symmetric group. Indeed, the $\tau_i$ still satisfy the braid and commutation relations \eqref{E:braid} and \eqref{E:commutation}, but rather than satisfying the involution relation \eqref{E:involution}, they satisfy an idempotent relation: 
\begin{align}
\label{E:idempotent}
\tau_i^2 &= \tau_i, \quad &&i = 1, \dots, n-1 \qquad &&&\text{(Idempotent relation)}
\end{align}
Because the relations for the $\tau_i$ are so similar to the relations for the adjacent transpositions $\pi_i$, the theory of the $0$-Hecke monoid is closely related to the theory of the symmetric group itself. In fact, this theory is almost equivalent to the study of reduced decompositions of elements of the symmetric group, and the map $\pi \mapsto \pi(\id)$ defines a bijection from $\mathfrak{M}_n \to S_n$ that commutes with both the braid and commutation relations.

The RSK correspondence across lines can be viewed as arising from an action of $\mathfrak{M}_n$. This perspective was taken up in \cite{biane2005littelmann}. In our setting, we will need to consider a related action of a larger monoid in $G(S_n)$. 
Define the \textbf{reverse adjacent sorting operator} $\bar \tau_i \in G(S_n)$ by
\begin{align*}
\bar \tau_i(\chi) = \begin{cases}
\chi, \qquad &\chi^{-1}(i) < \chi^{-1}(i+1), \\
\pi_i \chi, \qquad &\chi^{-1}(i) > \chi^{-1}(i+1).
\end{cases}
\end{align*}
The operators $\bar \tau_i$ sort permutations towards the identity, rather than towards the reverse permutation. Let the \textbf{biHecke monoid} $\mathfrak{D}_n$ be the submonoid of $G(S_n)$ generated by the elements $\tau_i, \bar \tau_i, i = 1, \dots, n-1$, see \cite{HST13}. Unlike with $\mathfrak{M}_n$, it does not seem straightforward to describe the biHecke monoid in terms of a set of relations or to relate it directly to the symmetric group. Because of this, we will need the following abstract proposition in order to define actions of $\mathfrak{D}_n$.

\begin{proposition}
	\label{p:semigroup-isom}
	Let $D$ be a partially ordered set and suppose that we have maps $\btau_i:D^n \to D^n$, $\bbtau_i:D^n \to D^n$, $i=1,\ldots n-1$ satisfying the following conditions.
	\begin{enumerate}[(i)]
		\item The maps $\btau_i$ satisfy the braid, commutation, and idempotent relations \eqref{E:braid}, \eqref{E:commutation}, \eqref{E:idempotent} as do the maps $\bbtau_i$. Moreover, we have the mixed commutation relation $\btau_i \bbtau_j = \bbtau_j \btau_i$ whenever $|i-j| \ge 2$.
		\item $\btau_i d=d$ if $d_i \not > d_{i+1}$ and $\bbtau_i d=d$ if $d_i \not < d_{i+1}$.
		\item For $d \in D^n$, let $\Xi_d$ be the set of elements $\xi \in S_n$ such that for $i, j \in \{1, \dots, n\}$ we have
		$$
		d_{i} <d_{j} \qquad \implies \qquad \xi^{-1}(i) < \xi^{-1}(j).
		$$
		Then for any $d\in D^n$ and $i = 1, \dots, n - 1$ we have
		$\Xi_{\btau_i d}=\tau_i \Xi_d$ and $\Xi_{ \bbtau_i d}=\bar \tau_i \Xi_d$. 
		\item $\btau_i\bbtau_id=d$ if $d_i  \not < d_{i+1}$ and $\bbtau_i\btau_id=d$ if $d_i \not > d_{i+1}$.
	\end{enumerate}
	Then:
	\begin{enumerate}[(I)]
		\item For a word $\mathbf w$ in $\btau_i, \bbtau_i$, let $w$ be the corresponding word in $\tau_i, \bar\tau_i$ in $\mathfrak{D}_n$. For any $d\in D^n$,
		$$\Xi_{\mathbf wd}=w\Xi_d.$$
		\item Consider $d \in D^n$, words $\bw,\bw'$ in $\btau_i, \bbtau_i$ and let $w, w'$ be the corresponding words in $\tau_i, \bar\tau_i$ in $\mathfrak{D}_n$. Suppose that there exists $\xi \in \Xi_d$ such that $w \xi = w' \xi$. Then $\bw d=\bw'd$. In particular, if $w, w'$ evaluate to the same word in $\mathfrak{D}_n$, then $\bw, \bw'$ define the same map from $D^n \to D^n$.
		\item The map $\btau_i\mapsto \tau_i$, $\bbtau_i\mapsto \bar \tau_i$ extends to a monoid homomorphism between the monoid $T$ generated by the maps $\btau_i,\bbtau_i,$ $i=1,\ldots, n-1$ and the biHecke monoid $\mathfrak{D}_n$. If there is an element $d \in D^n$ with $d_1 < d_2 < \dots < d_n$ then this is a monoid isomorphism.
	\end{enumerate}
\end{proposition}

Throughout the proof (and in claims (I, II) above) , we use the convention that for a word $\bw = \bw_k \cdots \bw_1$ where $\bw_i \in \{\btau_i, \bbtau_i : i = 1, \dots, n-1\}$, we write $w = w_k \cdots w_1$ for the word in $\mathfrak{D}_n$ given by taking $\bw$ and replacing all instances of $\btau_i, \bbtau_i$ with $\tau_i, \bar \tau_i$. Similarly, we write $\tilde w = \tilde w_k \cdots \tilde w_1$ for the word in $S_n$ given by mapping both $\btau_i, \bbtau_i$ to the adjacent transposition $\pi_i = (i, i+1)$.

\begin{proof}
	\textbf{Part (I).} \qquad 
	By induction, it suffices to check (I) for $\mathbf w=\btau_i$ and $\mathbf w=\bbtau_i$, in which case it is the content of (iii).
	
	\textbf{Part (II).} \qquad Let $\xi \in \Xi_d$ be such that $w \xi = w' \xi$. Let us call the letter $w_\ell$ in the word $w_k\cdots w_1$ {\bf ineffective} for $d$ if $w_\ell \cdots w_1\xi =w_{\ell - 1} \cdots w_1 \xi$. We call a word $w$ effective for $d$ if it has no ineffective letters. By part (I), we have that
	$w_{\ell-1}\cdots w_1 \xi \in \Xi_{\bw_{\ell-1} \cdots \bw_1 d}$. Therefore by (ii), we have that $\bw_{\ell} \cdots \bw_1 d = \bw_{\ell-1} \cdots \bw_1 d$. Therefore, we may drop all ineffective letters from $w, w'$ along with the corresponding letters from $\bw, \bw'$ without changing $w \xi, w' \xi, \bw d, \bw' d$. In other words, to prove part (II) it suffices to show that 
	\begin{equation}
		\label{E:ww'-xi}
		w \xi = w' \xi \qquad \implies \qquad \bw d = \bw' d
	\end{equation}
	whenever $w, w'$ have no ineffective letters.
	
	Now, suppose that the left equation in \eqref{E:ww'-xi} holds for two words $w, w'$ which are effective for $d$. The effectiveness of $w, w'$ implies that every letter in $w, w'$ acts as an adjacent transposition in the composition $w \xi, w' \xi$ and so
	$$
	w \xi = \tilde w \xi, \qquad w' \xi (d) = \tilde w' \xi.
	$$
	This implies that $\tilde w \xi = \tilde w' \xi$, and so $\tilde w, \tilde w'$ represent the same element of the symmetric group. Therefore there is a sequence of words $\tilde w = u^1, \dots, \tilde w' = u^k$ in the alphabet $\pi_1, \dots, \pi_{n-1}$ such that for all $i = 1, \dots, k$, $\tilde w_i$ and $u^i, u^{i+1}$ differ from each other by one of the relations \eqref{E:braid}, \eqref{E:commutation}, \eqref{E:involution}.
	
	Next, given each of the words $u^i$ in the $\pi_i$, there is a unique way to change each letter $\pi_i$ to either $\btau_i$ or $\bbtau_i$ so that every letter in the resulting word $w^i$ is effective. Moreover, this change is consistent in the sense that the assignment of bars agrees on consecutive words $w^i, w^{i-1}$ everywhere except where the relation is used.
	This procedure also gives rise to a sequence of words $\bw = \bw^1, \dots, \bw^k= \bw'$ in $\btau_i, \bbtau_i$.

	To complete the proof of (II), we just need to show that $\mathbf w^{i} d = \mathbf w^{i-1} d$ for all $i = 2, \dots, k$. First suppose that $u^{i-1}, u^i$ differ by a commutation relation \eqref{E:commutation} so that $u^{i-1} = \tilde a \pi_\ell \pi_j \tilde b, u^i = \tilde a \pi_j \pi_\ell \tilde b$. Then there is a unique choice of $f(\pi_\ell) \in \{\btau_\ell, \bbtau_\ell\}$ and $f(\pi_j) \in \{\btau_j, \bbtau_j\}$ such that
	$$
	\bw^{i-1} = \ba f(\pi_\ell) f(\pi_j) \bb, \qquad \bw^i = \ba f(\pi_j) f(\pi_\ell) \bb.
	$$
	By the commutation relation in (i), we have that $\bw^{i-1} = \bw^i$. Next, suppose that $u^{i-1}, u^i$ differ by an involution relation \eqref{E:involution} so that $u^{i-1} = \tilde a \pi_j^2 \tilde b, u^i = \tilde a \tilde b$ (or vice versa, with the $\pi_j^2$ on $u^i$). Then $\bw^i = \ba \bb$ and either
	$$
	\bw^{i-1} = \ba \btau_j \bbtau_j \bb \qquad \text{ or } \qquad \bw^{i-1} = \ba \btau_j \bbtau_j \bb.
	$$
	Without loss of generality, assume that $\bw^{i-1} = \ba \btau_j \bbtau_j \bb$. First, since the word $w^i$ is effective for $d$ we have that $\bar \tau_j b \xi \ne b \xi$ and so $[b \xi]^{-1}(j) > [b \xi]^{-1}(j + 1).$ Now by part i, $b \xi \in \Xi_{\bb d}$, and so $\bb d_j \not < \bb d_{j+1}$.
	Therefore by (iv), $\btau_j \bbtau_j$ acts as the identity in $\bw^{i-1} d$, and so $\bw^{i-1} d = \bw^i d$. 
	
	Finally, suppose that $u^{i-1}, u^i$ differ by a braid relation \eqref{E:braid}, so that $u^{i-1} = \tilde a \pi_j \pi_{j+1} \pi_j \tilde b$ and $u^{i} = \tilde a  \pi_{j+1} \pi_j \pi_{j+1} \tilde b$ (or vice versa). Without loss of generality, we may assume $j = 1$. 
	Here we divide into cases based on the relative order of $\al_1, \al_2, \al_3$, where $\al_i := [b \xi]^{-1}(i)$, as the relative order of the $\al_i$ determines the barring on the operators $\btau_1, \btau_{2}$ in $\bw^i, \bw^{i-1}$. 
	
	\textbf{Case 1:} $\al_1 < \al_2 < \al_3$ or $\al_1 > \al_2 > \al_3$. \qquad In the first of these cases, $\bw^{i-1} = \ba  \btau_{1} \btau_2\btau_1 \bb$ and $\bw^{i} = \ba  \btau_2 \btau_1\btau_2 \bb$ and in the second of these cases, $\bw^{i-1} = \ba \bbtau_1 \bbtau_2\bbtau_1 \bb$ and $\bw^{i} = \ba  \bbtau_2 \bbtau_1\bbtau_2 \bb$. In either case, $\bw^{i-1} = \bw^i$ by the braid relation in (i).
	
	\textbf{Case 2:} $\al_3 < \al_1 < \al_2$ or $\al_3 > \al_1 > \al_2$. \qquad We only deal with the first of these options, since as in Case 1, the second option is symmetric after swapping bars and non-bars. In this case, $\bw^{i-1} = \ba  \bbtau_1 \bbtau_2 \btau_1 \bb$ and $\bw^{i} = \ba  \btau_2\bbtau_1 \bbtau_2 \bb$. Letting $e = \bb d$, we have that $b \xi \in \Xi_e$, and so
	$$
	e_3 \not > e_1, \qquad e_3 \not > e_2, \qquad e_1 \not > e_2.
	$$
	Now, we have that
	$$
	\btau_2\bbtau_1 \bbtau_2 e = \btau_2\bbtau_1 \bbtau_2 \bbtau_1 \btau_1 e = \btau_2\bbtau_2 \bbtau_1 \bbtau_2 \btau_1 e.
	$$
	Here the first equality follows from (iv). The second equality is a braid relation for the $\bb\tau_i$. Setting $f = \bbtau_1 \bbtau_2 \btau_1 e$, by (I) we have that $\bar \tau_1 \bar \tau_2 \tau_1 b \xi \in \Xi_f$ and so the ordering on the $\al_i$ guarantees that $f_2 \not > f_3$. Therefore $\btau_2\bbtau_2$ acts as the identity on $f$ by (iv), and so $\btau_2\bbtau_1 \bbtau_2 e = \bbtau_1 \bbtau_2 \btau_1 e$ and hence $\bw^{i-1} d = \bw^i d$, as desired.
	
	\textbf{Case 2':} $\al_1 < \al_3 < \al_2$ or $\al_1 > \al_3 > \al_2$. \qquad Again, we only deal with the first option. In this case, $\bw^{i-1} = \ba  \btau_1 \bbtau_2 \bbtau_1 \bb$ and $\bw^{i} = \ba  \bbtau_2\bbtau_1 \btau_2 \bb$, and the same proof as in Case 2 will work with the roles of $\btau_1$, $\bbtau_1$ and $\btau_2$, $\bbtau_2$ reversed. This completes the proof of (II).
	
	\textbf{Part (III).} \qquad The fact that the map $w \mapsto \bw$ is a monoid homomorphism follows from (II). To see that it is an isomorphism when there exists $d \in D^n$ be such that $d_1 < d_2 < \dots < d_n$, consider two distinct elements $w, w' \in \mathfrak{D}_n$ and a permutation $\chi$ with $w \chi \ne w' \chi$. Let $w'' \in \mathfrak{D}_n$ be an element with $w'' \id = \chi$, and let $d \in D^n$ be such that $d_1 < d_2 < \dots < d_n$. Then $\Xi_{\bw' \bw'' d} = w' w'' \Xi_d = w' \chi$, and $\Xi_{\bw \bw'' d} = w w'' \Xi_d = w \chi$. Since $w \chi \ne w' \chi$ this implies that $\bw \ne \bw'$.
\end{proof}

\begin{example}	
	\label{E:basic-action}
	The following is essentially the simplest example of an application of Proposition \ref{p:semigroup-isom}. Consider $\bx \in \R^n$, and for $i = 1, \dots, k-1$ define
	\begin{align*}
		\btau_i(\bx) &= \begin{cases}
			\bx, \qquad & x_i \ge x_{i+1}, \\
			(x_1, \dots, x_{i-1},  x_{i+1}, x_i, x_{i+2},\dots, x_n), \qquad &x_i < x_{i+1},
		\end{cases} \\
		\bbtau_i(\bx) &= \begin{cases}
			\bx, \qquad & x_i \le x_{i+1}, \\
			(x_1, \dots, x_{i-1},  x_{i+1}, x_i, x_{i+2},\dots, x_n), \qquad &x_i > x_{i+1},
		\end{cases}
	\end{align*}
	It is straightforward to check that these maps satisfy the conditions of Proposition \ref{p:semigroup-isom}, and hence extend to an action of $\mathfrak{D}_n$ on $\R^n$. We call this action the \textbf{basic action} of $\mathfrak{D}_n$ on $\R^n$.
\end{example}
Example \ref{E:basic-action} shows that the biHecke monoid can be used to sorts lists with possibly equal elements.

\subsection{\texorpdfstring{$\mathfrak{D}_n$}{}-actions and Pitman operators}
\label{S:Dn-RSK}

We next define a $\mathfrak{D}_n$-action on the space $\cC^n(\R)$ of continuous functions $f:\R \to \R^n$ with an asymptotic slope. First, define the \textbf{Pitman operator} $\cP: \cC^2(\R) \to \cC^2(\R)$ by following rules. If $\la_1(f) \ge \la_2(f)$, then set $\cP f = f$. If $\la_1(f) < \la_2(f)$, then define
\begin{equation}
\label{E:Pitman}
\begin{split}
\cP f_1(x) &= f[(-\infty, 2) \to (x, 1)] = f_1(x) + \sup_{z \le x} [f_2(z) - f_1(z)], \\ \qquad \cP f_2(x) &= f_1(x) + f_2(x) - \cP f_1(x) =  f_2(x) - \sup_{z \le x} [f_2(z) - f_1(z)].
\end{split}
\end{equation}
In the above expression and in the remainder of the paper, we write $\cP f_i := (\cP f)_i$ to streamline notation.
Next, define the \textbf{co-Pitman operator} $\bar \cP:\cC^2(\R) \to \cC^2(\R)$ by letting $\cP f = f$ if $\la_1(f) \le \la_2(f)$, and setting
\begin{equation}
\label{E:reverse-Pitman}
\begin{split}
\bar \cP f_1(x) &= f_1(x) + \sup_{z \ge x} [f_2(z) - f_1(z)], \\ \qquad \bar \cP f_2(x) &= f_2(x) - \sup_{z \ge x} [f_2(z) - f_1(z)]
\end{split}
\end{equation}
if $\la_1(f) > \la_2(f)$. The co-Pitman operator can be given by conjugating the Pitman operator. Indeed, if we let $Rf(x) = f(-x)$ denote the reflection map, then
$
\bar \cP = R \cP R.
$
Similarly, if we let $\hat R f_1(x) = - f_2(-x)$ and $\hat R f_2(x) = -f_1(-x)$ denote the 180-degree rotation of the environment, then $
\bar \cP = \hat R \cP \hat R.
$

Next, for $f \in \cC^n(\R)$ and $i = 1, \dots, n-2$ define 
$$
\cP_{\tau_i} f = (f_1, \dots, f_{i-1}, \cP(f_i, f_{i+1}), f_{i+2}, \dots, f_n),
$$
and similarly set $\cP_{\bar \tau_i} = R \cP_{\tau_i} R$. Our goal in the remainder of this section is to extend the Pitman operators $\cP, \bar \cP$ to actions of the biHecke monoid on $\cC^n(\R)$. To do so, we check the conditions of Proposition \ref{p:semigroup-isom}. We first record an \emph{isometric} property for Pitman operators. This isometric property is one of the key reasons that Pitman operators are so useful in the study of last passage percolation.

Following \cite{dauvergne2022rsk}, we say that two environments $f, g \in \cC^n(\R)$ are \textbf{boundary isometric} and write $f \sim g$ if
\begin{equation}
\label{E:iso-def}
f[(\bx, n) \to (\by, 1)] = g[(\bx, n) \to (\by, 1)]
\end{equation}
for any vectors $\bx, \by \in \R^k_\le, k \in \N$. We say that a map $F:\cC^n(\R) \to \cC^n(\R)$ is a (boundary) isometry if $F f \sim f$ for all $f \in \cC^n(\R)$.

\begin{lemma}
\label{L:two-line-isometry}
For all $n \in \N$ and $i \in \{1, \dots, n-1\}$, the operators $\cP_{\tau_i}, \cP_{\bar \tau_i}$ are boundary isometries.
\end{lemma}

\begin{proof} 
We will just check that $\cP_{\tau_i}$ is an isometry. The claim for $\cP_{\bar \tau_i}$ follows from the fact that $X$ is an isometry if and only if $\hat R X \hat R$ is an isometry for any map $X$. If $\la_i(f) \ge \la_{i+1}(f)$, then $\cP_{\tau_i} f = f$ and the claim is immediate. Now suppose $\la_i(f) < \la_{i+1}(f)$. We check that \eqref{E:iso-def} holds for a fixed $\bx, \by$. Let $T = \argmax_{z \le x_1} f_{i+1}(z) - f_i(z)$, and let $f' = f - f(T)$. Observing that $f \sim f'$ and $\cP_{\tau_i} f \sim \cP_{\tau_i} f - f(T) = \cP_{\tau_i} f'$, it suffices to show that \eqref{E:iso-def} holds for $\bx, \by$ with $f, g$ replaced by $f', \cP_{\tau_i} f'$. In this case, $f'(T) = 0$ and by the definition of $T$, for $x_1 \le z$ we have that
$$
\cP_{\tau_i} f_i'(z) = f'[(T, 2) \to (z, 1)], \qquad \cP_{\tau_i} f_{i+1}' + \cP_{\tau_i} f_i' = f_{i+1} + f_i.
$$
For an operator defined using the left-hand sides above on continuous functions $f':[T, \infty) \to \R^n$ with $f'(T) = 0$, boundary isometry is proven as Lemma 4.3 in \cite{DOV}.
\end{proof}

We are now ready to prove the basic algebraic relations for Pitman operators.
\begin{lemma}
	\label{L:Pitman-properties}
	The Pitman operators $\cP_{\tau_i}, \cP_{\bar \tau_i}$ have the following properties.
	\begin{enumerate}[i.]
		\item Recall that $\la:\cC^n(\R) \to \R^n$ denotes the slope map, and let $\btau_i, \bbtau_i$ denote the generators for the basic $\mathfrak{D}_n$-action on $\R^n$ introduced in Example \ref{E:basic-action}.
		Then for $f \in \cC^n(\R)$ and $i = 1, \dots, n-1$ we have $\la \cP_{\tau_i} f = \btau_i \la f$ and $\la \cP_{\bar \tau_i} f = \bbtau_i \la f$.
		\item The maps $\cP_{\tau_i}$ satisfy the braid, commutation, and idempotent relations \eqref{E:braid}, \eqref{E:commutation}, \eqref{E:idempotent} as do the maps $\cP_{\bar \tau_i}$. Moreover, we have the mixed commutation relation $\cP_{\bar \tau_i} \cP_{\tau_j} = \cP_{\tau_j} \cP_{\bar \tau_i}$ when $|i-j| \ge 2$.
		\item $\cP_{\tau_i}\cP_{\bar \tau_i} f = f$ if $\la_i(f) \le \la_{i+1}(f)$ and $\cP_{\bar \tau_i}\cP_{\tau_i} f = f$ if $\la_i(f) \ge \la_{i+1}(f)$.
	\end{enumerate}
\end{lemma}

\begin{proof}
\textbf{Part i.} \qquad Here we just need to observe that if $\la_1(f) < \la_2(f)$ then $\la(\cP f) = (\la_2 (f), \la_1 (f))$.	
	
\textbf{Part ii.} \qquad The commutation relation \eqref{E:commutation} and the mixed commutation relation are immediate. The idempotent relation \eqref{E:idempotent} follows from part i, which implies that $\la_1(\cP f) \ge \la_2(\cP f)$ for any $f \in \cC^2(\R)$, and hence $\cP$ acts as the identity on its image. It suffices to check braid relation when $n = 3$, and by symmetry, only the unbarred identity
\begin{equation}
\label{E:unbarred-braid}
\cP_{\tau_1} \cP_{\tau_2} \cP_{\tau_1} f = \cP_{\tau_2} \cP_{\tau_1} \cP_{\tau_2} f.
\end{equation}
 We divide into cases, depending on the order of the slope $\la(f)$.
 
 \textbf{Case 1: $\la_1(f) < \la_2(f) < \la_3(f)$.} \qquad Letting $g, h$ denote the left- and right-hand sides of \eqref{E:unbarred-braid}. Observing that Pitman operators $\cP_{\tau_i}$ preserve the sum $f_1 + f_2 + f_3$, it suffices to show that $g_1 = h_1$ and $g_3 = h_3$. We will only prove the first of these equalities as the second follows from a symmetric argument. We claim that 
 \begin{equation}
 \label{E:claimed-iso}
 h_1 = (\cP_{\tau_1} \cP_{\tau_2} f)_1 = g_1.
 \end{equation}
 The first equality in \eqref{E:claimed-iso} is immediate since $\cP_{\tau_2}$ does not affect line $1$. For the second equality, using that $\la_1(f) \vee \la_2(f) < \la_3(f)$ we have that 
  \begin{equation}
 \label{E:claimed-iso-2}
(\cP_{\tau_1} \cP_{\tau_2} f)_1(x) = f_1(x) + \sup_{z_2 \le z_1 \le x} [f_2(z_1) - f_1(z_1)] + [f_3(z_2) - f_2(z_2)].
 \end{equation}
 In other words, we get the top line $(\cP_{\tau_1} \cP_{\tau_2} f)_1$ by reflecting $f_2$ off of $f_3$, and then reflecting $f_1$ off of the result. Next, we can rewrite the right-hand side of \eqref{E:claimed-iso-2} as
 $$
 \sup_{z_2 \le x} f_3(z_2) + f[(z_2, 2) \to (x, 1)].
 $$
 By Lemma \ref{L:two-line-isometry}, this equals 
 $$
 \sup_{z_2 \le x} f_3(z_2) + \cP_{\tau_1} f[(z_2, 2) \to (x, 1)] =\sup_{z_2 \le x} \cP_{\tau_1} f_3(z_2) + \cP_{\tau_1} f[(z_2, 2) \to (x, 1)].
 $$
 Now, by part i, the inequality $\la_1(f) \vee \la_2(f) < \la_3(f)$ is preserved by the map $\cP_{\tau_1}$. Hence the equality \eqref{E:claimed-iso-2} also holds with $\cP_{\tau_1} f$ in place of $f$. Putting this together with the previous three displays we get that $(\cP_{\tau_1} \cP_{\tau_2} f)_1 = (\cP_{\tau_1} \cP_{\tau_2} \cP_{\tau_1} f)_1$, giving the second equality in \eqref{E:claimed-iso}.
 
 \medskip
 
 \textbf{Case 2: $\la_i(f) \ge \la_{i+1}(f)$ for some $i \in \{1, 2\}$.} \qquad In this case, using the slope-interchange property in part i and working through different cases, in the composition $\cP_{\tau_1} \cP_{\tau_2} \cP_{\tau_1} f$ we have that
 \begin{itemize}
 	\item The rightmost $\cP_{\tau_1}$-operator acts as the identity if $i = 1$;
 	\item The leftmost $\cP_{\tau_1}$-operator acts as the identity if $i = 2$.
 \end{itemize}
 Similarly, in the composition $\cP_{\tau_2} \cP_{\tau_1} \cP_{\tau_2} f$ we have that
  \begin{itemize}
 	\item The rightmost $\cP_{\tau_2}$-operator acts as the identity if $i=2$;
 	\item The leftmost $\cP_{\tau_2}$-operator acts as the identity if $i=1$.
 \end{itemize}
 Therefore if $i=1$ then both sides of \eqref{E:unbarred-braid} are equal to $\cP_{\tau_1} \cP_{\tau_2} f$, and if $i=2$ then both sides of \eqref{E:unbarred-braid} are equal to $\cP_{\tau_2} \cP_{\tau_1} f$.
 
\textbf{Part iii.} \qquad This is also shown in the appendix of \cite{seppalainen2023busemann}, though the language used there is different. We only check the first identity as the second follows by symmetry. If $\la_i(f) = \la_{i+1}(f)$, then the identity holds since both $\cP_{\tau_i}$ and $\cP_{\bar \tau_i}$ act as the identity. Now assume $\la_i(f) < \la_{i+1}(f)$. Since the sum $f_i + f_{i+1}$ is preserved by both $\cP_{\tau_i}$ and $\cP_{\bar \tau_i}$ it suffices to check that $\cP_{\bar \tau_i} \cP_{\tau_i} f_i = f_i$. Define
 $
 S f(x) = \sup_{z \le x} f_{i+1}(x) - f_i(x)
 $
 so that
 $$
 \cP_{\bar \tau_i} \cP_{\tau_i} f_i(x) = f_i (x) + S f(x) + \sup_{z \ge x} [f_{i+1} (z) - f_i(z) - 2 S f(z)].
 $$
Noting that $S f \ge f_{i+1} - f_i$ and that $Sf$ is non-decreasing we can see that 
$$
\cP_{\bar \tau_i} \cP_{\tau_i} f_1(x) \le f_1 (x) + S f(x) + \sup_{z \ge x} [- S f(z)] = f_1 (x).
$$
On the other hand, since $f_1, f_2$ are continuous and the difference $f_2(z) - f_1(z) \to \infty$ with $z$, there must exist $z_0 \ge x$ where  $f_2(z_0) - f_1(z_0) = S f(z_0) = Sf(x).$ Therefore $\cP_{\bar \tau_i} \cP_{\tau_i} f_1(x) \ge f_1(x)$, yielding the result.
\end{proof}

Given Lemma \ref{L:Pitman-properties}, we can use Proposition \ref{p:semigroup-isom} to extend the definition of Pitman operators to the whole monoid $\mathfrak{D}_n$. Indeed, for $\sig \in \mathfrak{D}_n$, let $\xi_k \cdots \xi_1 = \sig$ where each $\xi_i \in \{\tau_i, \bar \tau_i : i = 1, \dots, n-1\}$ and define
$$
\cP_{\sig} = \cP_{\xi_k} \cdots \cP_{\xi_1}.
$$ 
\begin{corollary}
	\label{C:Pitman-iso}
The map $\sig \mapsto \cP_{\sig}:\cC^n(\R) \to \cC^n(\R)$ is a monoid isomorphism of $\mathfrak{D}_n$. Moreover, letting $\la:\cC^n(\R) \to \R^n$ denote the slope map, for any $\sig \in \mathfrak{D}_n$ we have 
$$
\la \circ \cP_{\sig} = \sig \circ \la
$$
where on the right-hand side of this equation $\sig$ acts on $\R^n$ through the basic action. 
\end{corollary}

\begin{proof}
To prove that the map $\sig \mapsto \cP_{\sig}$ is unambiguously defined and yields a monoid isomorphism we check the conditions of Proposition \ref{p:semigroup-isom} where $D = \cC^1(\R)$ is given the partial order induced by the slope map $\la$. Property (i) in Proposition \ref{p:semigroup-isom} is guaranteed by Lemma \ref{L:Pitman-properties}.ii, property (ii) is guaranteed by the definition, property (iii) is guaranteed by Lemma \ref{L:Pitman-properties}.i, and property (iv) follows from Lemma \ref{L:Pitman-properties}.iii. The `Moreover' claim then follows by Lemma \ref{L:Pitman-properties}.i again.
\end{proof}

The Pitman transforms described in this section have more structure than the abstract monoid actions in Proposition \ref{p:semigroup-isom}. In particular, we can describe orbits of elements in $\cC^n(\R)$ using this structure.

\begin{lemma}
	\label{L:orbits}
Consider the action $\sig \mapsto \cP_{\sig}:\cC^n(\R) \to \cC^n(\R)$, and for $f \in \cC^n(\R)$ let $O(f) = \{ \cP_\sig f : \sig \in \mathfrak{D}_n\}$ denote its orbit. Let $O(\la(f))$ be the orbit of $\la(f)$ under the basic action of $\mathfrak{D}_n$ on $\R^n$. Then:
	\begin{enumerate}[i.]
		\item For any $g \in O(f)$ there exists $\sig \in \mathfrak{D}_n$ such that $\cP_\sig g = f$ and so $O(f) = O(g)$.
		\item The slope map $g \mapsto \la(g)$ is a bijection from $O(f)$ to $O(\la(f))$.
	\end{enumerate}
\end{lemma}

\begin{proof}
For part i, if $g \in O(f)$ then $g = \cP_\sig f$ for some word $\sig = \pi_k \cdots \pi_1$ where  $\pi_j \in \{\tau_i, \bar \tau_i: i = 1, \dots, n-1\}$. Then from the definition of $\cP$ and Lemma \ref{L:Pitman-properties}.iii there exists a word $\sig' = \pi_1' \cdots \pi_k'$ where each $\pi_j'$ equals either $\pi_j$ or its barred/unbarred version, such that $\cP_{\sig' \sig} f = \cP_{\sig'} g = f$.

 For part ii, by Corollary \ref{C:Pitman-iso} the map $\la:O(f) \to O(\la(f))$ is onto. Now suppose that $g_1, g_2 \in O(f)$ are such that $\la(g_1) = \la(g_2)$. Using the notation $\Xi_{g_1}$ from Proposition \ref{p:semigroup-isom}, let $\xi \in \Xi_{g_1}$. We will aim to find $\sig \in \mathfrak{D}_n$ with $\sig \xi = \xi$ and $\cP_\sig g_1 = g_2$. If we can find such a $\sig$, then by Proposition \ref{p:semigroup-isom}(II), we have that $g_2 = \cP_\sig g_1 = \cP_{\id} g_1 = g_1$, as desired. First, since $\la:O(f) \to O(\la(f))$ is onto we can find $h \in O(f) = O(g_1)$ with $\la(h)_1 \ge \dots \ge \la(h)_k$. Let $\sig_1 \in \mathfrak{D}_n$ be such that $\cP_{\sig_1} g_1 = h$. Since $O(h) = O(f)$ by part i, we can then find $\sig_2$ such that $\cP_{\sig_2 \sig_1} g_1 = \cP_{\sig_2} h = g_2$. We can write $\sig_1 = \hat \tau_{i_k} \cdots \hat \tau_{i_1}$ and $\sig_2 = \hat \tau_{j_\ell} \cdots \hat \tau_{j_1}$ where $\hat \tau_i \in \{\tau_i, \bar \tau_i\}$ for all $i$. 
 We may also assume that $\sig_1, \sig_2$ are effective, in the sense that for any $a \le k$ or $b \le \ell$ we have
 \begin{equation*}
 \label{E:effective}
 \cP_{\hat \tau_{i_a} \cdots \hat \tau_{i_1}} g_1 \ne  \cP_{\hat \tau_{i_{a-1}} \cdots \hat \tau_{i_1}} g_1, \qquad  \cP_{\hat \tau_{j_b} \cdots \hat \tau_{j_1}} h \ne  \cP_{\hat \tau_{j_{b-1}} \cdots \hat \tau_{j_1}} h.
 \end{equation*}
Now, let $\tilde \sig_1 = \pi_{i_k} \cdots \pi_{i_1}$ and $\tilde \sig_2 = \pi_{j_\ell} \cdots  \pi_{j_1}$ be the corresponding products of adjacent transpositions in $S_n$. Effectiveness of $\sig_1, \sig_2$ implies that $\tilde \sig_2 \tilde \sig_1 \xi = \sig_2 \sig_1 \xi$. Next, let $\Pi(h)$ be the set of adjacent transpositions $\pi_i$ such that $\la(h)_i = \la(h)_{i+1}$, and let $H(h) \subset S_n$ be the subgroup generated by $\Pi(h)$. Consider $\tilde \ka \in H(h)$, and let $\pi_{m_k} \cdots \pi_{m_1}$ be a reduced word for $\tilde \ka$, written in the alphabet $\Pi(h)$. There is a unique way to map the $\pi_{m_i}$ to $\{\tau_{m_i}, \bar \tau_{m_i}\}$ such that the resulting element $\ka \in \mathfrak{D}_n$ is effective on $\sig_1 \xi$. Hence $\tilde \ka \tilde \sig_1 \xi = \kappa \sig_1 \xi$. On the other hand, from the definition of the Pitman transform applied to lines of equal slope, $ \cP_{\ka \sig_1} g_1 = \cP_{\kappa} h  = h$ for all $\ka \in H(h)$, so by the effectiveness of $\sig_2$ on $h$ we have that $\tilde \sig_2 \ka \sig_1 \xi = \sig_2 \ka \sig_1 \xi$. We also have that $\cP_{\sig_2 \ka \sig_1} g_1 = g_2$ and hence $\sig_2 \ka \sig_1 \xi = \tilde \sig_2 \tilde \ka \tilde \sig_1 \xi \in \Xi_{g_2} = \Xi_{g_1}$. Finally, the map
$$
\tilde \kappa \mapsto \sig_2 \ka \sig_1 \xi = \tilde \sig_2 \tilde \ka \tilde \sig_1 \xi
$$
from $H(h) \to \Xi_{g_1}$ is a bijection. Indeed, this map is one-to-one by the invertibility of $\tilde \sig_2, \tilde \sig_2$ and it is easy to see that $|H(h)| = |\Xi_{g_1}|$ (these are conjugate subgroups). Hence there must be some $\ka$ with $\sig_2 \ka \sig_1 \xi = \xi$. Setting $\sig = \sig_2 \ka \sig_1$ then gives the desired sorting element.
\end{proof}

\begin{corollary}\label{C:samestart} Let $f \in \cC^n(\R)$, and let $g \in O(f)$, where $O(f)$ is the $\cP$-orbit of $f$. Let $I$ be a set of the form $\{1,\ldots,k\}$ or $\{k,\ldots,n\}$, and suppose that
	$
	\la(f)_j = \la(g)_j
	$ for all $j \in I$. Then $f_j=g_j$ for all $j\in I$.
\end{corollary}
\begin{proof}
We can find $\sig \in \mathfrak{D}^n$ with $\sig \la(f) = \la(g)$, where $\sig$ acts through the basic action and can be written as a product of adjacent transpositions that do not use coordinates of $I$. By Lemma \ref{L:orbits}, $\cP_\sig f = g$, and since $\cP_\sig$ only affects coordinates in $I^c$, $f_j=g_j$ for all $j\in I$.
\end{proof}

\subsection{Sorting and last passage percolation}
The goal of this section is to represent the $\mathfrak{D}_n$-action from Corollary \ref{C:Pitman-iso} in terms of last passage values.  First consider $f \in \cC^n(\R)$ with $\la(f) \in \R^n_\le$, and recall the definition of last passage from $-\infty$ in \eqref{E:finftyM}. The limit in that definition exists when $\bm \in \{1, \dots, n\}^k_<$ satisfies the following property:
\begin{equation}
\label{E:N-property}
\text{If $m_i$ is a coordinate of $\bm$ but $m_i - 1$ is not, then $\la_{m_{i-1}} < \la_{m_i}$}.
\end{equation} 
Moving forward, we will say $\bm$ satisfies \eqref{E:N-property} with respect to $\la$ when the slope vector is not clear context. We call a $k$-tuple of paths $\pi = (\pi_1, \dots, \pi_k)$, where each $\pi_j:(-\infty, x_j] \to \{1, \dots, n\}$ is a nonincreasing cadlag path satisfying $\pi_j(x_j) = 1$ and $\lim_{t \to -\infty} \pi_j(t) = m_j$ a \textbf{disjoint optimizer} from $(-\infty, \bm)$ to $(\bx, 1)$ if $\pi|_{[t, x_k]}$ is a disjoint optimizer for all $t < x_1$. The condition \eqref{E:N-property} guarantees that disjoint optimizers from $(-\infty, \bm)$ to $(\bx, 1)$ exist and that they are pointwise limits of disjoint optimizers from $(t, \bm)$ to $(\bx, 1)$ as $t \to -\infty$.

For $f \in \cC^n(\R)$, we can describe elements of its $\cP$-orbit $O(f)$ using last passage percolation. We start with a few examples before moving towards a general theory.

\begin{example}[LPP to the top]
	\label{Ex:cars}
Let $f \in \cC^n(\R)$ and suppose that $\la(f)_k < \la(f)_j$ for all $j < k$.	Then 
	\begin{equation}\label{e:cars}
		f[(-\infty, k)\to (x,1)]=\cP_{\tau_1\cdots \tau_{k-1}} f_1(x).
	\end{equation}
This is almost immediate from the definition and the fact that the Pitman transform $\cP_{\tau_i}$ switches the slopes of $f_i, f_{i+1}$ when $\la(f)_i < \la(f)_{i+1}$.  
\end{example}
\begin{example}[Zigzag last passage percolation]
\label{Ex:zigzag}
Let $f \in \cC^n(\R)$, suppose that $\la(f) \in \R^n_<$. Let $g \in O(f)$, and let $\xi \in S_n$ be such that 
$$
\la(g)_{\xi^{-1}(1)} < \cdots < \la(g)_{\xi^{-1}(n)}.
$$
Then for every $j \in \{1, \dots, n\}$, 
$$
f[(-\infty, j)\to (x,1)]=\cP_{\sigma_j}g_1(x), \qquad \sigma_j=\sigma_{j,1}\cdots \sigma_{j,\xi^{-1}(j)-1}, \!\!\qquad \sigma_{j,i}=\begin{cases} \tau_i, \quad  &j<\xi_g(i)\\ \bar \tau_i, &j>\xi_g(i)\end{cases} 
$$
Indeed, by the previous example we have that $f[(-\infty, j)\to (x,1)] = \cP_{\tau_1\cdots \tau_{j-1}} f_1(x)$, and moreover $\la( \cP_{\tau_1\cdots \tau_{j-1}} f)_1 = \la(f)_j$. Similarly, by construction $\la(\cP_{\sig_j} g)_1 = \la(f)_j$. Therefore by Corollary \ref{C:samestart} we have that $\cP_{\tau_1\cdots \tau_{j-1}} f_1 = \cP_{\sig_j} g_1$. We can rewrite $\cP_{\sig_j} g_1$ more explicitly as a kind of \textbf{zig-zag last passage percolation}. Let $k = \xi^{-1}(j)$. Then
	$$
\cP_{\sigma_j}g_1(x)=g_1(x)+\sup_{\by \in \R^{k}}\left(\sum_{i=2}^{k} g_{i}(y_i)-g_{i-1}(y_i)\right),
$$
where the supremum is over all vectors $\by \in \R^{k}$ with $y_1 = x$ and satisfying the inequalities $y_i \le y_{i-1}$ or $y_{i} \le y_{i-1}$ for all $2 \le i \le k$ depending on whether $\la(g)_k > \la(g)_{i-1}$ or $\la(g)_k < \la(g)_{i-1}$. Remarkably, the Brownian Burke theorem (see Section \ref{S:Pitman-Brownian}) implies that we can construct models of zig-zag last passage percolation across independent Brownian motions whose stationary measures are themselves Brownian motions!
\end{example}

Both of the previous examples are special cases of the following general proposition. 

\begin{proposition}[Orbit elements as last passage values]
	\label{P:orbits} Let $f \in \cC^n(\R)$ with $\la(f) \in \R^n_\le$, let $g \in O(f)$, let $k \in \{1, \dots, n\}$, and let $\bx \in \R^k_\le$. Recall that $\Xi_g \subset S_n$ is the set of permutations satisfying 
	\begin{equation}
		\label{E:lag}
		\la(g)_{\xi_g^{-1}(1)} \le \cdots \le \la(g)_{\xi_g^{-1}(n)}.
	\end{equation}
	There exists a unique vector $\bm \in \{1, \dots, n\}^k_<$ satisfying \eqref{E:N-property} for $\la(f)$ such that $\bm = \xi\{1, \dots, k\}$ for some $\xi \in \Xi_g$ (here thinking of $\bm$ as a subset). We have 
		\begin{equation}
			\label{E:g1gk}
			g[(-\infty, (1, \dots, k)) \to (\bx, 1)]=  f[(-\infty,\bm)\to (\bx,1)]. \qquad 
		\end{equation}
		In particular, if $x = x_1 = \dots = x_n$ then 
		$$
		g_1(x) + \dots + g_k(x) = f[(-\infty,\bm)\to (x^k,1)].
		$$
\end{proposition}

The proof of Proposition \ref{P:orbits} is an induction based on the following lemma. 

	\begin{lemma}\label{L:block-up}
	Let $f \in \cC^n(\R)$. Define $\sig = \tau_{n-1} \cdots \tau_1$. If $\la(f)_1 < \la(f)_i$ for all $i > 1$, then for all $\bx \in \R^{n-1}_\le$:
	\begin{equation}
	\label{E:Pitman-unnormal}
	f[(-\infty, (2,\ldots , n))\to (\bx,1)] = \cP_{\sig} f[(-\infty,(1,\ldots, n-1))\to (\bx,1)]. 
	\end{equation}
	In particular, if $x_1 = \dots = x_{n-1} = x$ then:
	\begin{equation}
		\label{E:basic-xx}
	f[(-\infty, (2,\ldots , n))\to (x^{n-1},1)] = \sum_{i=1}^{n-1}\cP_{\sig} f_i(x).
	\end{equation}
\end{lemma}

\begin{proof}
We actually prove the restricted case \eqref{E:basic-xx} first and then use this to prove the general version \eqref{E:Pitman-unnormal}.
To shorten notation, we write $[i, j] := (i, i+1, \dots, j) \in \R^{j-i + 1}$ through the proof. First, for $t \le x$ we have
\begin{equation}
	\label{E:finfty-sum}
f[(t, [1, n])\to (x^n,1)]= \sum_{i=1}^n f_i(x) - f_i(t) = \sum_{i=1}^n \cP_\sig f_i(x) - f_i(t).
\end{equation}
The first equality is by definition, and the second equality uses that Pitman transforms preserve the sum of all lines. Moreover, it is easy to check from the definition that
\begin{align*}
f[(t,& [2, n])\to (x^n,1)] \\
&= f[(t, [1, n])\to (x^n,1)] +f_1(t) - f_n(x)  + \max_{t \le y_1 \le \dots \le y_{n-1} \le x} \sum_{i=1}^{n-1} f_{i + 1}(y_i) - f_{i}(y_i).
\end{align*}
Now, since $\la(f)_1 < \la(f)_i$ for all $i > 1$, the final term above equals
$$
\max_{y_1 \le \dots \le y_{n-1} \le x} \sum_{i=1}^{n-1} f_{i + 1}(y_i) - f_{i}(y_i).
$$
for all small enough $t$, and this minimum is attained. Therefore by \eqref{E:finfty-sum}, to verify \eqref{E:basic-xx} we just need to show that
\begin{equation}
	\label{E:P-sigma}
\cP_\sig f_n(x) = f_n(x) - \max_{y_1 \le \dots \le y_{n-1} \le x} \sum_{i=1}^{n-1} f_{i + 1}(y_i) - f_{i}(y_i).
\end{equation}
As in Example \ref{Ex:cars}, this is essentially immediate from the definition \eqref{E:Pitman} and an induction since $\la(f)_1 < \la(f)_i$ for all $i > 1$. 

We move to the general case. Fix $t\le x_1$. By Lemma \ref{L:two-line-isometry} we have that
	$$
	f[(t, [2, n])\to (\bx,1)]=	\cP_\sig f[(t, [2, n])\to (\bx,1)].
	$$
Now, since $\la(f)_1 < \la(f)_i$ for all $i \ge 1$, then there exists $t_0$ such that the function
$$
t \mapsto f[(t, [2, n])\to (\bx,1)] + \sum_{i=2}^n f_i(t)
$$
is constant for $t \le t_0$, and hence equals the left-hand side of \eqref{E:Pitman-unnormal}.
On the other hand, since $\la(f)_1 < \la(f)_i$ for $i > 1$ we have $\la(\cP_{\sig} f)_n = \la(f)_1 < \la(\cP_{\sig} f)_i = \la(f)_{i+1}$ for $i < n$. Therefore for all $t$ small enough we have
\begin{align*}
\cP_{\sig} &f[(t, [2, n])\to (x,1)] + \sum_{i=2}^n \cP_{\sig} f_i(t) \\
&= \cP_{\sig} f[(x_1, [1, n-1])\to (\bx,1)] + \sup_{y_1 \le \dots \le y_{n-1} \le x_1} \sum_{i=1}^{n-1} \cP_\sig f_{i+1}(y_i) - \cP_\sig f_{i}(y_i) \\
&= \cP_{\sig} f[(-\infty, [1, n-1])\to (\bx,1)] + \sum_{i=1}^{n-1} \cP_\sig f_i(x_1) \\&\;\;\;+ \sup_{y_1 \le \dots \le y_{n-1} \le x_1} \sum_{i=1}^{n-1} \cP_\sig f_{i+1}(y_i) - \cP_\sig f_{i}(y_i)
\end{align*}
Call the latter two terms on the right-hand side above $G(x_1)$.
Combining all of the above displays gives that for all small enough $t$ we have the equality
$$
f[(-\infty, [2, n])\to (\bx,1)]= \cP_{\sig} f[(-\infty,[1, n-1])\to (\bx,1)] + G(x_1) + \sum_{i=2}^n f_i(t) - \cP_\sig f_i(t).
$$
Therefore 
\begin{align*}
&f[(-\infty, [2, n])\to (\bx,1)] - f[(-\infty, [2, n])\to (x^{n-1},1)] \\
= &\cP_{\sig} f[(-\infty,[1, n-1])\to (\bx,1)] - \cP_{\sig} f[(-\infty,[1, n-1])\to (x^{n-1},1)].
\end{align*}
On the other hand, \eqref{E:basic-xx} ensures that the second terms on both the right- and left-hand sides above are equal, and hence so are the first terms, yielding \eqref{E:Pitman-unnormal}.
\end{proof}

\begin{proof}[Proof of Proposition \ref{P:orbits}]
First, we can construct $\bm$ by recursively constructing $\xi \in \Xi_g$. Indeed, for every $i \in \{1, \dots, n\}$, we always let $\xi(i)$ be the minimal index $j \in \{1, \dots, n\} \smin \{\xi(1), \dots, \xi(i-1)\}$ with $\la(f)_j = \la(g)_i$. Since $\la(g)$ is a permutation of $\la(f)$, this process results in a permutation $\xi$. Moreover, the use of minimal indices in the construction implies that resulting vector $\bm$ satisfies \eqref{E:N-property} for $\la(f)$. Uniqueness of $\bm$ follows since $\la(f) \in \R^n_\le$.

To prove \eqref{E:g1gk}, first assume that $m_k > k$. Let $\ell \in \{1, \dots, m_k - 1\}$ be the maximal index with $\ell \notin  \{m_1, \dots, m_{k-1}\}$. Let $\bm_2=(\ell+1,\ell+2,\ldots, m_k)$ so we can write $\bm = (\bm_1, \bm_2)$. Since $\bm$ satisfies \eqref{E:N-property}, both $\bm_1, \bm_2$ also satisfy \eqref{E:N-property} and so we may define last passage from $(-\infty, \bm_1), (-\infty, \bm_2)$. Next, for $\by \in \R^{m_k-\ell}_\le$, define 
\begin{align*}
G(f, \by, \bm_1) := \lim_{t \to -\infty} f[(t^{|\bm_1|}, \by ; \bm_1, (\ell-1)^{m_k-\ell}) \to (\bx, 1)] + \sum_{i \in \bm_1} f_i(t).
\end{align*}
Using this kind of hybrid last passage value, we can write down the following metric composition law:
\begin{align*}
f&[(-\infty, \bm)\to (\bx, 1)] = \sup_{\by \in \R^{m - \ell}_\le}f[(-\infty,\bm_2)\to (\by, \ell)]+G(f, \by, \bm_1).
\end{align*}
Now, the condition \eqref{E:N-property} implies that $\la(f)_\ell < \la(f)_{j}$ for all $j > \ell$. Therefore by Lemma \ref{L:block-up} we may write the above as 
\begin{align*}
\sup_{\by \in \R^{m - \ell}_\le}f'&[(-\infty,\bm_2')\to (\by, \ell)]+G(f', \by, \bm_1) =  f'[(-\infty,\bm')\to (\bx, 1)],
\end{align*}
where $f' =\btau_{\ell}\cdots\btau_{m-1}f$, $\bm_2'=\{\ell,\ldots,m_k-1\}$ and $\bm'=(\bm_1, \bm_2')$. Note that $G(f', \by, \bm_1)$ involves lines $1,\ldots, \ell-1$ only, which are equal in $f$ and $f'$. Now, observe that our construction gives that $\la(f')_{m_i'} = \la(f)_{m_i}$ for all $i$, and that in the new environment $f'$, the vector $\bm'$  satisfies \eqref{E:N-property} for $\la(f')$. Therefore if $m_k-1 \ne k$ we can repeat the above argument with $\bm', f'$ in place of $\bm, f$ to get $\bm'', f''$ with $m''_{k} = m_k- 2$, $\la(f'')_{m_i''} = \la(f)_{m_i}$ for all $i$, and $\bm''$ satisfying \eqref{E:N-property} for $\la(f'')$. Continuing in this way, we end at an environment $h$ satisfying 
$$
f[(-\infty ,\bm)\to (\bx, 1)] =h[(-\infty, (1, \dots, k))\to (\bx, 1)] = h[(x_1^k, k)\to (\bx, 1)] + \sum_{i=1}^k h_i(x_1)
$$
and satisfying $\la(h)_{i} = \la(f)_{m_i}$ for all $i$. If $m_k =m$, we can avoid the above argument entirely and simply let $h = f$. 

Now, there is a sorting operator $\sig \in \mathfrak{D}_n$ which can be written as a product of $\tau_i, \bar \tau_i, i \le k-1$ such that $\sig \la(g)_j = \la(h)_j$ for all $j \le k$. Therefore by Corollary \ref{C:samestart} we have that $\cP_\sig g_j = h_j$ for all $j \le k$, and so
\begin{align*}
h[(x_1^k, k)\to (\bx, 1)] + \sum_{i=1}^k h_i(x_1)
&= g[(x_1^k, k)\to (\bx, 1)] + \sum_{i=1}^k g_i(x_1) \\
 &= g[(-\infty, (1, \dots, k)) \to (\bx, 1)].
\end{align*}
Here the first equality uses that Pitman transforms preserve sums and are boundary isometries (Lemma \ref{L:two-line-isometry}). Putting this together with the previous display yields the result.
\end{proof}

The following corollary of Proposition \ref{P:orbits} gives a clean formula for inverting the full sort, where we completely reverse the order of the slopes. This corollary can be viewed as describing the RSK correspondence in this setting. 

Let $\cC^n_\le(\R), \cC^n_\ge(\R)$ be the subsets of $\cC^n(\R)$ such that if $f \in \cC^n_\le(\R)$ then $\la(f) \in \R^n_\le$ and if $f \in \cC^n_\ge(\R)$ then $\la(f) \in \R^n_\ge$. By Lemma \ref{L:orbits}, for $f \in \cC^n_\le(\R)$ there is a unique element $\cP_{\operatorname{RSK}} f$ in the $\cP$-orbit of $f$ such that $\cP_{\operatorname{RSK}} f \in \cC^n_\ge(\R)$ and given $g \in \cC^n_\ge(\R)$, there is a unique element $\cQ_{\operatorname{RSK}} g$ in its $\cP$-orbit contained in $\cC^n_\ge(\R)$. Lemma \ref{L:orbits}.1 guarantees that these maps are inverses, so we have defined a bijection
$$
\cP_{\operatorname{RSK}}: \cC^n_\le(\R) \to \cC^n_\ge(\R)
$$
with inverse $\cQ_{\operatorname{RSK}}$. The next corollary gives a simple global description of these maps without appealing to iterated Pitman transforms. This is the analogue of Greene's theorem in the present setting.

\begin{corollary}
\label{C:RSK}
Let $f \in \cC^n_\le(\R)$, let $g = \cP_{\operatorname{RSK}} f$, and set $Rg(x) = g(-x)$. For every $k \in \{1, \dots, n\}$ let $\bn_k  \in \{1, \dots, n\}^k_<$ denote the unique maximal vector (in the coordinatewise order) satisfying \eqref{E:N-property} for $\la(f)$. Similarly, let  $\bm_k \in \{1, \dots, n\}^k_<$ denote the unique maximal vector (in the coordinatewise order) satisfying \eqref{E:N-property} for $- \la(f) = (-\la(f)_k, \dots, -\la(f)_1)$. Then:
\begin{align*}
f[(-\infty, \bn_k) \to (x^k, 1)] &= g_1(x) + \dots + g_k(x), \qquad \text{ and } \\
Rg[(-\infty, \bm_k) \to (-x^k, 1)] &= f_1(x) + \dots + f_k(x).
\end{align*}
\end{corollary}

\begin{proof}
The recovery formulas for $g$ from $f$ are both special cases of Proposition \ref{P:orbits}. To recover $f$ from $g$ , first let $\sig$ be such that $\cP_\sig g = f$. Now, from the formulas $\cP_{\bar \tau_i} = R \cP_{\tau_i} R$ and the identity $R^2 = \id$, we have that $Rf = \cP_{\bar \sig} Rg$, where $\bar \sig$ is the given by taking the word $\sig$ and mapping $\tau_i \mapsto \bar \tau_i$ and $\bar \tau_i \mapsto \tau_i$ everywhere. Moreover, $Rg \in \cC^n_\le(\R)$ so we can apply Proposition \ref{P:orbits} to recover $Rf$ from $Rg$. Applying $R$ to $Rf$ then yields the formula above. 
\end{proof}

A different approach to the bijection in Corollary \ref{C:RSK} was developed in Sorensen's Ph.D. thesis using queuing maps rather than multi-path last passage, e.g. see Section 2.3.3 and Lemma 2.3.18 in \cite{sorensen2023stationary}. Another perspective on Corollary \ref{C:RSK} is in terms of the Sch\"utzenberger involution. Indeed, we can write $\cQ_{\operatorname{RSK}} = R \cP_{\operatorname{RSK}} R$, where $R f(x) = f(-x)$. From this point of view, the fact that $\cP_{\operatorname{RSK}}$ inverts $\cQ_{\operatorname{RSK}}$ on $\cC^n_\le(\R)$ says that $\cP_{\operatorname{RSK}} R$ is an involution when restricted to $\cC^n_\ge(\R)$. If we consider $\cC^n_\ge(\R) \subset \cC^n(\R)$ as the analogue of the set of Young tableaux, we can understand $\cP_{\operatorname{RSK}} R$ as the Sch\"utzenberger involution in the present setting. See \cite[Appendix A]{fulton1997young} for discussion of the classical Sch\"utzenberger involution and \cite{bisi2021geometric} for a comprehensive modern account of the connection between the Sch\"utzenberger involution, last passage percolation, and directed polymer partition functions.

\subsection{Burke theorems}
\label{S:burke-theorems}

Pitman transforms behave well with Brownian inputs. At the level of the two-line Pitman transform, this is the well-known \textbf{Brownian Burke property}, which has been previously used in the study of LPP in \cite{o2001brownian, seppalainen2023busemann, seppalainen2023global}. For this theorem and throughout the remainder of the paper, for a vector $\la \in \R^n$ we let $\mu_\la$ denote the law of $t \mapsto B(t) + \la t$, where $B \in \cC^n_0(\R)$ is $k$-tuple of independent, two-sided Brownian motions of variance $2$.

\begin{theorem}[Brownian Burke property]
	\label{T:Brownian-Burke}
	Let $B = (B_1, B_2) \sim \mu_{(\la_1, \la_2)}$ for some $\la_2 > \la_1$. Then
$\cP B - \cP B(0) \sim \mu_{(\la_2, \la_1)}$.
\end{theorem}

\begin{proof}
	This is part of \cite[Theorem 4]{o2001brownian}. We translate the language in order to assist the reader navigating between that paper and ours. It suffices to prove the theorem when $\la_1 = 0$, since the Pitman transform $\cP$ commutes with a common linear shift of both functions (Lemma \ref{L:shift-commute}).
	
	In the language of that paper, O'Connell and Yor show the following. Let $B', C:\R \to \R$ be two-sided, independent standard Brownian motions, and let $m > 0$. Define functions $q, e, d:\R \to \R$ as follows:
	\begin{align*}
		q(t) &= B_t' + C_t - mt + \sup_{s \le t} (ms - C_s - B_s), \\
		d(t) &= B_t' + q(0) - q(t),\\
		e(t)&= C_t + q(0) - q(t).
	\end{align*}
	Theorem 4 in \cite{o2001brownian} states $d, e$ are independent standard Brownian motions and that $(d, e)|_{(-\infty, t]}$ is independent of $q|_{[t, \infty)}$ for any $t \in \R$. We can rewrite their result in terms of Pitman transforms. Define $B_1(t) = \sqrt{2} B_t', B_2(t) = \sqrt{2}(mt - C_t)$ and observe that if $m = \la_2/\sqrt{2}$, then $B = (B_1, B_2) \sim \mu_{(0, \la_2)}$. Then
	\begin{align*}
		\sqrt{2} q(t) &= \cP B_1(t) - B_2(t), \\ \qquad \sqrt{2} d(t) &= B_1(t) + B_2(t) + \cP B_1(0) - \cP B_1(t) = \cP B_2(t) - \cP B_2(0).\\
		\sqrt{2} (mt - e(t)) &= B_2(t) + \cP B_1(t) - B_2(t) - \cP B_1(0) = \cP B_1(t) - \cP B_1(0)
	\end{align*}
	The claim that $d, e$ are independent standard Brownian motions gives that $(\sqrt{2}(mt - e), \sqrt{2}d) = \cP B - \cP B(0) \sim \mu_{(\la_2, 0)}$. 
\end{proof}

The following corollary extends the Brownian Burke property to the biHecke monoid. For this corollary and in the remainder of the paper, we write $f^0 := f - f(0)$ for $f \in \cC^I(\R)$, and write $\cP^0_\sig f = (\cP_\sig f)^0$.
\begin{corollary}
	\label{C:Brownian-sorting}
	Let $B \sim \mu_{\la}$ for some vector $\la \in \R^n$. Then for any $\sigma \in \mathfrak{D}_n$ we have
	\begin{equation}
		\label{E:extended-Burke}
		\cP_\sig^0 B \sim \mu_{\sig \la},
	\end{equation}
	where $\sig$ acts on $\la$ by the basic action.
\end{corollary}

\begin{proof}
	First, if $\sig = \tau_i$ for some $i = 1, \dots, n-1$ and $\la_i < \la_{i+1}$, then \eqref{E:extended-Burke} holds by Theorem \ref{T:Brownian-Burke}. If $\la_i \ge \la_{i+1}$, then \eqref{E:extended-Burke} holds since $\cP_{\tau_i} B = B$ and $\tau_i \la = \la$. The case of general $\sig$ follows by induction.
\end{proof}

We end this section with a simple consequence of Corollary \ref{C:Brownian-sorting} and Proposition \ref{P:orbits}. This corollary can be viewed as implying the existence of stationary measures for Brownian LPP.

\begin{corollary}
	\label{C:existence-of-stationary}
	Consider vectors $\la \in \R^n, \eta \in \R^m$ and suppose that $\la_i > \eta_j$ for all $i \in \{1, \dots, n\}, j \in \{1, \dots, m\}$. Let $W = (B', B) \sim \mu_{(\eta ,\la)}$. Then for $\bn \in \{1, \dots,n\}^k_<, \bx \in \R^k_\le$, we have
	\begin{align*}
		W[(-\infty, \bn + m) \to (\bx, 1)] - W[(-\infty, \bn + m) \to (0^k, 1)]   \eqd
		B[(-\infty, \bn) \to (\bx, 1)],
	\end{align*}
	
	where here $\bn + m = (n_1 + m, \dots, n_k+m)$, and the equality in distribution is joint in all $\bn, \bx$.
\end{corollary}

\begin{proof}
	Let $\sig \in \mathfrak{D}_{n + m}$ be such that $\sig (\nu, \la) = (\la, \nu)$ under the basic $\R$-action. Then by Proposition \ref{P:orbits}, we have
	$$
	W[(-\infty, \bn + m) \to (\bx, 1)] = \cP_\sig W[(-\infty, \bn) \to (\bx, 1)]
	$$
	for all $\bx, \bn$ for which either side is defined. On the other hand $\cP_\sig^0 W \eqd (B, B')$ by Corollary \ref{C:Brownian-sorting}, yielding the result.
\end{proof}

\section{RSK on \texorpdfstring{$\cC^\N(\R)$}{}: From Pitman to Busemann}
\label{S:Pitman-Brownian}

In this section we study limits of Pitman transforms on $\cC^n(\R)$ as $n \to \infty$. As discussed in the introduction, this is natural when we work with a line environment of independent Brownian motions. The resulting limits are expressed in terms of Busemann functions. Because of this, in taking this limit we will simultaneously build up a theory of multi-path Busemann functions in Brownian LPP. Throughout this section, we let $B = (B_i : i \in \Z)$ be an environment of independent $2$-sided Brownian motions. Also, recall from Section \ref{SS:RSK-2} the definition of multi-path Busemann functions and semi-infinite geodesics and optimizers.

\subsection{Multi-path Busemann functions}
\label{S:Busemann-A}

To study multi-path Busemann functions, we first need to understand single-path Busemann functions. The theory in this setting was built by Sepp\"al\"ainen and Sorensen \cite{seppalainen2023busemann, seppalainen2023global}, and we use their results as a starting point. 

\begin{theorem}
	\label{T:sepp-sorensen}
	In the environment $B$, define the centered Busemann function ending at $(x, n) \in \R \X \Z$ in direction $\theta \in [0, \infty)$ by
	\begin{equation}
		\label{E:buseman-brownian-def}
	\cB^\theta(x, n) = \lim_{t \to -\infty} B[(t, \fl{\theta t + n}) \to (x, n)] -  B[( t, \fl{\theta t + n}) \to (0, 1)].
\end{equation}
	Then there exists a random countable set $\Theta \subset (0, \infty)$ such that the following claims hold on an almost sure set $\Om$:
	\begin{enumerate}
		\item (Busemann existence, \cite[Theorem 2.5(i)]{seppalainen2023global}, \cite[Theorem 3.5(viii)]{seppalainen2023busemann}). For any $(x, n) \in \R \X \Z$ and $\theta \in (0, \infty) \setminus \Theta$ the above limit exists and is equal to
		$$
		\lim_{t \to -\infty} B[(t, \pi(t)) \to (x, n)] -  B[(t, \pi(t)) \to (0, 1)]
		$$
		for any function $\pi:(-\infty, x \wedge 0] \to \R$ such that $\pi(t)/|t| \to \theta$ as $t \to -\infty$. When $\theta = 0, n = 1$, the above limit exists and $\cB^\theta(x, 1) = B_1(x)$.
		\item (Metric composition, \cite[Theorem 3.5(vi)]{seppalainen2023busemann}). We have the following metric composition law for Busemann functions. For any $n > m \in \Z$, $x \in \R$, and direction $\theta \in (0, \infty) \smin \Theta$ we have
		$$
		\cB^\theta(x, m) = \max_{z \le x} \cB^\theta(z, n) + B[(z, n - 1) \to (x, 1)].
		$$ 
		\item (Geodesic existence, \cite[Theorem 3.1(i), Theorem 4.3(ii)]{seppalainen2023busemann}). There exist semi-infinite geodesics ending at every point $(x, n) \in \R \X \Z$ in every direction $\theta \in (0,\infty)$. Moreover, for every $\theta \in (0, \infty)$, $(x, n) \in \R \X \Z$ there exist \textbf{leftmost} and \textbf{rightmost} geodesics $\pi^L, \pi^R$ from $\theta$ to $(x, n)$ in the sense that for every geodesic $\tau$ in direction $\theta$ ending at $(x, n)$ we have $\pi^L \le \tau \le \pi^R$. 
		\item (Geodesic monotonicity, \cite[Theorem 4.3(iii)]{seppalainen2023busemann}). For $\theta_1 \le \theta_2, n \in \N$ and $x_1 \le x_1$, if $\pi^L_i, i = 1, 2$ is the leftmost geodesic from $\theta_i$ to $(x_i, n)$, then $\pi^L_1 \le \pi^L_2$. An identical statement holds for rightmost geodesics.
		\item (Geodesic coalescence, \cite[Theorem 2.8(i)]{seppalainen2023global}). For $\theta \in (0, \infty) \setminus \Theta$, any two semi-infinite geodesics $\pi, \tau$ in direction $\theta$ eventually coalesce, i.e. $\pi(t) = \tau(t)$ for all small enough $t$. Moreover, there is a unique semi-infinite geodesic in direction $\theta$ and ending at $(x, n)$ for every $x \in \Q, n \in \Z$.
		\item ($0$-directed geodesics, \cite[Theorem 4.3(iv, vi)]{seppalainen2023busemann}). For every $(x, n) \in \R \X \Z$ the unique semi-infinite geodesic in direction $0$ ending at $(x, n)$ is the constant path $\pi:(-\infty, x] \to \R, \pi \equiv n$. 
	\end{enumerate}
Moreover, by \cite[Theorem 2.5]{seppalainen2023global}, $\P(\theta \in \Theta) = 0$ for all $\theta > 0$. Therefore we may assume $\Q \cap \Theta = \emptyset$ on $\Om$. 
\end{theorem}
We will write $\cB^\theta(x, n; \cB)$ in place of $\cB^\theta(x, )$ if the environment is not clear from context and let $\cB^\theta(x) := \cB^\theta(x, 1)$. Moving forward, we will write $\cI:= [0, \infty) \smin \Theta$.

\begin{remark}
\label{R:sepp-notation}
In the above theorem, we have aimed to follow the notation of \cite{seppalainen2023busemann, seppalainen2023global} but have made a few changes in order to better integrate it with our present setup. Our Busemann functions follow semi-infinite paths starting at $t = -\infty$, whereas the semi-infinite paths in \cite{seppalainen2023busemann, seppalainen2023global} end at $t = \infty$ instead. We also use Brownian motions of variance $2$ instead of $1$. A more important difference is that our Busemann function in direction $\theta$ is their Busemann function in direction $1/\theta$. This choice is more natural for our eventual study of RSK in the directed landscape. Finally, the versions of Theorem \ref{T:sepp-sorensen}.2  and Theorem \ref{T:sepp-sorensen}.5 stated in \cite{seppalainen2023global} give queuing representations equivalent to our descriptions.
\end{remark}

We now extend the definition of Busemann functions to multiple \textit{distinct} directions. Later we will use a continuity argument to allow for repeated directions. For this definition, for $\theta \in \cI^k_\le$, write $\pi^{\theta} = (\pi^{\theta_1}, \dots, \pi^{\theta_k})$, where $\pi^{\theta_i}$ is the unique geodesic in direction $\theta_i$ ending at $(0, 1)$.

\begin{proposition}
	\label{P:extended-usemann}
	For $\theta \in \cI^k_<, \bx \in \R^k_\le$, define
	\begin{equation}
		\label{E:extended-busemann}
		\cB^\theta(\bx) = \lim_{t \to -\infty}  B[(t, \fl{\theta |t| + k}) \to (\bx, 1)] -  \sum_{i=1}^k B[(t, \fl{\theta_i |t| + k}) \to (0, 1)].
	\end{equation}
	Then the following claims hold on an almost sure set $\Om'$.
	\begin{enumerate}
		\item The limit \eqref{E:extended-busemann} exists for all $k \in \N, \theta \in \cI^k_<, \bx \in \R^k_\le$. Moreover, for any $k$-tuple of semi-infinite paths $\pi = (\pi_1, \dots, \pi_k)$ in direction $\theta$, we have
		\begin{equation}
			\label{E:B-t-la}
			\cB^\theta(\bx) = \lim_{t \to -\infty} B[(t, \pi(t)) \to (\bx, 1)] -  \sum_{i=1}^k B[(t, \pi_i(t)) \to (0, 1)].
		\end{equation}
		\item There exists leftmost and rightmost semi-infinite optimizers in every direction $\theta\in [0, \infty)^k_\le$ ending at every point $(\bx, 1), \bx \in \R^k_\le$. Moreover, if we let $\tau^{\bx, \theta}$ denote the leftmost semi-infinite optimizer in direction $\theta$ ending at $(\bx, 1)$, then $\tau^{\bx, \theta} \le \tau^{\bx', \theta'}$ whenever $\bx \le \bx', \theta \le \theta'$. Monotonicity similarly holds for rightmost optimizers.
\item Any semi-infinite optimizer $\tau$ in a direction $\theta\in \cI^k_<$ eventually coalesces with the $k$-tuple of paths $\pi^\theta$ in the sense that $\tau|_{(-\infty, t]} = \pi^\theta|_{(-\infty, t]}$ for small enough $t$.
	\end{enumerate}
\end{proposition}

To prove Proposition \ref{P:extended-usemann} we require two simple lemmas. For the first lemma, we let $\pi^{\theta, y}$ denote the rightmost geodesic in direction $\theta$, ending at $(y, 1)$.

\begin{lemma}
	\label{L:disjoint-geods}
	The following event holds a.s. For every $k \in \N, a > 0$ and every compact box $K = \prod_{i=1}^k [\theta_i^-, \theta_i^+]\subset \R^k_<$ there exist vectors $\by, \bz \in \Q^k_\le$ with $\by \le -a^k \le a^k \le \bz$ such that for any $\theta \in K$, the $k$-tuples of semi-infinite rightmost geodesics $\pi^{\theta, \by} := (\pi^{\theta, y_1}, \dots, \pi^{\theta, y_k})$ and $\pi^{\theta, \bz} := (\pi^{\theta, z_1}, \dots, \pi^{\theta, z_k})$ are both disjoint $k$-tuples. In particular, they are both  disjoint optimizers.
\end{lemma}

\begin{proof}
By induction it is enough to prove the result when $k = 2$. Moreover, by countable additivity it is enough to prove that the result for every fixed $\theta_i^-, \theta_i^+ \in \Q$ a.s.\ and it is enough to prove the existence of $\by$ since the existence of $\bz$ follows by symmetry. Finally, by monotonicity of geodesics (Theorem \ref{T:sepp-sorensen}.3), $\pi^{\theta, \by}$ is a disjoint $k$-tuple for all $\theta \in K$ if and only if $\pi^{(\theta_1^+, \theta_2^-), \by}$ is. We will prove this latter statement.
	
	For every fixed $\theta \ge 0$, the process $x \mapsto \pi^{\theta, x}(x + \cdot)$ is stationary, since the same holds for the background environment $B - B(x)$. Therefore for any pair $\theta_1 < \theta_2 \in [0, \infty)$, any $w \in \R$, and any function $f:(-\infty, w] \to \R$ with $f(t)/|t| \to \theta_2$ as $t \to \infty$ a.s.\ there exists a random integer $Y < w$ such that $\pi^{\theta_1, Y}(u) < f(u)$ for $u \le Y$. Setting $f = \pi^{\theta_2^-, -a}$ and $\by = (Y, - a)$, the lemma follows.
\end{proof}

\begin{lemma}
	\label{L:optimal-coal}
	The following holds almost surely. For every $k \ge 1, \theta \in \cI^k_<$ and $a > 0$, there exists $\ep > 0$, $\by \le -a^k < a^k \le \bz$, and $T < y_1$ such that whenever $\|\theta_1 - \theta\|_\infty < \ep, \|\theta_2 - \theta\|_\infty < \ep$, we have:
	\begin{itemize}
		\item For any $i = 1, \dots, k$, any geodesics from $\theta_{1, i}$ to $y_i$ and $\theta_{2, i}$ to $z_i$ are equal at time $T$.
		\item Any geodesics from $\theta_{1, i}$ to $y_i, i = 1, \dots, k$ are mutually disjoint. Similarly, any geodesics from $\theta_{2, i}$ to $z_i$, $i=1, \dots, k$, are  mutually disjoint.
	\end{itemize}
\end{lemma}

For the proof of the lemma, it will be convenient to use the \textbf{Hausdorff topology} on finite or semi-infinite paths. We say that a sequence of (nonincreasing cadlag) paths $\pi_n:[a_n, b_n] \to \N$ converges in the Hausdorff topology to a limiting path $\pi:[a, b] \to \R$ for $a < b$ if the graphs $\mathfrak g \pi_n := \{(r, \pi(r)) : r \in [a_n, b_n]\}$ converge in the Hausdorff topology on closed sets to the graph $\mathfrak g \pi$. We define Hausdorff convergence on semi-infinite paths by asking for Hausdorff convergence when restricted to every compact interval.
Path length is continuous and the disjointness required in the definition of disjoint $k$-tuples is a closed property in this topology, making it convenient to work with when studying LPP. See \cite[Section 2.3]{dauvergne2022rsk} for more discussion.

\begin{proof}
First, fix a compact set  $K = \prod_{i=1}^k [\theta_i^-, \theta_i^+]\subset \R^k_<$ containing $\theta$ in its interior, and choose $\ep_0 > 0$ so that if $\|\theta' - \theta\|_\infty < \ep_0$ then $\theta' \in K$. Applying Lemma \ref{L:disjoint-geods} with $K$ and $a > 0$ as in the statement of the lemma implies that the second bullet point above holds as long as we choose $\ep < \ep_0$. It remains to prove that we can choose $\ep$ small enough so that the first bullet holds. For this, fix a coordinate $\theta_i$. Consider sequences $\pi_{i,n}^\pm, n \in \N$ of rightmost geodesics (for $\pi_{i, n}^+$) in direction $\theta_i \pm 1/n$ ending at $(y_i, 1)$ (for $\pi_{i, n}^-$) and $(z_i, 1)$ (for $\pi_{i, n}^+$). By monotonicity of geodesics, we have that $\pi_{i, 1}^- \le \pi_{i, 2}^- \le \cdots \le \pi_{i, 2}^+ \le \pi_{i, 1}^+$. This implies that both of the sequence $\pi_{i, n}^\pm, n \in \N$ converge in the Hausdorff topology, and that the limits $\pi^+, \pi^-$ must be geodesics in direction $\theta$ ending at $(y_i, 1)$. In particular, by Theorem \ref{T:sepp-sorensen}.5, these geodesics must be equal on some interval $(-\infty, t_i]$. Hausdorff convergence of $\pi_{i, n}^\pm$ then implies that for any $t < t_i$ we have $\pi_{i, n}^-|_{[t, t_i]} = \pi_{i, n}^+|_{[t, t_i]}$ for all large enough $n$. Choosing $T < \min_i (t_i)$, we get that for large enough $n$, we have $\pi_{i, n}^-(T) = \pi_{i, n}^+(T)$ for all $i = 1, \dots, k$. Applying monotonicity of geodesics then yields the first bullet point.
\end{proof}
	
\begin{proof}[Proof of Proposition \ref{P:extended-usemann}]
	Let $\Om'$ be the intersection of the almost sure events in Theorem \ref{T:sepp-sorensen} and Lemma \ref{L:optimal-coal}, and let $\pi$ be as in part $1$ of the proposition. Fix $a > 0$, and let $\by, \bz, \ep, T$ be as in Lemma \ref{L:optimal-coal}. Then for $t < T$, if $\theta^1 < \theta < \theta^2$ satisfy $\|\theta^i - \theta\|_\infty < \ep$ then the optimizers $(t, \pi^{\theta^1}(t))$ to $(\by, 1)$ and from $(t, \pi^{\theta^2}(t))$ to $(\bz, 1)$ consist of $k$ geodesics, and all these geodesics agree at the common time $T$. Therefore for all large enough $t$, by monotonicity of optimizers, all optimizers from $(t, \pi(t))$ to a point $\bx$ with $\by \le \bx \le \bz$ agree at time $t$, and so;
	$$
	B[(t, \pi(t)) \to (\bx, 1)]  = B[(t, \pi(t)) \to (T, \pi^{\theta}(T))] + B[(T, \pi^{\theta}(T)) \to (\bx, 1)],
	$$
	and similarly, for all $i = 1, \dots, k$ we have
	$$
	B[(t, \pi_i(t)) \to (x_i, 1)]  = B[(t, \pi_i(t)) \to (T, \pi^{\theta}_i(T))] + B[(T, \pi^{\theta}_i(T)) \to (\bx, 1)],
	$$
	To complete the proof of part $1$, we claim that for large enough $t$ we have:
	$$
	\sum_{i=1}^k B[(t, \pi_i(t)) \to (T, \pi^{\theta}_i(T))] = B[(t, \pi(t)) \to (T, \pi^{\theta}(T))].
	$$
	Indeed, by the construction in Lemma \ref{L:optimal-coal}, for any $i \in \{1, \dots, k\}$, the geodesics from  $(\pi^{\theta_2}_i(t), t)$ to $(y_i, 1)$ and from $(\pi^{\theta_1}_{i+1}(t), t)$ to $(z_{i+1}, 1)$ are disjoint. These geodesics go through the points $(T, \pi^{\theta}_i(T)), (T, \pi^{\theta}_{i+1}(T))$. Therefore by monotonicity of geodesics, for any $i \in \{1, \dots, k\}$ the two geodesics from $(t, \pi_i(t))$ to $(T, \pi^{\theta}_i(T))$ and from $(t, \pi_{i+1}(t))$ to $(T, \pi^{\theta}_{i+1}(T))$ are also disjoint as long as $t$ is large enough so that
	$$
	\pi_i(t) \le \pi^{\theta_2}_i(t) < \pi^{\theta_1}_{i+1} \le \pi_{i+1}(t).
	$$
Therefore for large enough $t$, the $k$ geodesics from $(t, \pi(t))$ to $(T, \pi^{\theta}(T))$ are all disjoint, yielding the above equality and proving part $1$. 

We move on to existence of optimizers. We first consider the case when $\theta \in \cI^k_<$. Let all notation be as in the proof of part $1$.  In this case, if $\by \le \bz$ are as in part $1$, then the $k$-tuples of semi-infinite geodesics to $(\by, 1)$ and $(\bz, 1)$ are (unique) optimizers, and these agree on the interval $(-\infty, T]$ with $\pi^\theta$. By monotonicity of optimizers, for $\by \le \bx \le \bz$, if we then concatenate $\pi^\theta|_{(-\infty, T]}$ to an optimizer from $(T, \pi(T))$ to $(\bx, 1)$, this must be a semi-infinite optimizer in direction $\theta$ to $(\bx, 1)$, and all optimizers must be of this form. Existence of rightmost and leftmost optimizers for $\theta \in \cI^k_<$ then follows from existence of rightmost and leftmost optimizers in the finite setting.
	
	Now let $\theta \in [0, \infty)^k_\le, \bx \in \R^k_\le$, and let $\theta^n \in \cI^k_<$ be such that $\theta^n \cvgdown \theta$. Let $\pi^n$ be an optimizer to $(\bx, 1)$ in direction $\theta^n$. Then by monotonicity of optimizers, we have $\pi^1 \ge \pi^2 \ge \cdots$ and so $\pi^n$ has a limit $\pi$ in the Hausdorff topology satisfying
	$$
	\limsup_{t \to \infty} \frac{\pi(t)}{|t|} \le \lim_{n \to \infty} \lim_{t \to \infty}  \frac{\pi^n(t)}{|t|} = \theta.
	$$
	Since being an optimizer is a closed property in the Hausdorff topology, $\pi$ is also an optimizer. We check that $\pi$ has direction $\theta$. Let $\theta = (0^\ell, \theta_{\ell+ 1}, \dots, \theta_k)$, where $\theta_{\ell+ 1} > 0$. Let $\eta = (0^\ell, \eta_{\ell+1} < \dots < \eta_k)$ be such that $\eta_i \in \cI$ for all $i$ and $\eta \le \theta$. As in the proof of part $1$, we can find $y_i \le x_i, i = \ell + 1, \dots, k$ with $\by \in \Q^k_\le$ such that all the leftmost geodesics $\tau^\eta_i$ in direction $\eta_i$ to $(y_i, 1), i = \ell + 1, \dots, k$ are disjoint. Therefore we can find $y_\ell < y_{\ell + 1}$ so that defining $\tau^\eta_i:(-\infty, y_\ell] \to \N$ by $\tau^\eta_i(x) = i$ for all $i$, $\tau^\eta$ is a disjoint optimizer in direction $\eta$. Monotonicity of optimizers then guarantees that $\tau^\eta \le \pi^n$ for all $n$ and so $\tau^\eta \le \pi$. Hence
	$$
	\eta = \lim_{t \to \infty} \frac{\tau^\eta(t)}{|t|} \le \liminf_{t \to \infty} \frac{\pi(t)}{|t|},
	$$
	so combining the two displays and letting $\eta \uparrow \pi$ gives that $\pi$ is an optimizer in direction $\theta$. This completes the proof of existence of an optimizer for general $\theta$.
	
	To argue that there must exist rightmost and leftmost optimizers for general $\theta$, let $\rho, \rho'$ be two optimizers in direction $\theta$, ending at $(\bx, 1)$. By a standard argument (see the proof of Lemma 2.2, \cite{dauvergne2021disjoint}), the pointwise minima and maxima $\rho \wedge \rho'$ and $\rho \vee \rho'$ are also both optimizers, to the left and right of $\rho, \rho'$. Moreover, if $\rho_n$ is a sequence of semi-infinite disjoint optimizers with $\rho_1 \ge \rho_2 \ge \cdots$ then $\rho_n$ has a Hausdorff limit $\rho$, which is itself an optimizer to $(\bx, 1)$ in direction $\theta$. Indeed, the direction is preserved since all the $\rho_i$ must be bounded between $\tau^\eta, \eta \le \theta$ and $\pi^n, n \in \N$ for all $n$, and the optimizer property is preserved since path length is continuous and disjointness is a closed property in the Hausdorff topology. These two facts allow us to apply Zorn's lemma to the set of optimizers to $(\bx, 1)$ in direction $\theta$, yielding a leftmost optimizer (and similarly a rightmost one).
\end{proof}

Our next goal is to extend the process $\cB^\theta(\bx)$ defined in Proposition \ref{P:extended-usemann} to general $\theta$. We do this by continuity; later we will argue that doing so produces Busemann functions. To facilitate the proof, we need two short lemmas.

\begin{lemma}[Quadrangle inequality]
	\label{L:quadrangle-busemann}
	For $\bx\le \bx'\in \mathbb R^m_\ge$ and $\theta\le \theta'\in \cI^m_<$ we have
	$$
	\cB^{\theta}(\bx')+\cB^{\theta'}(\bx)\le 	\cB^{\theta}(\bx)+\cB^{\theta'}(\bx')
	$$
	with equality if there are semi-infinite optimizers $\pi, \pi'$ in directions $\theta, \theta'$ ending at $(\bx, 1), (\bx', 1)$ such that $\pi(T) = \pi'(T)$ for some $T \le x_1$.
	\end{lemma}
\begin{proof}
The quadrangle inequality is immediate from appealing to Proposition \ref{P:extended-usemann}.1 with $\pi = \pi^\theta$ for direction $\theta$ and $\pi = \pi^{\theta'}$ for direction $\theta'$, and then using the usual quadrangle inequality (Lemma \ref{L:quadrangle}). 
\end{proof}

\begin{lemma}
	\label{L:Blx-continuity}
	The following holds almost surely. For every $k \ge 1, \theta \in \cI^k_<$ and $a > 0$, there exists $\ep > 0$, such that if $\|\theta' - \theta\|_\infty < \ep$ with $\theta' \in \cI^k_<$ we have $\cB^\theta(\bx) - \cB^\theta(0^k) = \cB^{\theta'}(\bx) - \cB^\theta(0^k)$ whenever $\bx \in [-a, a]^k_\le$. 
\end{lemma}

\begin{proof}
Let $\by, \bz, \ep$ be as in Lemma \ref{L:optimal-coal}. Then by the `equality if' claim in Lemma \ref{L:quadrangle}, for $\theta_1 \le \theta_2$ such that $\|\theta - \theta_j\|_\infty < \ep, j = 1, 2$ we have
$$
	\cB^{\theta_1}(\by)+\cB^{\theta_2}(\bz) = 	\cB^{\theta_1}(\bz)+\cB^{\theta_2}(\by).
$$
The quadrangle inequality in Lemma \ref{L:quadrangle} forces this equality to hold when we replace $\by, \bz$ with any points in $[-a, a]^k$, yielding the result.
\end{proof}

We will move to the case of repeated endpoints by first extending the process by continuity.
When we extend to the case of repeated endpoints, we cannot normalize our Busemann functions by subtracting off $k$ single-path last passage values as in \eqref{E:extended-busemann}. Rather, we need a normalization that treats repeated endpoints together. First, define the \textbf{over-normalized Busemann function}
\begin{equation}
\label{E:overnormalized-Buse}
\hat \cB^\theta(\bx) = \lim_{\theta' \cvgdown \theta, \theta' \in \cI^k_<} \cB^{\theta'}(\bx) - \cB^{\theta'}(0^k).
\end{equation}
This Busemann function exists by a simple monotonicity argument, but loses some information about the process. We also define a minimally normalized Busemann function $\cB$ which has a more involved definition. First, for $\theta \in \R^k_\le$, let $\Pi(\theta) = (I_1(\theta), \dots, I_{m(\theta)}(\theta))$ be the partition of $\{1, \dots, k\}$ such that $\theta_i = \theta_j$ if and only if $i, j$ are in the same part of $\Pi(\theta)$.
Also, for $\bx \in \R^k_\le$ and $I = \{i_1, \dots, i_\ell\}\subset \{1, \dots, k\}$, write $\bx^I = (x_{i_1}, \dots, x_{i_\ell}) \in \R^\ell_\le$. We use the same notation when $I = (i_1, \dots, i_\ell)$ is a vector in $\{1, \dots, k\}^\ell_<$. Then for $\bx \in \R^k_\le$ and $\theta \in [0, \infty)^k_\le$ we can define the \textbf{Busemann function}
\begin{align}
	\label{E:B-theta-Pi-1}
	\cB^\theta(\bx) &= \lim_{\theta' \cvgdown \theta, \theta' \in \cI^k_<} \cB^{\theta'}(\bx) - \sum_{j=1}^{m(\theta)} \cB^{\theta^{\prime I_j(\theta)}}(0^{|I_j(\theta)|})
\end{align}
The strategy for proving that \eqref{E:B-theta-Pi-1} is well-defined will be to first work with the over-normalized function $\hat \cB$, and then compare normalizations. To do this comparison, we will need the following definition. Consider $\theta \in [0, \infty)^k_\le$. We say that $\bx \in \operatorname{DJ}(\theta)$ if there exists $\ep > 0$ such that if $|\theta' - \theta| \le \ep$ and $\pi^{\theta', i}$ is the leftmost optimizer ending at $(\bx^{I_i(\theta)}, 1)$ in direction $\theta^{\prime I_i(\theta)}$, then $(\pi^{\theta', 1}, \dots, \pi^{\theta', m(\theta)})$ is a (semi-infinite) disjoint $k$-tuple. We can argue exactly as in Lemma \ref{L:disjoint-geods} that almost surely, for all $\theta \in [0, \infty)^k_\le$ the set $\operatorname{DJ}(\theta)$ is non-empty. 	

\begin{proposition}
	\label{P:multi-Busemann-exist}
	Almost surely, the limits \eqref{E:overnormalized-Buse} and \eqref{E:B-theta-Pi-1} exist for every $\bx \in \R^k_\le$ and $\theta \in [0, \infty)^k_\le$. Moreover, for every $\bx \in\R^k_\le$ and every partition $\Pi$ of $\{1, \dots, k\}$, $\theta\mapsto \hat \cB^\theta(\bx)$ is a right-continuous function on $[0, \infty)^k_\le$. 
	
	The two Busemann functions are compatible in the following sense. Let $\Pi(\theta) = (I_1, \dots, I_\ell)$. Then almost surely, for any $\theta \in [0, \infty)^k_\le$ and any $\bz \in \operatorname{DJ}(\theta)$ we have
	\begin{equation}
	\label{E:theta-x}
	\cB^\theta(\bx) = \hat \cB^\theta(\bx) - \hat \cB^\theta(\bz) + \sum_{j=1}^\ell \hat \cB^{\theta^{I_j}}(\bz^{I_j}).
	\end{equation}
\end{proposition}

\begin{proof}
	For $\theta \in [0,\infty)^k_{\le}$ and $\bx \le \by$, define
	\begin{equation}
		\label{E:hat-B-theta}
		\hat \cB^\theta(\bx, \by)=\lim_{\theta'\downarrow \theta, \theta' \in \cI^k_<}\mathcal B^{\theta'}(\by)-\mathcal B^{\theta'}(\bx).
	\end{equation}
	By Lemma \ref{L:quadrangle-busemann}, the function $\theta \mapsto \cB^{\theta}(\bx) - \cB^{\theta}(\by)$ is nondecreasing in $\theta$, so the limit in \eqref{E:hat-B-theta} exists in $\R\cup \{-\infty\}$. To see that it cannot be $-\infty$, observe that by a quadrangle inequality (Lemma \ref{L:quadrangle}), we have the lower bound
	$$
	B[(-\infty, (1, \dots, k)) \to (\by, 1)] - B[(\infty, (1, \dots, k)) \to (\bx, 1)] \le B^{\theta}(\by)-\mathcal B^{\theta}(\bx)
	$$
	for all $\theta \in \cI^k_<$.
	To prove right continuity of $\hat \cB^\theta(\bx, \by)$ in $\theta$, given a sequence $\theta_m\downarrow\theta$, pick $\theta'_m \in [\theta_m,\theta_m+1/m] \cap \cI^k_<$ so that $|\cB^{\theta_m'}(\by)-\mathcal B^{\theta_m'}(\bx) - \hat \cB^{\theta_m}(\bx,\by) | \le 1/m$. By \eqref{E:hat-B-theta},
	$\cB^{\theta'_m}(\by)-\cB^{\theta'_m}(\bx)
	\to \hat \cB^{\theta}(\bx,\by),
	$ and so  $\hat \cB^{\theta_m}(\bx,\by)\to \hat \cB^{\theta}(\bx,\by)$. Finally, for $\bx \in \R^k_\le$ picking $\by \in \R^k_\le$ with $\by \le \bx$ and $\by\le 0^k$, we see that
	$$
	\hat \cB^{\theta}(\bx)=\hat \cB^{\theta}(\by,\bx)-\hat \cB^{\theta}(\by,0^k),
	$$
	which implies the existence of the limit \eqref{E:overnormalized-Buse}.
	
	We move to \eqref{E:theta-x} and the existence of the limit \eqref{E:B-theta-Pi-1}. For $\bz \in \operatorname{DJ}(\theta)$ and $\theta' \in \cI^k_<$ sufficiently close to $\theta$, we have that $\cB^{\theta'}(\bz) = \sum_{i=1}^\ell \cB^{\theta^{\prime I_i}}(z_i^{\ell_i})$. Therefore
	\begin{align*}
		\lim_{\theta' \cvgdown \theta, \theta' \in \cI^k_<} \cB^{\theta'}(0^k)& - \sum_{i=1}^m \cB^{\theta^{\prime I_i}}(0^{|I_i|}) \\
		&= \lim_{\theta' \cvgdown \theta, \theta' \in \cI^k_<} \cB^{\theta'}(0^k) - \cB^{\theta'}(\bz) + \lim_{\theta' \cvgdown \theta, \theta' \in \cI^k_<} \lf(\sum_{i=1}^\ell \cB^{\theta^{\prime I_i}}(z_i^{\ell_i}) - \sum_{i=1}^\ell \cB^{\theta^{\prime I_i}}(0^{|I_i|}) \rg) \\
		&=\hat \cB^\theta(\bx) - \hat \cB^\theta(\bz) + \sum_{j=1}^\ell \hat \cB^{\theta^{I_j}}(\bz^{I_j}).
	\end{align*}
This yields both \eqref{E:theta-x} and \eqref{E:B-theta-Pi-1}.
\end{proof}

For fixed $\theta$, $\cB^\theta$ can alternately be defined in terms of Busemann functions in the original  Brownian environment. In a weak sense, this is immediate from the quadrangle inequality. However, to prove a stronger result it will be convenient to have a full characterization of the law of $\cB^\theta$. This is the goal of the next subsection.

\subsection{The law of the multi-path Busemann process}

Next, we find the law of the process $\cB^\theta$ constructed in the previous section. The identification of the full law will follow from taking a limit of Pitman maps on $\cC^n(\R)$. 
\begin{theorem}
	\label{T:brownian-law}
	Consider $\theta \in [0, \infty)^n_\le$ for some $n \in \N$. Define $W \in \cC^n(\R)$ by:
	\begin{equation}
		\label{E:Wii-def}
\sum_{i=1}^k W_i(x) = \cB^{\theta^{\{1,\ldots,k\}}}(x^{i})
	\end{equation}
	for $x \in \R, k \in \{1, \dots, n\}$.
	 Then:
	\begin{enumerate}
		\item $W^0 := W - W(0) \in \cC^n_0(\R)$ is a sequence of independent two-sided Brownian motions  of drift $\sqrt{2\theta_i}$.
		\item Suppose that $\bn \in \N^k_<$ is such that $\bn \le \bm$ for any $\bm \in \N^k_<$ with $\theta_{n_i} = \theta_{m_i}$ (equivalently, $\bn$ satisfies \eqref{E:N-property} for $\theta$). Then almost surely, for $\bx \in \R^k_\le$, we have
		\begin{equation*}
			\cB^{\theta^{\bn}}(\bx) = W[(-\infty, \bn) \to (\bx, 1)].
		\end{equation*}
		\item For every $k \in \N$, let $\{\ell_k, \dots, k\}$ be the largest interval with $\theta_{\ell_k} = \theta_k$. We can recover $W$ from its centered version $W^0$ via the formula
		\begin{equation}
		\label{E:W-ellk}
			W_{\ell_k}(0) + \dots W_k(0) = - W^0[(-\infty, (\ell_k, \dots, k)) \to (0^{k-\ell_k + 1}, 1)].
		\end{equation}
	\end{enumerate}
\end{theorem}

The $n = 1$ case of Theorem \ref{T:brownian-law} is part of \cite[Theorem 3.7]{seppalainen2023global}. We will use this as an input for proving Theorem \ref{T:brownian-law}. As mentioned in the introduction (Equation \eqref{E:W-B-intro}), that theorem also contains a distributional identity for the joint law of the single-path Busemann process that Theorem \ref{T:brownian-law} realizes almost surely. Theorem \ref{T:brownian-law} is a more detailed version of Theorem \ref{T:Busemann-law-intro}.

\begin{proof}
First assume that all of the $\theta_i$ are distinct and nonzero. \\ Let $W^0 :=(W_1^0, \dots, W_n^0) \sim \mu_{\sqrt{2 \theta}}$, be chosen independently of the environment $B$. For $m = 0, 1, \dots$, consider the environment $W^m := (B_1, \dots, B_m, W^0)$, and consider the function
	$$
	\cK_m(\bn, \bx) := W^m[(-\infty, \bn + m) \to (\bx, 1)] - W^m[(-\infty, \bn + m) \to (0^{|\bn|}, 1)].
	$$
For every $m$, Corollary \ref{C:existence-of-stationary} implies that $\cK_m \eqd \cK_0$. Now consider the metric composition law for $W^m$:
\begin{equation}
\label{E:Wm-want}
W^m[(-\infty, \bn + m) \to (\bx, 1)] := \max_{\bz \le \bx} W^0[(-\infty, \bn) \to (\bz, 1)] + B[(\bz, m) \to (\bx, 1)].
\end{equation}
The shape theorems for Brownian last passage percolation (Propositions \ref{P:cross-prob} and \ref{P:top-bd}) imply that for any fixed $\bn \in \{1, \dots, n\}^k_<, \bx \in \R^k_\le$, the argmax in this metric composition law is contained in the box
$$
A(m, \theta) := \prod_{i=1}^k [-m/\theta_i - m^{3/4}, -m/\theta_i + m^{3/4}] 
$$
with probability tending to $1$ as $m \to \infty$. Note also that the location of this argmax is monotone in $\bx$ by monotonicity of optimizers, so for any fixed $a \in \R$, with probability tending to $1$, for all $\bx \in [-a, a]^k_\le$, this argmax is contained in $A(m, \theta)$,
and so \eqref{E:Wm-want} holds with $A(m, \theta)$ in place of $\R$. Therefore if for $\bz \in A(m, \theta)$, all of the functions
$$
B[(\bz, m) \to (\bx, 1)] - B[(\bz, m) \to (0^{|\bn|}, 1)], \qquad \bx \in [-a, a]^k_\le
$$
above are equal, then all of these functions equal $\cK_m(\bn, \bx)$. Equality of all these functions holds with probability tending to $1$ as $m \to \infty$ by Lemma \ref{L:Blx-continuity}, and moreover, equals $\hat \cB^{\theta^{\bn}}(\bx; B)  = \cB^{\theta^{\bn}}(\bx; B) - \cB^\theta(0^k; B)$ with probability tending to $1$. Therefore for any compact set $K$ in $\bn, \bx$, with probability tending to $1$ we have that 
$$
\cK_m(\bn, \bx) =\hat  \cB^{\theta^{\bn}}(\bx; B)
$$
for $\bx \in [-a, a]^k_\le$, and hence in distribution we have the equality
\begin{equation}
	\label{E:B-theta-norm-1}
\hat \cB^{\theta^{\bn}}(\bx) \eqd W^0[(-\infty, \bn) \to (\bx, 1)] - W^0[(-\infty, \bn) \to (0^k, 1)],
\end{equation}
jointly in all $\bx, \bn$, where $W^0 \sim \mu_{\sqrt{2 \theta}}$. Before moving on to using \eqref{E:B-theta-norm-1} to establish the theorem, let us extend to the case of general $\theta$. 

First, observe that the right-hand side of \eqref{E:B-theta-norm-1} is continuous in law in the uniform-on-compact topology on functions of $(\bn, \bx)$ with respect to the drift vector $\sqrt{2\theta}$, even when we allow repeated elements in $\theta$. Next, let $\theta^{n} = \theta + (1/m, 2/m, \dots, n/m)$ so that $\theta^{m} \cvgdown \theta$. By the right-continuity in Proposition \ref{P:multi-Busemann-exist} we have that
$
\hat \cB^{\theta^{m, \bn}}(\bx) \to \hat \cB^{\theta^{\bn}}(\bx)
$
for fixed $\bn, \bx$. Therefore by the case when the $\theta_i$ are distinct and non-zero, \eqref{E:B-theta-norm-1} holds at the level of finite-dimensional distribution. To show that it holds as continuous functions of $(\bn, \bx)$, we just need to check that $\hat \cB^\theta$ is continuous.

Let $g_n = \hat \cB^{\theta^n} - \hat \cB^{\theta}$. Then $g_n \downarrow 0$ pointwise, and for $\bx \le \by$, the increment $g_n(\by) - g_n(\bx) \ge 0$ by the quadrangle inequality. Therefore for any $a > 0, k \in \N$,
$$
\sup_{\bx \in [-a, a]^k_\le} |g_n(\bx)| \le \sup_{\bx \in [-a, a]^k_\le} |g_n(\bx) - g_n(-a^{k})| + |g_n(-a^{k})| = |g_n(a^{k}) - g_n(-a^{k})| + |g_n(-a^{k})|,
$$
where the final equality uses the quadrangle inequality. Hence $g_n \to 0$ uniformly on compacts, and so $\hat \cB^\theta$ is continuous, yielding \eqref{E:B-theta-norm-1}.

 Now, consider a coupling of $W^0, B$ where \eqref{E:B-theta-norm-1} holds. Define an environment $W$ from $W^0$ by the formula \eqref{E:W-ellk}.   
 We claim that Theorem \ref{T:brownian-law}.2 then holds. This immediately implies \eqref{E:Wii-def}, and hence the entire theorem. To see why Theorem \ref{T:brownian-law}.2 holds, fix a vector $\bn$ as in that part of the theorem, and first suppose that $\theta^{\bn}$ consists only of repeated elements. Since last passage commutes with constant shifts of the environment, \eqref{E:B-theta-norm-1} and the definition of $W$ implies that
 \begin{equation}
 	\label{E:repeated-elements}
 	\begin{split}
 	\cB^{\theta^{\bn}}(\bx) = \hat \cB^{\theta^{\bn}}(\bx) &= W^0[(-\infty, \bn) \to (\bx, 1)] - W^0[(-\infty, \bn) \to (0^k, 1)]\\
 	 &= W[(-\infty, \bn) \to (\bx, 1)].
 	\end{split}
 \end{equation}
Now suppose that we can write $\theta^\bn = (\la_1^{\ell_1}, \dots, \la_j^{\ell_j})$ for some $\la \in [0, \infty)^j_<$. By \eqref{E:B-theta-norm-1} we have that for some $a \in \R$,
$$
\cB^{\theta^{\bn}}(\bx) = W[(-\infty, \bn) \to (\bx, 1)] + a
$$
simultaneously for all $x$. On the other hand, arguing as in Lemma \ref{L:disjoint-geods}, we can find $\by = (y_1^{\ell_1}, \dots, y_j^{\ell_j})$ by sufficiently spacing out the coordinates $y_1 < \dots < y_j$ for which 
$$
\cB^{\theta^{\bn}}(\by) = \sum_{i=1}^j \cB^{\la_i^{\ell_i}}(y_i^{\ell_i}) , \qquad  W[(-\infty, \bn) \to (\by, 1)] = \sum_{i=1}^j W[(-\infty, \bn^i) \to (y_i^{\ell_i}, 1)],
$$
where $\bn^i$ are the coordinates of $n$ corresponding to the first $\ell_i$ coordinates of $\theta$ that equal $\la_i$. The right-hand sides above are equal by the repeated-element equality \eqref{E:repeated-elements}, yielding the result.
\end{proof}

One consequence of Theorem \ref{T:brownian-law} is the almost sure continuity of the multi-path Busemann process at a fixed $\theta$. This allows us to relate the functions $\hat \cB^\theta, \cB^{\theta}$ to true Busemann functions in the original  Brownian environment.

\begin{lemma}
	\label{L:hat-B-Busemann}
	Let $\theta \in [0, \infty)^k_\le$, and let $\Pi(\theta) = (I_1, \dots, I_m)$. For $\bx \in \R^k_\le$ define
	\begin{equation}
		\label{E:bar-B}
		\bar \cB^\theta(\bx) = \lim_{t \to -\infty} B[(t, \cl{\theta t} + k) \to (\bx, 1)] -  \sum_{i=1}^m B[( t, \cl{\theta^{I_j} t} + k) \to (0^{|I_j|}, 1)].
	\end{equation}
	Then for any fixed $\theta \in [0, \infty)^k_\le$, almost surely $\bar \cB^\theta = \cB^\theta$ for all $\bx \in \R^k_\le$ and
	\begin{equation*}
		\bar \cB^\theta(\bx) = \lim_{t \to -\infty} B[(t, \pi(t)) \to (\bx, 1)] -  \sum_{i=1}^m B[( t, \pi^{I_j}(t)) \to (0^{|I_j|}, 1)]
	\end{equation*}
	for any disjoint $k$-tuple $\pi$ ending at $(\bx, 1)$ in direction $\theta$. 
\end{lemma}
 The process $\bar \cB^\theta$ is more geometrically motivated than the original process $\cB^\theta$. We can think of the latter process as the right-continuous extension of $\bar \cB^\theta, \theta \in \R^k_\le$ to all choices of $\theta$.

\begin{proof}
	Let $a < b < a' < b'$, and consider any disjoint $k$-tuple $\pi$ ending at $(\bx, 1)$ in direction $\theta$. For $\bx \in [a, b]^k_\le, \by \in [a', b']^k_\le$ set 
	\begin{align*}
		\bar \cB^{\theta+}_\pi(\bx, \by) &= \limsup_{t \to -\infty} B[(t, \pi(t)) \to (\by, 1)] -  B[(t, \pi(t)) \to (\bx, 1)] \\
		\bar \cB^{\theta-}_\pi(\bx, \by) &= \liminf_{t \to -\infty} B[(t, \pi(t)) \to (\by, 1)] -  B[(t, \pi(t)) \to (\bx, 1)].
	\end{align*}
	We first show that on $[a, b]^k_\le \X [a', b']^k_\le$, almost surely we have $\bar \cB^{\theta+}_\pi(\bx, \by) = \bar \cB^{\theta-}_\pi(\bx, \by) = \hat \cB^\theta(\by) - \hat \cB^\theta(\bx)$. We first deal with the case when $\theta_1 \ne 0$. Let $\theta_n^-, \theta_n^+ \in \cI^k_<$ be such that 
	\begin{equation}
		\label{E:theta-cvg}
		\theta_{n, i}^- < \theta_i < \theta_{n, i}^+, \qquad \theta_n^- \to \theta, \theta_n^+ \to \theta \quad \text{ as } n \to \infty.
	\end{equation}
	Then by the quadrangle inequality, on $[a, b]^k_\le \X [a', b']^k_\le$ we have
	$$
	\hat \cB^{\theta^-_n}(\by) - \hat \cB^{\theta^-_n}(\bx) \le \bar \cB^{\theta\pm}_\pi(\bx, \by) \le \hat \cB^{\theta^+_n}(\by) - \hat \cB^{\theta^+_n}(\bx).
	$$
	On the other hand, as continuous functions on $[a, b]^k_\le \X [a', b']^k_\le$, both the left- and right-hand sides above converge in distribution to the same object: $\hat \cB^{\theta}(\by) - \hat \cB^{\theta}(\bx)$.  This follows from the characterization of the distribution of $\cB^\theta$ in Theorem \ref{T:brownian-law}. Hence, almost surely we have $\bar \cB^{\theta+}_\pi(\bx, \by) = \bar \cB^{\theta-}_\pi(\bx, \by) = \hat \cB^\theta(\by) - \hat \cB^\theta(\bx)$ on $[a, b]^k_\le \X [a', b']^k_\le$, as desired. 
	
	In the case when $\theta_1 = 0$ we cannot construct $\theta^-_n$ in $[0, \infty)^k_<$ the same way, so a minor modification is needed. Instead, let $\la$ be such that $\theta = (0^\ell, \la)$ for some $\ell \le k, \la \in (0, \infty)^{k-\ell}_\le$. Let $\la^-_n \in \cI^k_<$ be such that $\la^-_{n, i} < \la_i$ for all $i$ and $\la^-_n \to \la$ as $n \to \infty$. Then define
	$$
	F_n(\bx) = \sum_{i=1}^\ell B_i(x_1) + \max_{\bz \in [x_1, x_k]^{k - \ell}_\le} \cB^{\la^-_n}(\bz, \ell + 1) + B[(x_1^\ell, \bz) \to (\bx, 1)].
	$$
	Then $F_n(\by) - F_n(\bx)$ defines a lower bound on $\cB^{\theta\pm}_\pi(\bx, \by)$ by the quadrangle inequality. Moreover, in law as a function of $\bz$ we have that
	$$
	\cB^{\la^-_n}(\bz, \ell + 1) \eqd W^n[(-\infty, (\ell +1, \dots, k)) \to (\bz, \ell + 1)],
	$$
	where $(W^n)^0 \sim \mu_{\sqrt{2 \la^-_n}}$. This uses Theorem \ref{T:brownian-law}. Therefore as a function of $\bx$,
	$$
	F_n(\bx) \eqd (B_1, \dots, B_k, W^n)[(-\infty, (1, \dots, k)) \to (\bx, 1)].
	$$
	Therefore by Theorem \ref{T:brownian-law}, as $n \to \infty$, the law of $F_n(\by) - F_n(\bx)$ converges to the law of $\hat \cB^{\theta}(\by) - \hat \cB^{\theta}(\bx)$, which gives the desired result.
	
	Next, let $\bz \in \operatorname{DJ}(\theta)$ and let $\pi$ be a disjoint $k$-tuple to $(\bz, 1)$ in direction $\theta$. Then almost surely the following holds for all $\bx \in \R^k_\le$:
	\begin{align*}
		&\hat \cB^\theta(\bx) - \hat \cB^\theta(\bz) + \sum_{j=1}^\ell \hat \cB^{\theta^{I_j}}(\bz^{I_j}) = \lim_{t \to -\infty} B[(t, \pi(t)) \to (\bx, 1)] - B[(t, \pi(t)) \to (\bz, 1)] \\
		&+ \sum_{j=1}^\ell B[(t, \pi^{I_j}(t)) \to (\bz^{I_j}, 1)] + \sum_{j=1}^\ell B[(t, \pi^{I_j}(t)) \to (0^{|I_j|}, 1)].
	\end{align*}
	The two middle terms on the right-hand side above cancel out since $\bz \in \operatorname{DJ}(\theta)$, and so the whole expression equals $\bar  \cB^\theta(\bx)$. On the other hand, it also equals $\cB^\theta(\bx)$ by Proposition \ref{P:multi-Busemann-exist}.
\end{proof}

\begin{remark}
	\label{R:discty-set}
	Lemma \ref{L:hat-B-Busemann} guarantees the existence of the Busemann function $\bar \cB^\theta$ in every fixed direction almost surely, which shows that the set of directions where the Busemann function is not defined has Lebesgue measure $0$. In fact, the size of this set should be at most countable, just as in Theorem \ref{T:sepp-sorensen}. An analogue of the coalescence result in Theorem \ref{T:sepp-sorensen} should also hold on the set where the Busemann function exists. We do not pursue these directions here. Finally, when $\theta_1 > 0$, for $m \in \Z$ we can define the Busemann function in direction $\theta$ ending at $(\bx, m)$ by
	\begin{equation*}
		\bar \cB^\theta(\bx, m) = \lim_{t \to -\infty} B[(t, \cl{\theta t}) \to (\bx, 1)] - \sum_{i=1}^m B[( t, \cl{\theta^{I_j} t}) \to (0^{|I_j|}, 1)].
	\end{equation*} 
	Lemma \ref{L:hat-B-Busemann} implies that for fixed $\theta$, this limit exists almost surely for all $\bx, m$, and we can replace $\cl{\theta \; \cdot}$ with any disjoint $k$-tuple in direction $\theta$ ending at $(\bx, m)$. With this definition we have the following metric composition law (the analogue of Theorem \ref{T:sepp-sorensen}.2): for fixed $\theta \in (0, \infty)^k_\le$, almost surely for all $\bx \in \R^k_\le, m > n$:
	$$
	\bar \cB^\theta (\bx, n) = \sup_{\bz \le \bx} \bar \cB^\theta(\bz, m) + B[(\bz, m + 1) \to (\bx, n)].
	$$
	where the supremum is over all $\bz$ such that the last passage value above is well-defined.
\end{remark}

\subsection{The Busemann shear}
\label{S:multi-path-Busemann}

In this section, we use Busemann functions and Theorem \ref{T:brownian-law} to construct versions of the RSK correspondence for functions $f \in \cC^\N(\R)$, Theorem \ref{T:busemann-shear-intro}. We first recall a few definitions and restate that theorem.

For an environment $f \in \cC^\N(\R)$, let $Rf(x) = f(-x)$, and for $\bx \in \R^k_\le$, let $- \bx = (-x_k, \dots, -x_1) \in \R^k_\le$. For an environment of Brownian motions $B = (B_i, i \in \Z)$, we extend the definition of the Busemann process $\cB^\theta(\bx; B)$ to $\theta \in (-\infty, 0]^k_\le$ and $\bx \in \R^k_\le$ as follows:
\begin{equation}
	\label{E:conjugate-B}
\cB^\theta(\bx ; B) = \cB^{-\theta}(-\bx ; R B). 
\end{equation}
A quick comparison of the definitions guarantees that when $\theta = (0, \dots, 0)$, the unique point in $(-\infty, 0]^k_\le \cap [0, \infty)^k_\le$, our two definitions match up. Recall also the notation
$
\cW^\la(\bx) = \cW^\la(\bx; B) = \cB^{\la|\la|/2}(\bx, B),
$
which indexes Brownian Busemann functions by their slope (as $\bx \to \infty$), rather than by the Busemann direction in the original Brownian environment. 
\begin{theorem}
	\label{T:Busemann-shear}
	Let $B = (B_i, i \in \N) \in \cC^\N_0(\R)$ be an environment of independent two-sided Brownian motions. For $a \in \R$, define the environment $B^a \in \cC^\N(\R)$ by the rule that for $\ell \in \N$ we have
	$$
	\sum_{i=1}^\ell B^a_i(x) = \cW^{a^\ell}(x^\ell ; B).
	$$
	Then:
	\begin{enumerate}
		\item Almost surely, $B^0 = B$.
		\item For any $a \in \R$, $B^a$ is a sequence of independent Brownian motions with drift $a$. 
		\item Fix $a \in \R$ and $\la \in (-\infty, 0]^k_\le \cup [0, \infty)^k_\le$. Suppose also that $\la + a := (\la_1 + a, \dots, \la_k + a)\in (-\infty, 0]^k_\le \cup [0, \infty)^k_\le$. Then a.s.\ for all $\bx \in \R^k_\le$ we have that
		\begin{equation}
			\label{E:W-la-B-identity}
	\cW^{\la + a}(\bx; B) = \cW^{\la}(\bx; B^a).
		\end{equation}
		\item For any $a, b \in \R$, almost surely we have that: 
		\begin{equation}
			\label{E:Ba-a}
	(B^a)^{-a} = B, \qquad (B^a)^b = B^{a + b}.
		\end{equation}
		In particular, if for $a \in \R$ we define $a:\cC^\N(\R) \to \cC^\N(\R)$ by letting $a(f) = f^a(x) - ax$, then for any countable subgroup $G \subset \R$, this defines a measure-preserving $G$-action on a $\mu^\N$-almost sure subset of $\cC^\N(\R)$.
	\end{enumerate}
\end{theorem}

\begin{remark}
	\label{R:full-multi-path-Busemanns}
	Theorem \ref{T:Busemann-shear} allows us to naturally extend the definition $\cW^\la(\bx; B)$ to arbitrary $\la \in \R^k_\le$. Indeed, for $\la \in \R^k_\le$, define
	$$
	\cW^\la(\bx; B) := \cW^{\la - \la_1}(\bx; B^{\la_1}).
	$$
This definition is consistent with our original definitions when either $\la \in [0, \infty)^k_\le$ or $\la \in (-\infty, 0]^k_\le$ by Theorem \ref{T:Busemann-shear}.3. Moreover, since last passage percolation commutes with a common affine shifts applied to all lines, this extension has the remarkable invariance property that for any $a \in \R$ we have
	$$
	\cW^{\la + a}(\bx) \eqd \cW^{\la}(\bx) + a \bx,
	$$
	where the equality in distribution is joint in all $\la, \bx$. This invariance corresponds to shear invariance in the directed landscape. We call the full process
	$$
	(\la, \bx) \mapsto \cW^\la(\bx), \qquad (\la, \bx) \in \bigcup_{k \in \N} \R^k_\le \times \R^k_\le
	$$
	the \textbf{extended stationary horizon}. 
\end{remark}

\begin{proof}[Proof of Theorem \ref{T:Busemann-shear}]
Part $1$ is immediate from the definition of $\bar \cB$ in Lemma \ref{L:hat-B-Busemann}, and the fact that almost surely, $\bar \cB^{0^\ell} = \cW^{0^\ell}$ for all $\ell$ by that lemma. Part $2$ is a special case of Theorem \ref{T:brownian-law} when $a > 0$. When $a < 0$, it follows from Theorem \ref{T:brownian-law} and the formula \eqref{E:conjugate-B}. The first part of part $4$ is immediate from Part $3$. The `In particular' claim additionally uses that last passage commutes with a common shift of all functions (Lemma \ref{L:shift-commute}). The reason for restricting to a countable subgroup of $\R$ is because the almost sure set where \eqref{E:Ba-a} holds depends on the choice of $a, b$. 

We move to part $3$. First, for any $\la \in [0, \infty)^k_\le, a \in \R$ by Theorem \ref{T:brownian-law}.2 as functions of $\bx$ we have that
$$
\cW^\la(\bx, B^a) \eqd W[(-\infty, (1, \dots, k)) \to (\bx, 1)],
$$
where $W^0 \sim \mu_{\la+a}$ and $W$ is constructed from $W^0$ via \eqref{E:W-ellk}. 
The right-hand side above is continuous in $\bx$ almost surely, and hence so is the left. A similar argument holds when $\la \in (-\infty, 0]^k_\le$ by \eqref{E:conjugate-B}. Therefore it is enough to prove \eqref{E:W-la-B-identity} almost surely for any fixed $\bx$. We break into cases. 

\textbf{Case 1: $\la \in [0, \infty)^k_\le, a > 0$.} \qquad Consider the vector $\theta^m \in [0, \infty)^{m + k}_\le$ given by $\theta^m = (a^m, \la + a)$. Define the environment $W^m \in \cC^{m + k}(\R)$ by letting 
$$
\sum_{i=1}^\ell W^m_i(x) = \cW^{\theta^{\{1,\ldots,\ell\}}}(x^\ell).
$$
for $x \in \R, \ell \le m + k$. Then $W^m_i = B^a_i$ for $i \le m$ and by Theorem \ref{T:brownian-law} we have:
\begin{equation}
	\label{E:Wm-law}
\cW^{\la+a}(\bx; B) = W^m[(-\infty, (m+1, \dots, m+k)) \to (\bx, 1)].
\end{equation}
Now, let $\tilde W^m \in \cC^{m + k}_0(\R)$ be the environment with $W^m_i = B^a_i$ for $i \le m$ and with $W^m_{m+i}, 1 \le i \le k$ given by the identity:
$$
\sum_{i=1}^j W^m_{m+i}(x) = \bar \cB^\la(x^j, m + 1; B^a), \qquad 1 \le j \le k.
$$
Here we use notation as in Remark \ref{R:discty-set}.
Then by Theorem \ref{T:brownian-law} and the metric composition law (see Remark \ref{R:discty-set}), for any $\bn \in \{1, \dots, k\}^\ell_<$ satisfying \eqref{E:N-property} for $\la$ and $\bx \in \R^{\ell}_\le$ we have that
\begin{equation}
\label{E:tildeWm-law1}
\cW^{\la^\bn}(\bx; B^a) = \tilde W^m[(-\infty, \bn + m) \to (\bx, 1)].
\end{equation}
Now, by Theorem \ref{T:brownian-law} we have that $(W^m)^0 \eqd (\tilde W^m)^0 \sim \mu_{\theta^m}$. Moreover, $W^m_i(0) = \tilde W^m_i(0) = 0$ for $i \le m$ and  we can extract the values of $W^m_i(0), i \ge m + 1$ from $(W^m)^0$ by the formula in Theorem \ref{T:brownian-law}.3. The same formula extracts the values of $\tilde W^m_i(0), i \ge m + 1$ from $(\tilde W^m)^0$ by virtue of \eqref{E:tildeWm-law1} and the fact that $\cW^{\la^\bn}(0^{|\bn|}) = 0$ whenever $\la^\bn = (y, \dots, y)$ for some $y \ge 0$. Putting all this together implies $W^m \eqd \tilde W^m$.

 Now, using that $W^m_i = \tilde W^m_i = B^a_i$ for $i \le m$ and $W^m \eqd \tilde W^m$, letting 
$\cF_m$ denote the $\sig$-algebra generated by the lines $B^a_i, i \le m$ formulas \eqref{E:Wm-law}, \eqref{E:tildeWm-law1} give:
\begin{align}
	\label{E:xB}
(\cW^{\la+a}(\bx; B), \E(\cW^{\la+a}(\bx; B) \mid \cF_m)) &\eqd (\cW^{\la}(\bx; B^a), \E(\cW^{\la}(\bx; B^a) \mid \cF_m)), \\
\nonumber
\qquad \E(\cW^{\la+a}(\bx; B) \mid \cF_m) &= \E(\cW^{\la}(\bx; B^a) \mid \cF_m).
\end{align}
Now, the martingale convergence theorem ensures that
$\E(\cW^{\la}(\bx; B^a) \mid \cF_m) \to \cW^{\la}(\bx; B^a)$ almost surely as $m \to \infty$, since $\cW^{\la}(\bx; B^a)$ is $\sig(\bigcup_{m=1}^\infty \cF_m)$-measurable. We can use this to take the distributional limit of \eqref{E:xB} as $m \to \infty$ to get that
$$
(\cW^{\la+a}(\bx; B), \cW^{\la}(\bx; B^a)) \eqd (\cW^{\la}(\bx; B^a), \cW^{\la}(\bx; B^a)),
$$
which implies that $\cW^{\la+a}(\bx; B)= \cW^{\la}(\bx; B^a)$ almost surely, as desired.

\textbf{Case 2: $a + \la \in [0, \infty)^k_\le, a > 0, \la \in (-\infty, 0]^k_\le$.} \qquad In this case, let $\theta^m = (\la + a, a^m) \in [0, \infty)^{m+k}_\le$, and form $W^m$ from $\theta^m$ as above.
Let $\cP_{\operatorname{RSK}}$ be the Pitman transform introduced prior to Corollary \ref{C:RSK} sending $W^m$ to an environment with its slopes reversed. Set $X^m = \cP_{\operatorname{RSK}} W^m$. 
Since $\la(W^m) = \theta^m$ is in $\R^k_\le$, we can apply Proposition \ref{P:orbits} to recover the lines of $X^m$ from last passage values across $W^m$. In particular, for $j \le m$ we have
$$
\sum_{i=1}^j X^m_i(x) = W^m[(-\infty, (k + 1, \dots, k +j)) \to (x^j, 1)] = \sum_{i=1}^j B^a_i(x).
$$
Here the first equality uses Proposition \ref{P:orbits} and the second equality uses Theorem \ref{T:brownian-law}. Therefore $X_j = B^a_j$ for $j \le m$. We can also recover last passage values in $RW^m$ from those in $RX^m$ since $\la(R X^m) = \theta^m$ is in $\R^k_\le$:
\begin{align}
	\nonumber
RX^m[(-\infty, (m+1, \dots, m+k)) \to (-\bx, 1)] &= RW^m[(-\infty, (1, \dots, k)) \to (-\bx, 1)] \\
\nonumber
&=W^m[(-\infty, (1, \dots, k)) \to (\bx, 1)] \\
\label{E:one-way}
&= \cW^{\la+a}(\bx; B).
\end{align}
The first equality again uses Proposition \ref{P:orbits} and the second is easy to see from the definition. The third equality is Theorem \ref{T:brownian-law}. Next, let $\tilde X^m$ denote the environment with $X^m_i = \tilde X^m_i$ for $i \le m$ and such that for $1 \le j \le k$ we have
$$
\sum_{i=1}^j X^m_{m+i} = \bar \cB^\la(\bx, m + 1; B^a).
$$
Then by the metric composition law (see Remark \ref{R:discty-set}) we have that
\begin{equation}
	\label{E:tildeWm-law}
	\cW^{\la}(\bx; B^a) = R\tilde X^m[(-\infty, (m+1, \dots, m+k)) \to (-\bx, 1)].
\end{equation}
At this point we are in a similar situation to Case 1, with equations \eqref{E:one-way} and \eqref{E:tildeWm-law} playing the role of equations \eqref{E:Wm-law}, \eqref{E:tildeWm-law1} (equivalently, the environments $R X^m, R \tilde X^m$ take the place of $W^m, \tilde W^m$ from that case). Taking $m \to \infty$ in \eqref{E:one-way} and \eqref{E:tildeWm-law} and applying the same martingale argument as above gives that $\cW^{\la}(\bx; B^a) = \cW^{\la + a}(\bx; B)$ almost surely.

\textbf{Cases 3 and 4: $\la \in (-\infty, 0]^k_\le, a < 0$ or $a + \la \in (-\infty, 0]^k_\le, a < 0, \la \in [0, \infty)^k$.} \qquad These are simply cases $1$ and $2$ applied to the time-reversed environment $RB$. 

\textbf{Case 5: $a > 0, \la \in (-\infty, 0]^k_\le, a + \la \in (-\infty, 0]^k_\le$.} By Case $2$, we have the almost sure equality $(B^a)^{-a} = B$. Using this, we can rewrite our desired formula as
$$
\cW^{\la + a}(\bx; (B^a)^{-a}) = \cW^{\la}(\bx; B^a).
$$ 
This equality holds almost surely by applying Case $3$ to the environment $B^a$ with parameters $-a < 0, \la + a \in (-\infty, 0]^k_\le$ instead of $a, \la$. Case $3$ holds for the environment $B^a$ since we can remove the drift because last commutes with common functional shifts.

\textbf{Case 6: $a < 0, \la \in [0, \infty)^k_\le, a + \la \in [0, \infty)^k_\le$.} \qquad This is case $5$ applied to $RB$.
\end{proof}

%\begin{remark}
%	\label{R:Z-index}
%	We can extend the map in Theorem \ref{T:Busemann-shear}.4 by symmetry to map Brownian environments \textbf{indexed by $\Z$} instead of $\N$ as follows. Let $B = (B_i, i \in \Z)$ be an environment of independent Brownian motions. Just as we have defined Busemann functions along semi-infinite paths with $t \to -\infty$, we can define Busemann functions along paths with $t \to \infty$. Define the $180$-degree rotation $R_{180}:\cC^\Z(\R) \to \cC^\Z(\R)$ by letting $R_{180} B_i(x) =(R_{180} B)_i(x) = - B_{1-i}(-x)$, and for $f \in \cC^\Z(\R)$ let $f^+$ denote the projection of $f$ onto $\cC^\N(\R)$. Define
%	$$
%	\cW^{\uparrow, \la}(\bx ; B) = \cW^\la(-\bx ; (R_{180}B)^+), \qquad \qquad \cB^{\uparrow, \theta}(\bx ; B) = \cB^\theta(-\bx ; (R_{180}B)^+),
%	$$
%	and extend the definition of $B^a$ in Theorem \ref{T:Busemann-shear} to a $\Z$-indexed environment by letting
%	$$
%	\sum_{i=1-\ell}^0 B^a_{i}(x) = \cW^{\uparrow, a^\ell}(x^\ell ; B)
%	$$
%	for all $\ell \ge 1$. By Theorem \ref{T:Busemann-shear}, our extension of $B^a$ gives a $\Z$-indexed environment of independent Brownian motions of drift $a$.  All parts of Theorem \ref{T:Busemann-shear} have clear analogues this extension.
%\end{remark}

\section{RSK for the directed landscape}
\label{S:reconstruction}

In this section, we take a final RSK limit to prove Theorem \ref{T:landscape-recovery-intro}. As in the previous section, we will first need to develop a basic theory of multi-path Busemann functions in this setting.

\subsection{The extended stationary horizon and the directed landscape}
\label{SS:extended-multi-horizon}

With the machinery we have built studying Brownian LPP, we can now build a comprehensive theory of multi-path Busemann functions for the directed landscape. We start with the analogue of Theorem \ref{T:sepp-sorensen} for the directed landscape, which was proven in \cite{busani2024stationary} and \cite{rahman2021infinite}. Throughout the remainder of the paper, we say that a $k$-tuple of continuous functions $\pi = (\pi_1, \dots, \pi_k)$ where $\pi_i:(-\infty, t] \to \R$ is a \textbf{disjoint $k$-tuple} in direction $\theta$ to $(\bx, t)$ for some $\theta, \bx \in \R^k_\le$, if:
\begin{itemize}
	\item $\pi_i(s) < \pi_{i+1}(s)$ for $i = 1, \dots, k-1$ and $s < t$.
	\item $\pi(s)/|s| \to \theta$ as $s \to -\infty$ and $\pi(t) = \bx$.
\end{itemize}
We say that $\pi$ is a \textbf{(semi-infinite) optimizer} in $\cL$ if $\pi|_{[s, t]}$ is an optimizer for all $s < t$. As per usual, we talk call an optimizer with a single path a geodesic. 
\begin{theorem}\label{T:sepp-sor-busani}
	The following statements hold on a single event of full probability. 
	\begin{enumerate}
		\item(Theorem 3.14, \cite{rahman2021infinite}) For every point $p = (x, s) \in \R^2$ and every direction $\theta \in \R$, there exist leftmost and rightmost semi-infinite geodesics ending at $p$ in direction $\theta$. Call these geodesics $\pi^{\theta, p}$ and $\pi_R^{\theta,p}$. For fixed $p, \theta$, almost surely $\pi^{\theta, p} = \pi_R^{\theta, p}$.
			\item (Theorem 6.3(i), \cite{busani2024stationary}) For any $t < s$, the functions $(\theta, s) \mapsto \pi^{\theta, (x, s)}(t)$ and $(\theta, s) \mapsto \pi_R^{\theta, (x, s)}(t)$ are nondecreasing in $\theta, s$.
		\item (Theorem 2.5, \cite{busani2024stationary}) There exists a
		random, translation invariant, countably infinite, dense subset $\Xi \subset \R$ such that for $\theta \notin \Xi$, if  $\pi, \pi'$ are two semi-infinite geodesics in direction $\theta$, then $\pi(s) = \pi'(s)$ for all small enough $s$.
		\item (Theorem 5.1, \cite{busani2024stationary}) Consider $\theta \notin \Xi, p = (y, t) \in \R^2$, and any path $\pi:(-\infty, t] \to \R$ with $\pi(s)/|s| \to \theta$ as $s \to \infty$. Then the Busemann function
		$$
		\mathfrak{B}^\theta(p; \cL) := \lim_{s \to -\infty} \cL(\bar \pi(s); p) - \cL(\bar \pi(s); 0, 0)
		$$
		 exists and does not depend on the choice of $\pi$. Here recall the notation $\bar \pi(s):= (\pi(s), s)$. Moreover, for fixed $\pi$, the convergence above is uniform on compact subsets of $p \in \R^2$.
		\item (Theorem 4.3, \cite{rahman2021infinite}) For $s < t$ and $y \in \R$, we have the metric composition law
		$$
		\fB^\theta(y, t; \cL) = \max_{x \in \R} \fB^\theta(x, s; \cL) + \cL(x, s; y, t).
		$$
	\end{enumerate}
\end{theorem}

We can now state analogues of the main results of Section \ref{S:Busemann-A} for the directed landscape.
 The proofs are almost identical to the proofs for Brownian LPP, so we will go through the steps briefly. For $\bx \in \R^k_\le$ and $\theta \in \cJ^k_<$ define
 \begin{equation}
 	\label{E:B-la-x}
 	\begin{split}
 		\mathfrak{B}^\theta(\bx, t; \cL) &:= \lim_{s \to -\infty}  \cL( \theta|s|, s ; \bx, t) - \sum_{j=1}^{k} \cL(\theta_j|s|, s; 0, 0) 
 	\end{split}
 \end{equation}
 When $t = 0$ we omit it from the notation, and when $\cL$ is unambiguously defined, we also omit it from notation, so $\fB^\theta(\bx)= \fB^\theta(\bx, 0; \cL)$. 

\begin{proposition}
	\label{P:extended-usemann-landscape}
	Almost surely, the following claims hold for the directed landscape $\cL$.
	\begin{enumerate}
		\item For every $k \in \N, a > 0, t \in \R$ and every compact box $K = \prod_{i=1}^k [\theta_i^-, \theta_i^+]\subset \R^k_<$ there exist vectors $\by, \bz \in \Q^k_\le$ with $\by \le -a^k \le a^k \le \bz$ such that for any $\theta \in K$ the $k$-tuples of semi-infinite leftmost geodesics $\pi^{\theta, (\by, t)} := (\pi^{\theta, (y_1, t)}, \dots, \pi^{\theta, (y_k, t)})$ and $\pi^{\theta, \bz} := (\pi^{\theta, (z_1, t)}, \dots, \pi^{\theta, (z_k, t)})$ are both disjoint $k$-tuples. In particular, they are both disjoint optimizers.
		\item For any $a > 0, t \in \R, \theta \in \cJ^k_<$, there exists $\ep > 0$, there exists $\by \le -a^k < a^k \le \bz$ and $T < t$ such that whenever $\|\theta_1 - \theta\|_\infty < \ep, \|\theta_2 - \theta\|_\infty < \ep$:
		\begin{itemize}
			\item For any $i = 1, \dots, k$, any geodesics in direction $\theta_{1, i}$ to $(y_i, t)$ and direction $\theta_{2, i}$ to $(z_i, t)$ are equal at time $T$.
			\item Any geodesics from $\theta_{1, i}$ to $(y_i, t), i = 1, \dots, k$ are mutually disjoint. Similarly, any geodesics from $\theta_{2, i}$ to $(z_i, t)$ are mutually disjoint.
		\end{itemize}
		\item The limit \eqref{E:B-la-x} exists for all $k \in \N, \theta \in \cJ^k_<, \bx \in \R^k_\le, t \in \R$, and equals
		$$
		\lim_{s \to -\infty}  \cL(\bar \pi(s) ; \bx, t) - \sum_{j=1}^{k} \cL(\bar\pi(s); 0, 0) 
		$$
		for any $k$-tuple of paths $\pi = (\pi_1, \dots, \pi_k)$ with direction $\theta$.
		\item There exist semi-infinite optimizers in every direction $\theta \in \R^k_\le$ ending at every point in $(\bx, s) \in \R^k_\le \X \R$. 
		\item For any $\theta \le \theta', \bx \le \bx'$ and $t \in \R$ we have
		$$
		\fB^\theta(\bx', t) + \fB^{\theta'}(\bx, t) \le \fB^\theta(\bx, t) + \fB^{\theta'}(\bx', t).
		$$
				\item Let $\theta \in \cJ^k_<, a\in \R, \ep > 0$ be as in part $2$ above. If $\|\theta - \theta'\|_\infty < \ep$ and $\theta' \in \cJ^k_<$ as well, then for all $\bx \in [-a, a]^k_\le$,\begin{flushleft}
					
				\end{flushleft}
		$$
		\fB^\theta(\bx, t) - \fB^\theta(0^k, 0) = \fB^{\theta'}(\bx, t) - \fB^{\theta'}(0^k, 0).
		$$
		\item For $\bx, \theta \in \R^k_\le$ and $t \in \R$ define
		\begin{align}
			\label{E:overnormalized-Pi}
		\hat \fB^\theta(\bx, t) &= \lim_{\theta' \cvgdown \theta, \theta' \in \cI^k_<} \fB^{\theta'}(\bx, t) - \fB^{\theta'}(0^k) \\
			\label{E:B-theta-Pi-landscape}
			\fB^\theta(\bx, t) &= \lim_{\theta' \cvgdown \theta, \theta' \in \cI^k_<} \fB^{\theta'}(\bx, t) - \sum_{j=1}^{m(\theta)} \fB^{\theta^{\prime I_j(\theta)}}(0^{|I_j(\theta)|})
		\end{align}
		Then the limits above exists for all fixed $\bx \in \R^k_\le, t \in \R$.
	\end{enumerate}
\end{proposition}

\begin{proof}
Part $1$ follows as in the proof of Lemma \ref{L:disjoint-geods}, uses that translation invariance of the directed landscape in place of translation invariance of Brownian last passage percolation. Now fix $\theta \in \cJ^k_<, t > 0, a > 0$ and let $K =\prod_{i=1}^k [\theta_i^-, \theta_i^+]\subset \R^k_<$ contain $\theta$ in its interior. Let $\by, \bz$ be chosen as in part $1$ from $a, K$. Let $\theta_n^- = \theta - 1/n, \theta_n^+ = \theta + 1/n$ so that for large enough $n$, $\theta_n^{\pm} \in K$. Let $\pi^{\theta_n^\pm, \by}$ denote the $k$-tuple of leftmost geodesics in direction $\theta_n^\pm$ ending at $(\by, t)$. Then both of the monotone sequences $\pi^{\theta_n^-, \by}$ and $\pi^{\theta_n^+, \bz}$ have pointwise limits $\pi, \pi'$. We have $\pi^{\theta_n^-, \by} \le \pi \le \pi' \le \pi^{\theta_n^+, \bz}$ for all $n$ by monotonicity of geodesics (Theorem \ref{T:sepp-sor-busani}.2), and so $\pi, \pi'$ must be $k$-tuples of paths in direction $\theta$.

 Next, sequences of geodesics that converge pointwise must converge in overlap to limiting geodesics (Lemma \ref{L:overlap-cvg}.2). Therefore both $\pi, \pi'$ must both be $k$-tuples of geodesics in direction $\theta$, and hence coalesce together at all small enough times by Theorem \ref{T:sepp-sor-busani}.4. Moreover, overlap convergence ensures that $\pi^{\theta_n^-, \by}|_K = \pi|_K$ and ${\pi}^{\theta_n^+, \bz}|_K$ for all compact sets $K \subset (-\infty, t)$ and all large enough $n$, and hence by coalescence of $\pi, \pi'$ there exists $T$ such that $\pi^{\theta_n^-, \by}(T) = {\pi}^{\theta_n^+, \bz}(T) = \pi(T)=\pi'(T)$ for some large $n$, yielding part $2$. Therefore by monotonicity of finite optimizers, for $s$ small enough, for any $\bx \in [-a, a]^k_\le$, we have:
 $$
 \cL( \bar \pi^\theta(s) ; \bx, t) - \sum_{j=1}^{k} \cL(\bar \pi^\theta_j(s); 0, 0) =  \cL( \bar \pi(T); \bx, t) - \sum_{j=1}^{k} \cL(\bar \pi_j(T); 0, 0),
 $$
 yielding part $3$.
 
 For Part $4$, we can establish the existence of optimizers in directions $\theta \in \cJ^k_<$ as in the proof of Proposition \ref{P:extended-usemann}. Now consider $\theta \notin \cJ^k_<$ and let $\theta_n^\pm$ be sequences with $\theta_{n, i}^- < \theta_n < \theta_{n, i}^+$ that converge to $\theta$ and are in $\cJ^k_<$. Let $(\bx, s)\in \R^k_\le \X \R$, and let $\pi_n^-, \pi_n^+$ be sequences of optimizers to $(\bx, s)$ in directions $\theta_n^-, \theta_n^+$. 
 
 Consider the paths $\tau_n^- = \pi_n^- \wedge \pi_n^+$ and $\tau_n^+ = \tau_n^+ = \pi_n^- \vee \pi_n^+$. Both $\tau_n^-, \tau_n^+$ always consist of disjoint paths. Moreover, from the definition of path length, it is easy to check that 
 $$
 \|\tau_{n, i}^+|_{[0, t]}\|_\cL + \|\tau_{n, i}^-|_{[0, t]}\|_\cL =  \|\pi_{n, i}^+|_{[0, t]}\|_\cL + \|\pi_{n, i}^-|_{[0, t]}\|_\cL .
 $$
 This forces both $\tau_n^-$ and $\tau_n^+$ to be semi-infinite optimizers in directions $\theta_n^-, \theta_n^+$ with $\tau_n^- \le \tau_n^+$. In other words, in the initial choice of $\pi_n^-, \pi_n^+$ we may assume $\pi_n^- \le \pi_n^+$. Similarly, we may assume $\pi_1^- \le \pi_2^- \le \dots \le \pi_2^+ \le \pi_1^+$. Therefore letting $\pi = \lim_{n \to \infty} \pi_n^+$, we have that $\pi$ is a $k$-tuple of paths with asymptotic direction $\theta$ ending at $(\bx, 1)$. Moreover, for any $r_1 < r_2 < r_3 \in (-\infty, t]$ we have that
 \begin{equation}
 	\label{E:pi-r-MC}
 \cL(\bar \pi(r_1) ; \bar \pi(r_2)) +  \cL(\bar \pi(r_2) ; \bar \pi(r_3)) =  \cL(\bar \pi(r_1) ; \bar \pi(r_3)),
\end{equation}
 where $\bar \pi(r_i) = (\pi(r_i), r_i)$, since this equality holds for all the approximations $\pi_n^+$ and the extended landscape is continuous. Lemma 8.2 in \cite{dauvergne2021disjoint} implies that almost surely, any path $\pi$ satisfying \eqref{E:pi-r-MC} satisfies $\pi(q) \in \R^k_<$ for any rational point $q$. Now, for every $n \in \Z \cap (-\infty, t)$, we can find a disjoint optimizer from $\bar \pi(n-1)$ to $\bar \pi(n)$. We can also find a disjoint optimizer from $\bar \pi(\lceil t-1 \rceil)$ to $\bar \pi(t)$.
 Concatenating together all of these optimizers yields a $k$-tuple of disjoint paths $\pi'$ satisfying \eqref{E:pi-r-MC}. Moreover, the asymptotic direction of $\pi'$ is still $\theta$, since Lemma \ref{L:shape-land} controls how far our new $k$-tuple $\pi'$ can wander from the linear interpolation of $\pi|_{(-\infty, t] \cap \Z}$. Therefore $\pi'$ is a disjoint optimizer in direction $\theta$ ending at $(\bx, 1)$.
 
 Part $5, 6$ follows from the quadrangle inequality for finite values of $\cL$ (Lemma \ref{L:quadrangle-landscape}) and in Part $7$, the proof in Proposition \ref{P:multi-Busemann-exist} goes through verbatim. 
\end{proof}

Next, we aim to identify the law of the process $(\theta, \bx) \mapsto \fB^\theta(\bx)$. We can guess this law by using convergence of Brownian LPP from Theorem \ref{T:DL-cvg-extended}. Indeed, let $\cL^a$ be as in that theorem. For $\bx \in \R^k_\le$ and $\theta \in \cJ^k_<$, we can again attempt to define $\fB^\theta(\bx, 0; \cL^{a})$ via \eqref{E:B-la-x}, with $\cL^{a}$ used in place of $\cL$. Given an interval $[-b, b]$, this is almost surely well-defined for all $\bx \in [-b, b]^k_<$ when $a < 4 \theta_1$, and equals a Busemann function in the original environment $B^{a}$:
$$
\fB^\theta(\bx, 0; \cL^{a}) = \cB^{a^3/(2a - 8\theta)}(\bx, 0; B^{a}).
$$ 
We can identify the joint law in $\theta, \bx$ of right-hand side above by Theorem \ref{T:Busemann-shear}:
\begin{equation}
	\label{E:cB-nn}
	\begin{split}
&\cB^{a^3/(2a - 8\theta)}(\bx, 0; B^{a}) = \cW^{\sqrt{a^3/(a - 4\theta)}}(\bx, 0; B^{a})  \\
\eqd &\cW^{\sqrt{a^3/(a - 4\theta)} + a}(\bx, 0; B^0)
= \cW^{2 \theta + O(\theta^2/|a|)}(\bx, 0; B^0).
	\end{split}
\end{equation}
In other words, we should expect that $\fB^\theta(\bx; \cL) \eqd \cW^{2 \theta}(\bx ; B)$ jointly in all $\theta, \bx$, where $B$ is an environment of independent drift-free Brownian motions. When $\bx, \theta$ are both singletons, this was established in \cite[Theorem 5.3(iii)]{busani2024stationary}. The next proposition extends this to the level of the whole extended stationary horizon.
\begin{proposition}
	\label{P:landscape-law}
For any finite or countable subset $F \subset \bigcup_{k=1}^\infty \R^k_\le$, we have the equality in law
$
\fB^{\la/2}(\bx ; \cL) \eqd \cW^\la(\bx ; B)
$
jointly as continuous functions of $\la \in F, \bx \in \R^{|\la|}_\le$ with respect to the compact topology. Moreover, we have the following convergence in law as $a \to -\infty$:
\begin{equation}
\label{E:cvg-law-B-theta}
(\fB^\theta(\bx, 0; \cL^{-a}), \cL^{-a}) \cvgd (\fB^\theta(\bx, 0; \cL), \cL).
\end{equation}
Here the first coordinate should be viewed as a function of $\theta \in F, \bx \in \R^{|\theta|}_\le$, with the compact topology, and the second coordinate should be viewed as a function on $\mathfrak X_\uparrow$ with the compact topology.
\end{proposition}

Note that Proposition \ref{P:landscape-law} is technically restricted to only looking at finite dimensional distributions in $\la$. This is purely to avoid discussing topological issues involving jumps of the Busemann process. To prove Proposition \ref{P:landscape-law}, we need a version of the continuous mapping theorem.

\begin{lemma}
\label{L:cty-theorem}
Suppose that $X_n, X$ are random variables taking values in a metric space $(S, e)$ and that $X_n \cvgd X$. Let $f:S \to T$ be a Borel-measurable function taking values in another metric space $(T, d)$, and suppose that there are functions $f_m:S \to T$ such that the following conditions hold:
\begin{itemize}
	\item $f_m(X) \cvgd f(X)$ as $m \to \infty$.
	\item $f_m(X_n) \cvgd f_m(X)$ as $n \to \infty$ for all fixed $m$.
	\item $\lim_{m \to \infty} \limsup_{n \to \infty} \E(d(f_m(X_n), f(X_n)) \wedge 1) = 0$.
\end{itemize}
Then $f(X_n) \cvgd f(X)$.
\end{lemma}

\begin{proof}
Let $g:T \to \R$ be any $1$-Lipschitz function with $\|g\|_\infty \le 1$. To prove the lemma it is enough to show that $\E g \circ f(X_n) \to \E g\circ f(X)$ as $n \to \infty$. For any $m \in \N$ we can write
\begin{align*}
|\E g \circ f(X_n) - \E g\circ f(X)| &\le |\E g \circ f(X_n) - \E g\circ f_m(X_n)| \\
&+ |\E g \circ f_m(X_n) - \E g\circ f_m(X)| + |\E g \circ f_m(X) - \E g\circ f(X)|.
\end{align*}
As we take $n \to \infty$, the middle term on the right-hand side above converges to $0$ since $f_m(X_n) \cvgd f_m(X)$. Moreover, as we take $m \to \infty$, the final term above converges to $0$ by the dominated convergence theorem since $g$ is continuous and $f_m(X) \cvgd f(X)$. Finally, since $g$ is $1$-Lipschitz and has $L^\infty$-norm at most $1$, we have:
$$
\lim_{m \to \infty} \limsup_{n \to \infty} |\E g \circ f(X_n) - \E g\circ f_m(X_n)| \le \lim_{m \to \infty} \limsup_{n \to \infty} \E(d(f_m(X_n), f(X_n)) \wedge 1) = 0,
$$
where the finally bound uses the second bullet point.
\end{proof}

\begin{proof}[Proof of Proposition \ref{P:landscape-law}]
The claimed convergence in law implies the equality in distribution, so we focus on this point. In order to prove this convergence in distribution, it suffices to show that if we consider a finite set $F \subset \bigcup_{k=1}^\infty \R^k_\le$ and the compact set
$$
K_b := \{(\theta, \bx) : \theta \in F, \bx \in [-b, b]^{|\theta|}_\le\}, 
$$
we have convergence in law in \eqref{E:cvg-law-B-theta} when $(\theta, x)$ is restricted to $K_b$ and convergence of the first coordinates is with respect to the $L^\infty$-norm on functions from $K_b \to \R \cup \{\pm \infty\}$. We first deal with the case when $F \subset \bigcup_{k=1}^\infty \R^k_<$. We will also assume that $F$ has the following \emph{downward closed} property: if $\theta \in F \cap \R^k_<$ and $I \subset \{1, \dots, k\}$, then $\theta^I \in F$ as well.

We appeal to the framework of Lemma \ref{L:cty-theorem}. Let $a_n \to -\infty$ be an arbitrary sequence. The random variables $X_n, X$ in that lemma will be given by $\cL^{a_n}, \cL$; the compact topology on functions $h:\mathfrak X_\uparrow \to \R$ is metrizable and the choice of our metric is not important. The function $f$ will be given by mapping $h:\mathfrak X_\uparrow \to \R$ to the pair
$$
f(h) := (\limsup_{s \to -\infty}  h( \theta|s|, s ; \bx, t) - \sum_{j=1}^{k} h(\theta_j|s|, s; 0, 0), h):K_b \times \mathfrak X_\uparrow \to \R \cup \{\pm \infty\}.
$$
When $h = \cL^a$, the first coordinate is $\mathfrak B^\theta (\bx, 0; \cL^{-a})$ almost surely on all of $K_b$ when $a$ is small enough by Proposition \ref{P:extended-usemann}, 
and is $\mathfrak B^\theta(\bx, 0; \cL)$ almost surely on all of $K_b$ by Proposition \ref{P:extended-usemann-landscape} when $h = \cL$. We define $f_m(h)$
to have the same second coordinate as $f(h)$, but with the first coordinate replaced with the following:
\begin{equation}
	\label{E:tmm}
	\begin{split}	
	h( \theta t_m, -t_m ; \bx, t) &- \sum_{j=1}^k h(\theta_j t_m, -t_m; 0, 0) \\
	&+ \Big[\inf_{\bz \in [-m, m]^k_\le} \sum_{j=1}^k h(\theta_j t_m, -t_m; z_j, 0) - h(\theta t_m, -t_m; \bz, 0)\Big].
	\end{split}
\end{equation}
Here $t_m$ is a sequence that will tend to $\infty$ very quickly with $m$; we set this sequence in the course of the proof. In our setup, Proposition \ref{P:extended-usemann-landscape} ensures that the first condition of Lemma \ref{L:cty-theorem} is satisfied; note that the shape theorem for $\cL$ (Lemma \ref{L:shape-land}) guarantees that the bracketed term in \eqref{E:tmm} is $0$ with probability tending to $1$ for all large enough $m$. The second condition follows since $\cL^{a} \cvgd \cL$ as $a \to -\infty$ and $f_m$ is continuous. For the final condition, we show that
\begin{equation}
\label{E:lim-mm}
\lim_{m \to \infty} \liminf_{a \to -\infty} \P(f_m(\cL^a) = f(\cL^a)) = 1.
\end{equation}
First, using Lemma \ref{L:Blx-continuity}, we can find $\de > 0$ such that the following condition holds with probability $1 - 1/m$ for all sufficiently large $a$:
\begin{itemize}
	\item For all $\bx \in [-m, m]^k_\le, \theta \in F, \|\theta' - \theta\|_\infty \le \de$:
	$$
	\fB^{\theta'}(\bx, 0; \cL^{-a}) - \fB^{\theta'}(0^k, 0; \cL^{-a}) = \fB^{\theta}(\bx, 0; \cL^{-a}) - \fB^{\theta}(0^k, 0; \cL^{-a}).
	$$
\end{itemize}
Here the fact that we can use the same $\ep$ uniformly over all large enough $a$ follows from the computation \eqref{E:cB-nn}, which implies that all claims for $\cL^{-a}$ are implied by a common claim for an undrifted Brownian environment $B^0$. Now, the shape theorems for Brownian LPP (Propositions \ref{P:cross-prob} and \ref{P:top-bd}) and the quadrangle inequality imply that there exists some $s_m < 0$, such that for all large enough $a$, the following bound holds with probability $1 - 1/m$:
\begin{itemize}
	\item For all $\by \le \bx \in [-b, b]^k_\le, \theta \in F, t < s_m$, we have:
	\begin{align*}
		\fB^{\theta - \de}(\bx, 0; \cL^{-a}) - \fB^{\theta - \de}(\by, 0; \cL^{-a}) &\le \cL^{-a}(\theta |t|, t ; \bx, 0) - \cL^{-a}(\theta |t|, t ; \by, 0) \\
		&\le \fB^{\theta + \de}(\bx, 0; \cL^{-a}) - \fB^{\theta + \de}(\by, 0; \cL^{-a})
	\end{align*}
\end{itemize} 
Combining this with the previous bullet, as long as we take $t_m \le s_m$, with probability tending to $1$ we have that with probability at least $1-2/m$, for all large enough $a$, for $(\theta, \bx) \in K_b$ we have that the first coordinate of $f(\cL^{-a})$ at $(\theta, \bx)$ equals
$$
\hat {\mathfrak B}^\theta (\bx, 0; \cL^{-a}) + \inf_{\bz \in [-m, m]^k_\le} \sum_{j=1}^k \mathfrak B^{\theta_i}(z_i, 0; \cL^{-a}) - \hat {\mathfrak B}^\theta(\bz, 0; \cL^{-a}).
$$
On the other hand, the infimum on the right-hand side equals $0$ with probability tending to $1$ with $m$, uniformly over all large enough $a$. The uniformity over $a$ follows from the representation of the processes $\mathfrak B^\cdot(\cdot; \cL^{-a})$ in terms of the Busemann process across $B^0$, \eqref{E:cB-nn}. In summary, we have shown \eqref{E:lim-mm} as desired. 

It remains to check the case when $F \subset \bigcup_{k=1}^\infty \R^k_\le$ (i.e. elements of $F$ can have repeated points). Let 
$$
f(h) := (\limsup_{s \to -\infty}  h( \theta|s|, s ; \bx, t) - \sum_{j=1}^{m(\theta)} h(\theta^{I_j(\theta)}|s|, s; 0^{|I_j(\theta)|}, 0), h),
$$
and set the approximations $f_m$ to be given by replacing $\theta$ with $\theta + (0, 1/m, \dots, (k-1)/m)$ in the definition of $f$. Then by the previous case, $f_m(\cL^{-a}) \cvgd f_m(\cL)$ as $a \to \infty$ for every fixed $m$ so the second condition of Lemma \ref{L:cty-theorem} is satisfied. The first condition is satisfied by Proposition \ref{P:extended-usemann-landscape}.7. Finally, if we let $f_1, f_{m, 1}$ denote the first coordinate of these functions, we have
$$
\lim_{m \to \infty} \limsup_{a \to \infty} \E \|f_1(\cL^{-a}) - f_{m, 1}(\cL^{-a})\|_\infty \wedge 1 = 0.
$$
This follows again since we know the joint law of $f_1(\cL^{-a}), f_{m, 1}(\cL^{-a})$ for all fixed $a$: it is simply given via Theorem \ref{T:brownian-law}, and this explicit representation in terms of Brownian motions allows us to easily check the above convergence to $0$. This yields the third condition of Lemma \ref{L:cty-theorem}, so by that lemma the proof is complete.
\end{proof}

Given that we have identified the law of the process $\fB$, we can now give the analogue of Lemma \ref{L:hat-B-Busemann} in the limiting setting of $\fB$.
\begin{corollary}
	\label{C:hat-B-Busemann-landscape}
For $\theta, \bx \in \R^k_\le$ let $\Pi(\theta) = (I_1, \dots, I_{m(\theta)})$ and define
$$
\bar \fB^\theta(\bx, t; \cL) = \lim_{s \to -\infty} \cL(\theta |s|, s ; \bx, 0) - \sum_{j=1}^{m(\theta)}  \cL(s, \theta^{I_j} |s| ; \bx^{I_j}, 0).
$$
Then for any fixed $\theta \in \R^k_\le$, almost surely $\bar \cB^\theta(\bx, t; \cL) = \cB^\theta(\bx, t; \cL)$ for all $\bx \in \R^k_\le, t \in \R$ and
\begin{equation*}
\bar \fB^\theta(\bx, t; \cL) = \lim_{s \to -\infty} \cL(\pi(s), s ; \bx, 0) - \sum_{j=1}^{m(\theta)}  \cL(s, \pi^{I_j}(s) ; \bx^{I_j}, 0)
\end{equation*}
for any disjoint $k$-tuple $\pi$ ending at $(\bx, 1)$ in direction $\theta$. 
\end{corollary}

The proof of Lemma \ref{L:hat-B-Busemann} applies verbatim to Corollary \ref{C:hat-B-Busemann-landscape}, since that proof only used monotonicity/the quadrangle inequality and a characterization of the Busemann law, which we have in this context by Proposition \ref{P:landscape-law}. We leave the details to the reader. Corollary \ref{C:hat-B-Busemann-landscape} also implies a metric composition law for Busemann functions in $\cL$, which can be used to prove optimizer uniqueness.

\begin{lemma}
\label{L:mc-law-busemann-L}
Fix $\theta \in \R^k_\le$. Then the following holds almost surely for all $\bx \in \R^k_\le$ and $s < t$:
\begin{equation}
\label{E:fB-theta-mc}
\fB^\theta(\bx, t; \cL) = \max_{\bz \in \R^k_\le} \fB^\theta(\bz, s; \cL) + \cL(\bz, s; \by, t).
\end{equation}

\end{lemma}

\begin{proof}
We work on the almost sure set where $\fB^\theta = \bar \fB^\theta$ from Corollary \ref{C:hat-B-Busemann-landscape}. Let $\pi$ be a semi-infinite optimizer to $(\bx, t)$ in direction $\theta$, and for $r < s \le t$ and $\bz \in \R^k_\le$ define 
$$
\fB^\theta_r(\bz, s; \cL) = \cL(\pi(r),r ; \bz, s) - \sum_{j=1}^{m(\theta)}  \cL(r, \pi^{I_j}(r) ; \bx^{I_j}, 0).
$$
Then the identity \eqref{E:fB-theta-mc} holds with $\bar \fB^\theta_r$ in place of $\bar \fB^\theta$ on both sides by the usual metric composition law. Since $\pi$ is an optimizer we also have that
$$
	\fB^\theta_r(\bx, t; \cL) = \fB^\theta_r(\pi(s), s; \cL) + \cL(\pi(s), s; \by, t).
$$
for all $r < s$. Combining these facts with the pointwise convergence $\fB^\theta_r \to \bar \fB^\theta = \fB^\theta$ from Corollary \ref{C:hat-B-Busemann-landscape} yields the result.
\end{proof}

\begin{lemma}
	\label{L:optimizer-as-unique}
Fix $\theta \in \R^k_\le, \bx \in \R^k_\le$, and $t \in \R$. Then almost surely, there is a unique semi-infinite optimizer in direction $\theta$ to the point $(\bx, t)$.
\end{lemma}

\begin{proof}
If $\pi$ is a semi-infinite optimizer in direction $\theta$ to $(\bx, t)$, then using Lemma \ref{L:mc-law-busemann-L}, for every $s < t$, the point $\pi(s)$ must be an argmax for the metric composition law \eqref{E:fB-theta-mc}. Therefore to prove Lemma \ref{L:mc-law-busemann-L} it is enough to show that the optimizer in \eqref{E:fB-theta-mc} is almost surely uniquely achieved for any fixed $s$. This will imply that $\pi$ is a.s.\ uniquely specified at all rational $s$, and hence by continuity is a.s.\ unique. The proof that the argmax has a unique maximum can be found in \cite[Lemma 7.3]{dauvergne2021disjoint}. (Note: that proof gives uniqueness for the same problem when the function $g(\bz, s) := \fB^\theta(\bz, s; \cL)$ is replaced by $h(\bz, s) := \cL(\by, r; \bz, s)$ for general $(\by, r)$. However, the only information about the function $h$ used in that proof is a soft probabilistic representation which in the context of our function $g$ is given by Theorem \ref{T:brownian-law}).
\end{proof}

\subsection{RSK for the directed landscape}

Now that we have defined Busemann functions for the directed landscape, we are ready to define the RSK correspondence in the directed landscape limit and prove Theorem \ref{T:landscape-recovery-intro}. We begin by restating a version of the theorem here.

\begin{theorem}
\label{T:landscape-recovery}
Let $\cL$ be a directed landscape, and for every $a \in \R$ define a sequence of lines $B^a = (B^a_i, i \in \N)$ via the formula
$$
\sum_{i=1}^k B^a_i(x) = \fB^{(a/2)^k}(x^k, 0; \cL).
$$
Then:
\begin{enumerate}
	\item (Law of the RSK image) The joint law of $(B^a, a \in \R)$ is same as the joint law of the processes $(B^a, a \in \R)$ defined in Theorem \ref{T:Busemann-shear}. In particular, each $B^a$ is a sequence of independent two-sided Brownian motions of drift $a$ and we have the Busemann isometry
	$$
	\cW^{\theta}(\bx; B^a) = \fB^{(\theta + a)/2}(\bx; \cL)
	$$
	for all $\bx, \theta \in \R^k_\le, k \in \N$. 
	\item (Abstract Invertibility) Define $\mathbb H^2_\uparrow := \{(x, s; y, t) \in \Rd : t \le 0\}$. There is a measurable function $f:\cC^\N(\R) \to \cC(\mathfrak{X}^-_\uparrow)$ such that almost surely $f(B^0) = \cL|_{\mathfrak{X}^-_\uparrow}$.
	\item (Explicit Inversion Formula $1$) For every $a \in R$, as in Theorem \ref{T:DL-cvg-extended}, for $(\bx, s; \by, t) \in \mathfrak{X}^-_\uparrow :=\{(\bx, s; \by, t) \in \mathfrak{X}_\uparrow : t \le 0\}$ define
	$$
	\cL^a(\bx, s; \by, t) = B^a[(\bx, s)_a \to (\by, t)_a] - \frac{a^2}{4}(t-s),
	$$
	where $(\bx, s)_a = (\bx-as/4, \cl{s a^3/8} + 1)$. Then as $a \to -\infty$, $\cL^a$ converges to $\cL$ in probability, in either of the following equivalent senses:
	\begin{enumerate}
		\item For every compact set $K \subset \mathfrak{X}^-_\uparrow$ we have
		$$
		\sup_{u \in K} |\cL^a(u) - \cL(u)| \to 0
		$$
		in probability as $a \to -\infty$.
		\item For every sequence $a_n \to -\infty$ there is a subsequence $n_1 < n_2 < \dots$ such that $\cL^{a_{n_k}} \to \cL$ almost surely as $k \to \infty$, in the compact topology on functions on $\mathfrak{X}^-_\uparrow$.
			\end{enumerate}
		\item (Explicit Inversion Formula 2) For every $a \in \R$, extend the Brownian line environments $B^a$ to environments indexed by $i \in \Z$ by the rule that:
		$$
		\sum_{i=1}^k B^a_{1 - i}(x) = \lim_{t \to \infty} \cL(0^k,0; -(a/2)^k t, t) - \cL(x^k,0; -(a/2)^k t, t),
		$$
		and use this to extend the definition of $\cL^a$ to all of $\mathfrak X_\uparrow$. Then as $a \to - \infty$, $\cL^a$ converges to $\cL$ in probability as functions in the compact topology on $\mathfrak{X}_\uparrow$.
\end{enumerate}
\end{theorem}
Theorem \ref{T:landscape-recovery-intro} is Theorem \ref{T:landscape-recovery}(iv). Note that we have defined $(\bx, s)_a$ using a different rounding convention than in the introduction. The two choices are clearly equivalent, but the latter choice associates $\mathfrak{X}^-_\uparrow$ to $\R \X \N$.

Theorem \ref{T:landscape-recovery}.1 is a consequence of Proposition \ref{P:landscape-law} and the Busemann shear characterization, Theorem \ref{T:Busemann-shear}. Most of the remainder of Section \ref{S:reconstruction} is devoted to proving the abstract inversion result in Theorem \ref{T:landscape-recovery}.2. Given abstract invertibility, we can use the following measure-theoretic lemma to move to the explicit inversion formulas in parts $3, 4$.

\begin{lemma}
	\label{L:measure-theory-lemma-b}
	Let $\{X_n, n \in \N\}, X$ be random variables taking values in a complete separable metric space $(S, d)$ equipped with its Borel $\sig$-algebra, and let $\{Y_n, n \in \N\}, Y$ be random variables taking values in a Banach space $(T, \|\cdot\|)$ equipped with its Borel $\sig$-algebra.
	Suppose that:
	\begin{enumerate}
		\item As $n \to \infty$, $(X_n, Y_n) \cvgd (X, Y)$ and $X_n \cvgp X$.
		\item There is a measurable function $h:S \to T$ such that $Y = h(X)$ a.s.
	\end{enumerate}
	Then $Y_n \cvgp Y$ as $n \to \infty$.
\end{lemma}

\begin{proof}
	For every $\ep > 0$, by Lusin's theorem we can find a closed set $E \subset S$ such that $h|_E$ is continuous and $\P(X \notin E) \le \ep$. Next, Dugundji's extension theorem \cite{dugundji1951extension} states that if $S$ is any metric space, $E \subset S$ is closed, and $T$ is a locally convex topological vector space, then any continuous function from $E \to T$ can be extended to a continuous function from $S \to T$. In our setting, since Banach spaces are locally convex, this implies that there exists a continuous function $h_\ep:S \to T$ such that $h|_E = h_\ep|_E$. 
	Now, for $\de > 0$ we can write
	\begin{align}
	\nonumber
	\P(\|Y_n -Y\| > \de) &\le \P(X \notin E) + \P(\|Y_n - Y\| > \de, X \in E) \\
	\nonumber
	&\le \ep + \P(\|Y_n - h_\ep(X)\| > \de) \\
	\label{E:left-with}
	&\le \ep + \P(\|Y_n - h_\ep(X_n)\| \ge \de/2) + \P(\|h_\ep(X_n)- h_\ep(X)\| > \de/2).
	\end{align}
	Here the second inequality uses that $Y = h(X) = h_\ep(X)$ when $X \in E$. Now, since $X_n \cvgp X$ and $h_\ep$ is continuous, $h_\ep(X_n) \cvgp h_\ep(X)$ as well, so the second probability in \eqref{E:left-with} converges to $0$ as $n \to \infty$. Also, since both of the functions $\|\cdot\|, h_\ep$ are continuous and $(X_n, Y_n) \cvgd (X, Y)$ we have that $\|Y_n - h_\ep(X_n)\| \cvgd \|Y -h_\ep(X)\|$. The random variable $\|Y -h_\ep(X)\|$ equals $0$ unless $X \in E$, and so 
	$$
	\limsup_{n \to \infty} \P(\|Y_n- h_\ep(X_n)\| \ge \de/2) \le \P(\|Y - h_\ep(X)\| \ge \de/2) \le \P(X \in E) \le \ep.
	$$
	Therefore as $n \to \infty$, the limsup of \eqref{E:left-with} is at most $2 \ep$. Letting $\ep \to 0$ completes the proof.
\end{proof}

\begin{proof}[Proof of Theorem \ref{T:landscape-recovery}.3, \ref{T:landscape-recovery}.4 from Theorem \ref{T:landscape-recovery}.2]
	We start with part $3$. Our aim is to appeal to Lemma \ref{L:measure-theory-lemma-b}. We first observe that the space $T$ of functions $f:\mathfrak X_\uparrow^- \to \R \cup\{\pm \infty\}$ is a Banach space in the following norm:
	$$
	\|f\| = \sum_{i=1}^\infty 2^{-i} \frac{\|f|_{K_i}\|_\infty}{1 + \|f|_{K_i}\|_\infty},
	$$
	where $K_i, i \in \N$ is any collection of compact subsets of $\mathfrak X^-_\uparrow$ whose union is all of $\mathfrak X^-_\uparrow$. Convergence in probability in $(T, \|\cdot\|)$ is equivalent to the two types of convergence described in the statement of the theorem.
	
	Now, fix a sequence $a_n \to -\infty$. Our goal will be to apply Lemma \ref{L:measure-theory-lemma-b} with $Y_n = \cL^{a_n}, Y = \cL|_{\mathfrak X_\uparrow^-}$. We let $X_n, X \in \cC^\N(\R)$ be the line ensembles given by
	$$
	\sum_{i=1}^k (X_n)_i(x) = \fB^{0^k}(x^k; \cL^{a_n}), \qquad \sum_{i=1}^k X_i(x) = \fB^{0^k}(x^k; \cL).
	$$
	From our construction, we actually have that $X_n = X$ for all $n$. Moreover, Theorem \ref{T:landscape-recovery}.2 implies that $Y = h(X)$ for some measurable function $h$. This uses that $\cL|_{\mathfrak X^-_\uparrow}$ is defined as a function of $\cL|_{\mathbb H^2_\uparrow}$. Therefore Theorem \ref{T:landscape-recovery}.3 follows from Lemma \ref{L:measure-theory-lemma-b} and the joint convergence 
	\begin{equation}
	\label{E:Wbeta-B}
	(\cL^a, \fB^{0^k}(\cdot; \cL^a), k \in \N) \cvgd (\cL, \fB^{0^k}(\cdot; \cL), k \in \N)
	\end{equation}
	as $a \to -\infty$, which is proven in Proposition \ref{P:landscape-law}.
	
	The proof of Theorem \ref{T:landscape-recovery}.4 is essentially the same. In this setting, we instead work with the Banach space of functions from $\mathfrak X_\uparrow \to \R \cup \{\pm \infty\}$ with the compact topology, set $Y_n = \cL^{a_n}, Y = \cL$ and let $X_n, X \in \cC^\Z(\R)$ extend the previous definitions of $X_n, X$ by letting
$$
\sum_{i=1}^k (X_n)_{1-i}(x) = \fB^{0^k}(-x^k; \hat \cL^{a_n}), \qquad \sum_{i=1}^k X_{1-i}(x) = \fB^{0^k}(-x^k; \hat \cL),
$$
where $\hat \cL^a(\bx, s; \by, t) = \cL^a(-\by, -t; -\bx, -s)$ (and similarly for $\cL$). Again, $X_n = X$ for all $n$, and by Proposition \ref{P:landscape-law} and symmetry, we again have the convergence $(X_n, Y_n) \cvgd (X, Y)$. Finally, we claim that $Y \in \sig(X)$. From Theorem \ref{T:landscape-recovery}.2 and symmetry, we have that $\cL|_{\mathfrak X_\uparrow^-}, \hat \cL|_{\mathfrak X_\uparrow^-} \in \sig(X)$. Moreover $\cL|_{\mathfrak X_\uparrow} \in \sig(\cL|_{\mathfrak X_\uparrow^-}, \hat \cL|_{\mathfrak X_\uparrow^-})$ by the metric composition law, yielding the result.
\end{proof}

As discussed in the introduction, we can prove that the directed landscape on a strip is a measurable function of the Airy line ensemble with relative ease given Theorem \ref{T:landscape-recovery}. We restate the result and prove it here. For this corollary, let
$$
\mathbb S^2_\uparrow = \{(x, s; y, t) \in (\R \X [0, 1])^2 : s < t\}.
$$

\begin{corollary}
	\label{C:strip-reconstruction}
	Let $\cL$ be the directed landscape, and define process $\mathfrak A$ as follows:
	$$
	\sum_{i=1}^k \mathfrak A_i(x) = \cL(0^k, 0; x^k, 1).
	$$
	Then $\mathfrak A \in \cC^\N(\R)$ is a parabolic Airy line ensemble and there is a measurable function $f:\cC^\N(\R) \to \cC(\mathbb S^2_\uparrow)$ such that almost surely
	$
	f(\mathfrak A) = \cL|_{\mathbb S^2_\uparrow}.
	$
\end{corollary}

Corollary \ref{C:strip-reconstruction} is a restatement of Theorem \ref{T:Airy-RSK-intro}.
\begin{proof}
First, let $\cL_1$ be a directed landscape, independent of $\cL$ and define $\cL_2$ so that
\begin{itemize}
	\item $\cL_2(x, s; y, t) = \cL_1(x, s; y, t)$ if $t \le 0$.
	\item $\cL_2(x, s; y, t) = \cL(x, s; y, t)$ if $s \ge 0$.
	\item $\cL_2(x, s; y, t) = \max_{z \in \R} \cL_1(x, s; z, 0) + \cL(z, 0; y, t)$ if $s < 0 < t$.
\end{itemize}
The process $\cL_2$ is another directed landscape by the axioms in Theorem \ref{T:L-unique}. Therefore by Theorem \ref{T:landscape-recovery} and temporal shift invariance, we can reconstruct $\cL_2$ on the set
$
(\R \X (-\infty, 1])^2_\uparrow
$
from its Busemann functions
$
\cB^{0^k}(x^k, 1), k \in \N.
$
By the metric composition law for Busemann functions (Lemma \ref{L:mc-law-busemann-L}), we can build these using $\cL_1$ and the extended Airy sheet $\cS(\bx, \by) = \cL(\bx, 0; \by, 1)$. By \cite[Theorem 1.3]{dauvergne2021disjoint}, the entire extended Airy sheet $\cS$ is almost surely a measurable function of $\mathfrak A$. Therefore since $\cL|_{\mathbb S^2_\uparrow}$ is contained as a restriction of $\cL_2|_{(\R \X (-\infty, 1])^2_\uparrow}$, we have shown that there is a measurable function $g$ such that 
$$
g(\cL_1, \mathfrak A) = \cL|_{\mathbb S^2_\uparrow}.
$$
Since $\cL_1$ is independent of $\cL$, this implies the corollary by the abstract Lemma \ref{L:independent-construction} below.
\end{proof}

\begin{lemma}
	\label{L:independent-construction}
Suppose that $X, Y, Z$ are random variables taking values in three Polish spaces $S_1, S_2, S_3$ equipped with their Borel $\sig$-algebras. Assume that $X$ is independent of the pair $(Y, Z)$ and that there is a measurable function $g:S_1 \X S_2 \to S_3$ such that almost surely 
$
g(X, Y) = Z.
$
Then there is a measurable function $g':S_2 \to S_3$ such that almost surely $g'(Y) = Z$.
\end{lemma}

\begin{proof}
For a Borel set $A \subset S_3$, define $Z^A_Y := \P(Z \in A \mid Y)$ and similarly define $Z^A_{X, Y} := \P(Z \in A \mid X, Y)$.
It is enough to show that almost surely, $Z^A_Y \in \{0, 1\}$ for any Borel set $A \subset S_3$. This is true for $Z^A_{X, Y}$ since $Z = g(X, Y)$ so we just need to show that $Z^A_Y = Z^A_{X, Y}$ almost surely. Indeed, for any product of Borel sets $B_1 \X B_2 \in S_1 \X S_2$ we have
\begin{align*}
\E[Z^A_Y \mathbf{1}(X \in B_1, Y \in B_2)] &= \P(X \in B_1) \E[Z^A_Y \mathbf{1}(Y \in B_2)] \\
&= \P(X \in B_1)\P(Z \in A, Y \in B_2)\\
 &= \P(X \in B_1, Z \in A, Y \in B_2).
\end{align*}
Here the first and third equalities use independence of $X$ and $(Y, Z)$ and the second equality is the definition of conditional expectation. Now, we can extend this by applying Dynkin's $\pi-\la$ theorem to get that for all Borel sets $B \subset S_1 \X S_2$ we have $\E[Z^A_Y \mathbf{1}((X, Y) \in B)) = \P(Z \in A, (X, Y) \in B)$. Since $Z^A_Y$ is also $\sig(X, Y)$-measurable since it is $\sig(Y)$-measurable this implies that $Z^A_Y$ is a version of the conditional probability $Z^A_{X, Y}$, as desired.
\end{proof}

\subsection{Single-slit and double-slit Busemann functions}
\label{SS:slit}

In the remainder of Section \ref{S:reconstruction} we will prove Theorem \ref{T:landscape-recovery}.2. The starting point for the proof lies in the observation that certain semi-infinite optimizers are frozen in the Brownian environments $B^a$, and so differences of certain Busemann functions give Busemann functions for paths in restricted domains. Under the limit where $a \to -\infty$ and $\cL^a \cvgd \cL$ we can show that the same phenomenon holds in $\cL$, allowing us to construct `single-slit' and `double-slit' Busemann functions. Let us explain more precisely what we mean.

In the Brownian environments $B^a$, the leftmost optimizer $\pi = (\pi_1, \dots, \pi_{k+1})$ in direction $(0^k, \theta)$ ending at $((x^k, y), 1)$ has the following specific form:
\begin{enumerate}[i.]
	\item $\pi_i(z) = i$ for all $z \le x, i \le k$.
	\item $\pi_{k+1}$ is a semi-infinite geodesic restricted to avoid the set $D_{x, k} := (-\infty, x) \times \{1, \dots, k\}$. That is, for any $t < z$, the path $\pi|_{[t, z]}$ has maximal length among all path $\tau$ from $(t, \pi(t))$ to $(z, 1)$ with $\mathfrak g \tau \cap D_{x, k} = \emptyset$, where $\mathfrak g \tau = \{(r, \tau(r)) : r \in [t, z]\}$.
\end{enumerate}
Indeed, for fixed $x$, then for large enough $y$, the leftmost semi-infinite geodesic $\tau$ to $(y, 1)$ in direction $\theta$ will be disjoint from the paths $\pi_1, \dots, \pi_k$ above. Hence $(\pi_1, \dots, \pi_k, \tau)$ will be an optimizer since both $(\pi_1, \dots, \pi_k)$ and $\tau$ are separately optimizers. Monotonicity of optimizers then implies that even if $y$ is not large enough so that the semi-infinite geodesic to $(y, 1)$ is disjoint from $\pi_k$, the leftmost optimizer $\pi$ in direction $(0^k, \theta)$ ending at $((x^k, y), 1)$ must take the above form.
Since the paths $(\pi_1, \dots, \pi_k)$ also form the optimizer from  from $(0^k, \infty)$ to $(x^k, 1)$ we can actually recover the relative length of $\pi_{k+1}$ from multi-path Busemann function values across the Brownian environment. This is immediate once we set up the definitions correctly.

Setting some notation, for $f \in \cC^\N(\R)$, points $(x, n), (y, m)$ with $x \le y$ and $n \ge m$, and a set $U \subset \R \X \Z$ such that there is at least one path $\pi$ from $(x, n) \to (y, m)$ with $\mathfrak g \pi \subset U$ define
$$
f[(x, n) \to (y, m) \mid U] = \sup_{\pi : \mathfrak g \pi \subset U} \|\pi\|_f,
$$
where the supremum is over all paths $\pi$ from $(x, n) \to (y, m)$ with $\mathfrak g \pi \subset U$. For $\theta \ge 0$ and $x \in \R$, we also define the restricted Busemann function
\begin{equation}
\label{E:BL-restricted}
\cB^\theta(x ; f \mid U) = \lim_{t \to \infty} f[(t, \fl{\theta t} + 1) \to (x, 1) \mid U] - f[(t, \fl{\theta  t} + 1) \to (0, 1)].
\end{equation}
In \eqref{E:BL-restricted}, the term we subtract off does not depend on the set $U$ that we are restricting to. The implies that restricted Busemann functions $\cB^\theta(x ; f \mid U)$ are monotone in $U$ with respect to set inclusion.

\begin{lemma}
\label{L:restricted-buse}
Fix $k \in \N, x \le y$ and $\theta > 0$, and let $B \in \cC^\N (\R)$ be an environment of independent Brownian motions of drift $0$. Then almost surely,
$$
\cB^{(0^k, \theta)}((x^k, y) ; B) - \cB^{0^k}(x^k ; B) = \cB^\theta(y ; B \mid D_{x, k}^c).
$$
\end{lemma}

\begin{proof}
Setting some notation, let $\cB_t^\theta(x ; f \mid U)$ denote the right-hand side of \eqref{E:BL-restricted} before taking a limit in $t$. Similarly, let $\cB_t^{\theta}(\bx ; f)$ denote the right-hand side of \eqref{E:bar-B} before taking a limit in $t$. Then for $t \le x$ we have 
$$
\cB^{(0^k, \theta)}_t((x^k, y) ; B) - \cB^{0^k}_t(x^k ; B) = \cB^\theta_t(y ; B \mid D_{x, k}^c).
$$
Taking a limit in $t$ then gives the result.
\end{proof}

Lemma \ref{L:restricted-buse} shows that certain restricted Busemann functions for the Brownian environments $B^{a}$  (and hence also $\cL^a$) can be constructed from the multi-path Busemann process for $\cL$. We will examine what this reveals about the environments $\cL^a$ as $a \to -\infty$. 
Consider a point $(x, t) \in \R \X (-\infty, 0)$, and look at the set 
$$
\{(x', t') \in \R \X (-\infty, 0] : (x', t')_a \in D_{(x, t)_a} \}.
$$
As $a \to -\infty$, this set converges to the ray $(-\infty, x] \X \{t\}$. This suggests that there should be a way to take a limit of Lemma \ref{L:restricted-buse} to recover a Busemann function in the directed landscape for paths restricted to the slit
$$
S_{x, t}^- := \R \X (-\infty, 0] \smin (-\infty, x) \X \{t\}.
$$
In fact, by a symmetry argument we will be able to prove something stronger: not only can we approximate single-slit Busemann functions for $\cL$, but we can also approximate \textbf{double-slit Busemann functions}. These are Busemann functions of the form
$$
\fB^\theta(y ; \cL \mid S_u) := \lim_{s \to -\infty} \cL(\theta|s|, s; y, 0 \mid S_u) -  \cL(\theta|s|, s; 0, 0),
$$	
where $u = (x, s; y, t)$ and $U(u) = S^-_{x, t} \cap S^+_{x, t}$, where $S^+_{y, t} = \R^2 \setminus ((y, \infty) \X \{t\})$. Here $\theta, y \in \R$, and just as in finite LPP, for a set $S$ and $(p; q) = (x, s; y, t) \in \Rd$, we write
$$
\cL(p; q \mid S) = \sup_{\pi:p \to q, \mathfrak g \pi \subset S} \|\pi\|_\cL,
$$
where the supremum is over all paths $\pi$ from $p$ to $q$ with $\mathfrak g \pi := \{(r, \pi(r)) : r \in [s, t]\} \subset S\}.$ We have omitted the time coordinate from the $\fB^\theta$ notation since it will always be $0$. We also omit the landscape $\cL$ when it is clear from context.

 The next proposition states our precise result.

\begin{proposition}
	\label{P:double-slit-reconstruction}
	Fix $u = (x, s; y, t) \in \Rd$ with $t < 0$, and let $z, \mu \in \R$. For $n \in \N$, define
	$$
	x_n = -|s|^{2/3} n^{1/3}/2 + x, \quad s_n = -|s|^{-1/3} n^{1/3}, \quad y_n = |t|^{2/3} n^{1/3}/2 + y, \quad t_n = |t|^{-1/3} n^{1/3},
	$$
	and as usual let $x_n^n, s_n^n, y_n^n, t_n^n$ denote vectors of repeated elements in $\R^n$. Then the following convergence holds in probability:
	\begin{equation}
		\label{E:lim-ninfty}
		\lim_{n \to \infty} \mathfrak B^{(s_n^n, \mu, t_n^n)}(x_n^n, z, y_n^n ; \cL) - \mathfrak B^{(s_n^n, t_n^n)}(x_n^n, y_n^n ; \cL) = \mathfrak B^\theta(z ; \cL \mid S_u).
	\end{equation}
\end{proposition}

We will prove Proposition \ref{P:double-slit-reconstruction} in Section \ref{SS:multi-path-to-restricted} and \ref{SS:optimizers-rays}. In the remainder of this section, we use Proposition \ref{P:double-slit-reconstruction} to prove Theorem \ref{T:landscape-recovery}.2. This requires a straightforward metric composition law for double-slit Busemann functions. We state this law together with a metric composition law for single-slit Busemann functions, which we will need later on.
\begin{lemma}
	\label{L:metric-comp}
Almost surely, for all $\theta \in \cJ, u = (x, s; y, t) \in \Rd$ with $t < 0$ and $w \in \R$, the double-slit and single-slit Busemann functions $\mathfrak B^\theta(w \mid S_u), \mathfrak B^\theta(w \mid S_{x, s}^-)$ exist and we have the following metric composition laws:
	\begin{align}
		\label{E:mc-Buse}
		\mathfrak B^\theta(w \mid S_u) &= \max_{z_1 \ge x, z_2 \le y} \mathfrak B^\theta(z_1, s) + \cL(z_1, s; z_2, t) + \cL(z_2, t; w, 0). \\
		\label{E:mc-Buse-single}
		\mathfrak B^\theta(w \mid S_{x, s}^-) &= \max_{z_2 \le y} \mathfrak B^\theta(z_2, t) + \cL(z_2, t; w, 0).
	\end{align}
\end{lemma}

\begin{proof}
	We only prove \eqref{E:mc-Buse} as the proof of \eqref{E:mc-Buse-single} is similar but simpler.    
    Pick $\theta' \in \Z \cap (\theta, \infty)$.
		For every $n \in \Z$, let $\pi^n$ denote the leftmost semi-infinite geodesic to $(n, 0)$ in direction $\theta'$. By translation invariance of the directed landscape, there exists some random $N \in \Z$ such that if $n \ge N$ then $\pi^n(s) > x$. Therefore for $n > N$ we have:
		\begin{align*}
&\fB^\theta(w \mid S_u) \\&= \lim_{r \to -\infty} \max_{z_1 \ge x, z_2 \le y} \fB(\theta|r|, r; z_1, s) - \fB(\theta|r|, r; 0,0) + \cL(z_1, s; z_2, t) + \cL(z_2, t; w, 0) \\
&= \lim_{r \to -\infty} \max_{z_1 \in [x, \pi^n(s)], z_2 \le y} \fB(\theta|r|, r; z_1, s) - \fB(\theta|r|, r; 0,0) + \cL(z_1, s; z_2, t) + \cL(z_2, t; w, 0) \\
 &= \max_{z_1 \in [x, \pi^n(s)], z_2 \le y} \fB^\theta(z_1, s) + \cL(z_1, s; z_2, t) + \cL(z_2, t; w, 0)
\end{align*}
Here the first equality is simply the metric composition law for $\cL$. The second equality uses that for small enough $r$ we have $\pi^n(r) > \theta|r|$, and so also using that $\pi^n(s) > x$, the geodesic from $(\theta|r|, r)$ to $(w, 0)$ that stays in $S_u$ will be to the left of $\pi^n$, and so in the metric composition law we can maximize over $z_1 \le \pi^n(s)$. The third equality uses the compact convergence of Busemann functions (Theorem \ref{T:sepp-sor-busani}.4).  Taking $n \to \infty$, we have $\pi^n(s) \to \infty$ by translation invariance of $\cL$, yielding \eqref{E:mc-Buse}.
\end{proof}

\begin{proof}[Proof of Theorem \ref{T:landscape-recovery}.2 given Proposition \ref{P:double-slit-reconstruction}]
Proposition \ref{P:landscape-law} and Theorem \ref{T:Busemann-shear} imply that we can almost surely reconstruct the whole multi-path Busemann process $(\theta, \bx) \mapsto \fB^\theta(\bx, 0; \cL)$ from $B^0$. Moreover, Proposition \ref{P:double-slit-reconstruction} implies that for all rational choices of $w, \theta$, and $u = (x, s; y, t)$ with $t < 0$, we can almost surely reconstruct the double-slit Busemann functions $\mathfrak B^\theta(w \mid S_u)$ from $(\theta, \bx) \mapsto \fB^\theta(\bx, 0; \cL)$. 
	
	Now, fix rational points $s < t < 0$ and introduce the shorthand $\{x, y\} = (x, s; y, t)$. We can define the \textbf{shock measure} $\mu_{s, t}$ for the time interval $(s, t)$ by the following formula:
	\begin{equation}
		\label{E:shock-measure}
		\mu_{s, t}([x_1, x_2] \X [y_1, y_2]) = \cL\{x_2, y_2\} + \cL\{x_1, y_1\} - \cL\{x_1, y_2\} - \cL\{x_2, y_1\}. 
	\end{equation}
	This defines a positive measure by the quadrangle inequality for $\cL$. Moreover, given the measure $\mu_{s, t}$, we can reconstruct the Airy sheet $\cL\{\cdot, \cdot\}$ via ergodicity of the Airy process $\cL(0,0; x, 1) + x^2$ (see \cite[Section 5]{prahofer2002scale}). Indeed, for any fixed $x_1 < x_2$ and $y_1$, the following holds almost surely:
	\begin{align*}
		\lim_{n \to \infty} \frac{1}{n} &\sum_{j=1}^n [\mu_{s, t}([x_1, x_2] \X [y_1, y_1 +j]) - \E \mu_{s, t}([x_1, x_2] \X [y_1, y_1 +j])] \\
		&= \cL\{x_1, y_1\} - \cL\{x_2, y_1\} - \E (\cL\{x_1, y_1\} - \cL\{x_2, y_1\})  \\
		&+ \lim_{n \to \infty} \frac{1}{n} \sum_{j=1}^n (\cL\{x_2, y_1 + j\} - \cL\{x_1, y_1 + j\}) - \E ((\cL\{x_2, y_1 + j\} - \cL\{x_1, y_1 + j\})) \\
		&= \cL\{x_1, y_1\} - \cL\{x_2, y_1\} - \E (\cL\{x_1, y_1\} - \cL\{x_2, y_1\}).
	\end{align*}
	Here the final equality uses ergodicity of the two shifted and rescaled Airy processes $\cL\{x_1, \cdot \} - \E \cL\{x_1, \cdot\}$ and $\cL\{x_2, \cdot \} - \E \cL\{x_2, \cdot\}$. This equality holds almost surely, simultaneously for all rational $x_1, x_2, y_1$. Next, using ergodicity of the shifted and rescaled Airy process $\cL\{\cdot, y_1\} - \E \cL\{\cdot, y_1\}$ we have that almost surely
	$$
	\lim_{n \to \infty} \frac{1}n \sum_{i=1}^n \cL\{x_1, y_1\} - \cL\{x_1 + i, y_1\} - \E(\cL\{x_1, y_1\} - \cL\{x_1 + i, y_1\}) = \cL\{x_1, y_1\} - \E \cL\{x_1, y_1\}.
	$$
	Therefore almost surely, for all rational $x_1, y_1$ the value $\cL\{x_1, y_1\}$ can be reconstructed from $\mu_{s, t}$ and hence by continuity, so can the whole Airy sheet $\cL\{\cdot, \cdot\}$. 
	
	Therefore if we can reconstruct $\mu_{s, t}([x_1, x_2] \X [y_1, y_2])$ almost surely for any rational $s, t, x_1, x_2, y_1, y_2$ from the collection of double-slit Busemann functions with rational parameters, then we can reconstruct all increments $\cL(x, s; y, t)$ for $x, y \in \R$ and $s < t \in \Q \cap (-\infty, 0)$  and hence the full directed landscape on $\mathbb H^2_\uparrow$ by continuity.
	
	Fix $s, t, x_1, x_2, y_1, y_2$, and define
	$$
	X^a = \fB^{-a}(a \mid S_{\{x_2, y_2\}}) + \fB^{-a}(a \mid S_{\{x_1, y_1\}}) - \fB^{-a}(a \mid S_{\{x_1, y_2\}}) - \fB^{-a}(a \mid S_{\{x_2, y_1\}}).
	$$
	We claim that in probability we have the convergence
	\begin{equation}
		\label{E:double-slit-limit}
		\lim_{a \to \infty} X^a = \mu_{s, t}([x_1, x_2] \X [y_1, y_2]),
	\end{equation}
	from which it follows that $\mu_{s, t}([x_1, x_2] \X [y_1, y_2])$ is a measurable function of the double slit Busemann environment. Note that with a bit more work, it can be shown that \eqref{E:double-slit-limit} in fact holds almost surely.
	
	Fix $x \in \{x_1, x_2\}, y \in \{y_1, y_2\}$. By Lemma \ref{L:metric-comp} we have that
	\begin{align}
		\label{E:fBa-Sxy}
		\fB^{-a}(a \mid S_{\{x, y\}}) = \max_{z_1 \ge 0, z_2 \le 0} \fB^{-a}(z_1 + x, s) + \cL\{z_1 + x, z_2 + y\} + \cL(z_2 + y, t; a, 0).
	\end{align}
	We aim to show that
	\begin{equation}
		\label{E:a-Sxyfba}
		\fB^{-a}(a \mid S_{\{x, y\}}) = \fB^{-a}(x, s) + \cL\{x, y\} + \cL(y, t; a, 0) + Y^a(x, y),
	\end{equation}
	where the error $Y^a(x, y)$ converges to $0$ in distribution as $a \to \infty$ for any fixed $x, y$. Indeed, we have the following symmetries of $\cL$, see Lemma \ref{L:invariance}.
	\begin{align*}
		\cL(z + y, t; a, 0) - \cL(y, t; a, 0) &\eqd t^{1/3} A(t^{-2/3} z) + \frac{2z(a - y) - z^2}{t}, \\
		\qquad \fB^{-a}(z + x, s) - \fB^{-a}(x, s) &\eqd B(z) - 2a z.
	\end{align*}
	Here $A$ is a (stationary) Airy$_2$ process, and $B$ is a Brownian motion without drift,
	and both equalities are jointly as continuous functions of $z$. These identities, together with the shape bound on $\cL\{x, y\}$ (Lemma \ref{L:shape-land}) imply that $Y^a_{x, y} \to 0$ in probability as $a \to \infty$, as desired. The claim \eqref{E:double-slit-limit} then follows immediately from \eqref{E:a-Sxyfba} and \eqref{E:shock-measure}.
\end{proof}

\subsection{From multi-path to restricted Busemann functions}
\label{SS:multi-path-to-restricted}

In the remainder of Section \ref{S:reconstruction} we prove Proposition \ref{P:double-slit-reconstruction}. We start with  two geometric results that will allow us to connect multi-path Busemann functions to restricted Busemann functions in $\cL$.

\begin{lemma}
	\label{L:strong-quadrangle}
	Let $(\bx, s; \by, t) \in \mathfrak X_\uparrow$, and let $\pi = (\pi_1, \dots, \pi_k)$ be the a.s.\ unique optimizer from $(\bx, s)$ to $(\by, t)$. Fix  $m \in \{0, \dots, k\}$, and consider $x_m \le a \le x_{m+1}$ and $y_m \le b \le y_{m+1}$, where we use the convention that $x_0 = y_0 = -\infty$ and $x_{k+1} = y_{k+1} = \infty$. Let 
	\begin{align*}
	R &= \{(z, r) \in \R \X [s, t] : \pi_m(r) < z < \pi_{m+1}(r)\},
	\end{align*}
	and let $\bar R$ be the closure of $R$. Here again we use the convention that $\pi_0 = -\infty, \pi_{k+1} = \infty$. 
	Let $\bx_a = (x_1, \dots, x_m, a, x_{m+1}, \dots, x_k) \in \R^{k+1}$ and similarly define $\by_b$. Then a.s.\
	\begin{align}
	\label{E:xsyt-cha}
	\cL(\bx, s; \by, t) + \cL(a, s; b, t \mid R) \le \cL(\bx_a, s; \by_b, t) \le \cL(\bx, s; \by, t) + \cL(a, s; b, t \mid \bar R).
	\end{align}
\end{lemma}

\begin{proof}
First, observe that if $\tau$ is any path from $(a, s)$ to $(b, t)$ that stays in $R$, then $(\pi_1, \dots, \pi_m, \tau, \pi_{m+1}, \dots, \pi_k)$ is a disjoint optimizer from $(\bx_a, s)$ to $(\by_b, t)$, and so
$$
\|\tau\|_\cL + \sum_{i=1}^k \|\pi_i\|_\cL \le \cL(\bx_a, s; \by_b, t),
$$
and the first inequality in \eqref{E:xsyt-cha} follows. Next, let $\tilde \pi = (\tilde \pi_1, \dots, \tilde \pi_m, \tau, \tilde \pi_{m+1}, \dots, \tilde \pi_k)$ be a disjoint $k$-tuple from $(\bx_a, s)$ to $(\by_b, t)$. By monotonicity of optimizers (Lemma \ref{L:mono-tree-multi-path-L}), have that $\pi_m \le \tau \le \pi_{m+1}$, so $\|\tau\|_\cL \le \cL(a, s; b, t \mid \bar R)$. Also, $\sum_{i=1}^k \|\tilde \pi_i\|_\cL \le \cL(\bx, s; \by, t)$, and the second inequality follows.
\end{proof}

Lemma \ref{L:strong-quadrangle} has the following consequence for Busemann functions in $\cL$. For this proposition, we say that a sequence of functions $\pi_n:[0, \infty) \to [-\infty, \infty]$ converges to a lower semicontinuous limit $\pi:[0, \infty) \to  [-\infty, \infty]$ in the \textbf{epigraph topology} and write $\pi_n \to_{\mathfrak e} \pi$ if:
\begin{equation}
\label{E:epigraph-cvg}
\inf_I \pi_n \to \inf_I \pi
\end{equation}
as $n \to \infty$
for all open closed intervals $I = [a, b]$ with $0 \le a < b < \infty$. The name comes of the fact that this convergence is equivalent to Hausdorff convergence of the epigraphs of $\pi_n$ to the epigraph of $\pi$. We require that $\pi$ be lower semicontinuous so that the limit is unique. We say that $\pi_n$ converges to an upper semicontinuous function $\pi$ in the \textbf{hypograph topology} and write $\pi_n \to_{\mathfrak h} \pi$ if \eqref{E:epigraph-cvg} holds with $\sup$ in place of $\inf$. These topologies make the space of lower/upper semicontinuous functions from $[0, \infty) \to [-\infty, \infty]$ compact.

\begin{proposition}
	\label{P:busemann-quadrangle}
	Consider deterministic sequences $\bx^n, \la^n \in \R^n_\le$ and $\by^n, \kappa^n \in \R^n_\le$. Let $\pi^n$ be the a.s.\ unique semi-infinite optimizer in direction $\la^n$ to $(\bx^n, 0)$ and let $\tau^n$ be the a.s.\ unique semi-infinite optimizer in direction $\ka^n$ to $(\by^n, 0)$. Next, for $x \in \R, t < 0$, let $R_{x, t}^\pm:[0, \infty) \to [-\infty, \infty]$ be the function with $R_{x, t}^\pm(t) = x$ and $R_{x, t}^\pm(r) = \pm \infty$ for $r \ne t$, and suppose that the following conditions hold almost surely:
	\begin{itemize}
		\item For some $u = (x, s; y, t) \in \Rd$ with $t < 0$, we have $\pi^n_n \to_{\mathfrak h} R^-_{x, s}$ and $\tau^n_1 \to_{\mathfrak e} R^+_{y, t}$ as $n \to \infty$.
		\item For any $\al \in \R$, there exists $n_0 \in \N$ such that 
		$
		\pi^n_n(r) < \al r < \tau^n_1(r)  
		$
		for all $r \le s - 1$ and $n \ge n_0$.
	\end{itemize}  
	Then for any $\mu, z \in \R$, the following two convergences hold a.s.
	\begin{align}
	\label{E:slit-convergence}
	\lim_{n \to \infty} \mathfrak B^{(\la^n, \mu, \kappa^n)}(\bx^n, z, \by^n) - \mathfrak B^{(\la^n, \kappa^n)}(\bx^n, \by^n) &= \mathfrak B^{\mu}(z \mid S_u) \\
	\label{E:single-slit-convergence}
	\lim_{n \to \infty} \mathfrak B^{(\mu, \kappa^n)}(z, \by^n) - \mathfrak B^{\kappa^n}(\by^n) &= \mathfrak B^{\mu}(z \mid S_{y, t}^+).
	\end{align}
\end{proposition}

To prove Proposition \ref{P:busemann-quadrangle}, we need two easy preliminary results. 

\begin{lemma}
	\label{L:one-side-bound}
Let $u =(x, s; y, t) \in \Rd$, and let $R \subset [s, t] \times \R$. Let $[s', t'] \subset [s, t]$, and let $\tau:[s', t'] \to \R$ be any geodesic with $\mathfrak g \tau \subset R$. Then we can write
$
\cL(x, s; y, t \mid R) = \sup_\pi \|\pi\|_\cL,
$
where the supremum is over all paths $\pi$ for which the set
$$
\{r \in [s, t] : \pi(r) = \tau(r)\}
$$
is a (possibly empty) closed interval.
\end{lemma}

\begin{proof}
Suppose that $\pi$ is any path from $(x, s)$ to $(y, t)$ with $\mathfrak{g} \pi \subset R$ and non-empty intersection with $\tau$. Let $r_1 < r_2 \in [s, t]$ be the first and last times when $\pi$ and $\tau$ agree. Since $\tau$ is a geodesic, we have that
$$
\|\pi\|_\cL \le \|\pi|_{[s, r_1]} \oplus \tau|_{[r_1, r_2]} \oplus \pi|_{[r_2, t]}\|_\cL,
$$
and since $\tau$ stays in $R$, so does the path $\pi|_{[s, r_1]} \oplus \tau|_{[r_1, r_2]} \oplus \pi|_{[r_2, t]}$. Here $\oplus$ denotes paths concatenation. Since this path overlaps with $\tau$ on a closed interval, the lemma follows.
\end{proof}

\begin{lemma}
	\label{L:stays-inside}
Fix $u = (x, s; y, t)$ with $t < 0$ and a point $(z, r)$ with $r < s$ and let $w \in \R$. Then almost surely, the restricted landscape values $\cL(z, r; w, 0 \mid S_u), \cL(z, r; w, 0 \mid S_{y, t}^+)$ are achieved by paths that stay in the interior of the sets $S_u, S_{y, t}^+$ respectively.
\end{lemma}

\begin{proof}
We only treat the $S_u$ case, as the $S_{y, t}^+$ case is similar but simpler.	
By metric composition,
\begin{equation}
\label{E:Lzr-Lzr}
\cL(z, r; w, 0 \mid S_u) = \max_{z_1 \ge x, z_2 \le y} \cL(z, r; z_1, s) + \cL(z_1, s; z_2, t) + \cL(z_2, t; w, 0).
\end{equation}
This maximum must be achieved by the landscape shape theorem (Lemma \ref{L:shape-land}). To complete the proof, it suffices to show that this maximum is achieved at a point $(Z_1, Z_2)$ with $Z_1 > x, Z_2 < y$ almost surely. 
We show this for $Z_1$ as the proof for $Z_2$ is similar. First, for all $n \in \N$, the maximization problem
 $$
\max_{z_1 \ge x} \cL(z, r; z_1, s) + \cL(z_1, s; -n, t)
 $$
is almost surely achieved uniquely at a point $Y_n > x$. This follows since both the functions $z_1 \mapsto \cL(z, r; \cdot, s), z_1 \mapsto \cL(z_1, s; -n, t)$ are locally absolutely continuous with respect to Brownian motion (see \cite{CH}) and independent. Now, monotonicity of geodesics implies that if $(Z_1, Z_2)$ is any pair achieving the maximum in \eqref{E:Lzr-Lzr}, then $Z_1 \ge Y_n$ whenever $-n < Z_2$. Hence $Z_1 < x$ almost surely, as desired.
\end{proof}

\begin{proof}[Proof of Proposition \ref{P:busemann-quadrangle}]
	We only treat the double-slit convergence \eqref{E:slit-convergence}, as the single-slit convergence \eqref{E:single-slit-convergence} is similar but simpler.
	We first use Lemma \ref{L:strong-quadrangle} to bound the difference $\mathfrak B^{(\la^n, \mu, \kappa^n)}(\bx^n, z, \by^n) - \mathfrak B^{(\la^n, \kappa^n)}(\bx^n, \by^n)$ above and below. 
	
	First, observe that the two bullet points in the lemma guarantee that $x^n_n, \la^n_n \to -\infty$ as $n \to \infty$ and $y^n_1, \kappa^n_1 \to \infty$ as $n \to \infty$. Moreover, these bullets implies that $\pi^n, \tau^n$ are disjoint for large enough $n$, and so $(\pi^n, \tau^n)$ is a semi-infinite disjoint optimizer in direction $(\la^n, \kappa^n)$ to $((\bx^n, \by^n); 0)$. Now let
	$$
	\tilde \pi^n_\mu(r) = (\pi^n(r), \mu |r|, \kappa^n(r)), \qquad \bz^n := (\bx^n, z, \by^n).
	$$
	By Corollary \ref{C:hat-B-Busemann-landscape}, we have that almost surely:
	\begin{equation}
	\label{E:mathfrak-Bla}
	\begin{split}
&\mathfrak B^{(\la^n, \mu, \kappa^n)}(\bx^n, z, \by^n) - \mathfrak B^{(\la^n, \kappa^n)}(\bx^n, \by^n) \\
= &\lim_{r \to -\infty} \cL(\tilde \pi^n_\mu(r), r; \bz^n, 0) - \cL((\pi^n, \kappa^n)(r), r; (\bx^n, \by^n), 0) - \cL(\mu |r|, r; z, 0).
	\end{split}
	\end{equation}
	Next, define
	$$
	R_n := \{(z, r) \in \R \X (-\infty, 0] : \pi^n_n(r) < z < \tau^n_1(r)\}.
	$$
	By Lemma \ref{L:strong-quadrangle}, \eqref{E:mathfrak-Bla} is bounded above by
	\begin{equation}
	\label{E:limsup-above}
	X_n^+ := \limsup_{r \to -\infty} \cL(\mu |r|, r; z, 0 \mid \bar R_n) - \cL(\mu |r|, r; z, 0),
	\end{equation}
	and bounded below by
		\begin{equation}
	\label{E:liminf-below}
	X_n^- := \liminf_{r \to -\infty} \cL(\mu |r|, r; z, 0 \mid R_n) - \cL(\mu |r|, r; z, 0).
	\end{equation}
	To complete the proof, we just need to show that for all large enough $n$:
	\begin{equation}
	\label{E:X-limit}
X_n^- = X_n^+ = \mathfrak B^{\mu}(z \mid S_u).
	\end{equation}
	Now, we can find points $z_- < z < z_+ \in \R$ such that there are leftmost semi-infinite geodesics $\pi_\pm$ in direction $\mu \pm 1$ to $(z_\pm, 0)$ such that
	\begin{equation}
\pi_-(s) < x < \pi_+(s), \qquad  \pi_-(t) < y < \pi_+(t).
	\end{equation}
	Now, let $M > 0$ be large enough so that $\pi^-(r) \le \mu |r| \le \pi^+(r)$ for all $r \le -M$. By the second bullet point, there exists $M' > M$ such that $\pi_\pm|_{(-\infty, -M']} \subset R_n$ for all large enough $n$. Therefore by Lemma \ref{L:one-side-bound}, for large enough $n$, in the definition of 
	$
	\cL(\mu|r|, r; z, 0 \mid A)
	$
	where $A = R_n, \bar R_n, S_u$,
	it suffices to consider paths that stay in the set
	$$
	G := \{(w, r') \in \R \X (-\infty, -M'] : \pi_-(r') \le w \le \pi_+(r')\} \cup (\R\X [-M', 0]).
	$$
	Next, by Lemma \ref{L:stays-inside}, for each $m \in \N$ almost surely we can find paths $\tau_-, \tau_+$ that achieve the restricted landscape values
	$\cL(\fl{\pi_-(-\cl{M'})}, -\cl{M'}; z_-, 0 \mid S_u)$ and 	$\cL(\cl{\pi_+(-\cl{M'})}, -\cl{M'}; z_+, 0 \mid S_u)$ and stay in the interior of the set $S_u$. The first bullet point then implies that $\mathfrak{g} \tau_\pm \subset R_n$ for all large enough $n$. Therefore for large enough $n$ and $r \ge M'$, it suffices to further restrict our paths to the region
	$$
	G' := G \cap (\{(w, r') \in \R \X (-M', 0] : \tau_-(r') \le w \le \tau_+(r')\} \cup \R\X (-\infty, -M'])
	$$
	in the definitions of $
	\cL(\mu|r|, r; z, 0 \mid A)
	$
	where $A = R_n, \bar R_n, S_u$.
	Noting that $G' \subset S_u \cap R_n \cap \bar R_n$ for all large enough $n$, we get that 
	$$
	\cL(\mu|r|, -r; z, 0 \mid R_n) = \cL(\mu|r|, -r; z, 0 \mid \bar R_n)= \cL(\mu|r|, -r; z, 0 \mid S_u)
	$$
	for $r \ge M'$ and all large enough $n$. This yields \eqref{E:X-limit}, completing the proof.
\end{proof}

\subsection{Optimizers converging to rays}
\label{SS:optimizers-rays}

To prove Proposition \ref{P:double-slit-reconstruction}, by the subsequential characterization of convergence in probability it is enough to show that for any subsequence $Y \subset \N$ that there is a further subsequence $Y' \subset Y$ such that \eqref{E:lim-ninfty} holds almost surely along $Y'$. 
By Proposition \ref{P:busemann-quadrangle}, we can show this by studying the shape of the optimizer $\pi^n$ in direction $s^n_n$ to $(x^n_n, 0)$ and the optimizer $\tau^n$ in direction $t^n_n$ to $(y^n_n, 0)$. These optimizers are a.s.\ unique by Lemma \ref{L:optimizer-as-unique}. More precisely, it is enough to show that for any subsequence $Y \subset \N$, there is a further subsequence $Y' \subset Y$ along which the two bullets in Proposition \ref{P:busemann-quadrangle} hold a.s. Using the subsequential characterization of convergence in probability and a Cantor diagonalization argument, this follows from the following proposition.

\begin{proposition}
	\label{P:slit-existence}
The following three claims hold:
	\begin{itemize}
	\item For any $a \in (s, 0)$ and $b \in (t, 0)$, as $n \to \infty$ we have
	$$
	\sup_{r \in [a, 0]} \pi^n_n(r) \cvgp -\infty, \qquad 	\inf_{r \in [b,0]} \tau^n_1(r) \cvgp \infty.
	$$
	\item For any $\al \in \R, \ep > 0$, 
	$$
	\lim_{n \to \infty} \P(\sup_{r \le s - \ep} \pi^n_n(r) - \al r > 0) = 0, \qquad \lim_{n \to \infty} \P(\inf_{r \le s - \ep} \tau^n_1(r) - \al r < 0) = 0.
	$$
	\item $\sup_{r \le 0} \pi^n_n(r) \cvgp x$ and $\inf_{r \le 0} \tau^n_n(r) \cvgp y$. 
\end{itemize} 
\end{proposition}
The proof of the claims in Proposition \ref{P:slit-existence} for $\pi^n$ and $\tau^n$ are symmetric, so we only prove the claims for $\pi^n$. Moreover, by spatial stationarity and KPZ scale invariance of $\cL$ it suffices to prove the proposition in the special case when $x = 0, s = -1$. In this case $x_n = -n^{1/3}/2, s_n = -n^{1/3}$. 

To prove Proposition \ref{P:slit-existence}, define $F^n:\R^2 \to \R$ by
$$
F^n(\la, w) = \fB^{((-n^{1/3})^n, \la)}((-n^{1/3}/2)^n,w) - \fB^{(-n^{1/3})^n}((-n^{1/3}/2)^n). 
$$
Then letting $\pi_{\la, w}$ denote the a.s.\ unique semi-infinite geodesic in $\cL$ in direction $\la$ to $(w, 0)$, we have the following simple observation.

\begin{lemma}
\label{L:Fnab}
For any $n \in \N$ and $\la, w \in \R$ with $\la > -n^{1/3}$ and $w > - n^{1/3}/2$ the following two events differ on a null set:
\begin{enumerate}
	\item the geodesic $\pi_{\la, w}$ is disjoint from the path $\pi^n_n$.
	\item $F^n(\la, w) = \fB^\la(w)$.
\end{enumerate}
\end{lemma}

\begin{proof}
This follows since 
$$
\fB^{((\la, (n^{1/3})^n))}(w, (n^{1/3}/2)^n) = \fB^{(n^{1/3})^n}((n^{1/3}/2)^n) + \fB^\la(w)
$$	
if and only if there is an optimizer in direction $(n^{1/3})^n$ to $((n^{1/3}/2)^n, 0)$ which is disjoint from a geodesic to $(w, 0)$ in direction $\la$. 
\end{proof}
We can study the joint law of $F^n(\la, w), \fB^\la(w)$ by using Theorem \ref{T:Busemann-shear} and the isometry in Theorem \ref{T:landscape-recovery}.1. This yields the following result.

\begin{lemma}
	\label{L:Bla-distribution}
Consider the function $G^n:\R^2 \to \R \X \{0, 1\}$ given by
$$
G^n(\la, w) = (F^n(\la, w), \mathbf{1}(F^n(\la, w) = \fB^\la(w)).
$$
Then as $n \to \infty$, the finite dimensional distributions of $G^n$ converge to those of 
$$
G(\la, 1) := (\cB^\la(w \mid S^-_{0, -1}), \indic(\cB^\la(w; \cL \mid S^-_{0, -1})) = \fB^\la(w)).
$$
\end{lemma}

\begin{proof}
By Theorem \ref{T:Busemann-shear} and the isometry in Theorem \ref{T:landscape-recovery}.1, we can rewrite $F^n$ in terms of the Brownian environment $B^{-2n^{1/3}}$. Indeed, by Lemma \ref{L:restricted-buse} we have
\begin{align*}
\fB^{((-n^{1/3})^n, \la)}&((-n^{1/3}/2)^n,w) - \fB^{(-n^{1/3})^n}((-n^{1/3}/2)^n) \\
&= \cW^{2(\la + n^{1/3})}(w; B^{-2n^{1/3}} \mid D^c_{(-n^{1/3}/2, n)}) \\
&= \max_{-n^{1/3}/2 \le z \le w} \cW^{2(\la + n^{1/3})}(z, n + 1; B^{-2n^{1/3}}) + B^{-2n^{1/3}}[(z, n) \to (w, 1)].
\end{align*}
We also have the identity
\begin{align*}
\fB^{\la}(w) &= \cW^{2(\la + n^{1/3})}(w, 1; B^{-2n^{1/3}}) \\
&= \max_{z \le w} \cW^{2(\la + n^{1/3})}(z, n + 1; B^{-2n^{1/3}}) + B^{-2n^{1/3}}[(z, n) \to (w, 1)].
\end{align*}
From these two formulas, the lemma then follows from the fact that 
$$
z \mapsto \cW^{2(\la + n^{1/3})}(z - n^{1/3}/2, n + 1; B^{-2n^{1/3}}) - \cW^{2(\la + n^{1/3})}(- n^{1/3}/2, n + 1; B^{-2n^{1/3}})
$$
is a Brownian motion of drift $2\la$ (e.g. Theorem \ref{T:brownian-law}), the convergence 
$$
B^{-2n^{1/3}}[(z- n^{1/3}/2, n) \to (w, 1)] -  n^{2/3} \cvgd \cL(z, -1; w, 0)
$$
in the compact topology on $\Z^2$ (Theorem \ref{T:DL-cvg-extended}), and the Brownian LPP shape theorems (Propositions \ref{P:cross-prob}, \ref{P:top-bd}) to control the argmax locations above. Here we have used that we can identify the law of $G$ from a Brownian motion and $\cL(z, -1; w, 0)$ using \eqref{E:mc-Buse-single}.
\end{proof}
	
\begin{proof}[Proof of Proposition \ref{P:slit-existence}]
Fix $r_0 > -1$ and $m \in \R$. Then by Lemma \ref{L:Fnab}, for $\la > -n^{1/3}, w > -n^{1/3}/2$ we have
$$
\P(\sup_{r \in [r_0, 0]} \pi^n_n(r) > m) \le \P(\sup_{r \in [r_0, 0]} \pi_{\la, w}(r) > m) + \P(F^n(\la, w) \ne \fB^\la(w))
$$
Taking $n \to \infty$ and using Lemma \ref{L:Bla-distribution} shows that the right-hand side above converges to
\begin{equation}
	\label{E:rrr}
\P(\sup_{r \in [r_0, 0]} \pi_{\la, w}(r) > m) + \P(\fB^\la(w; \cL \mid S^-_{0, -1}) \ne \fB^\la(w)).
\end{equation}
Taking $\la = k, w = -k + \log k$ in \eqref{E:rrr}, shear and spatial stationarity of $\cL$ (Lemma \ref{L:invariance}.2,4) implies that \eqref{E:rrr} converges to $0$ as $k \to \infty$. This yields the first bullet. The second bullet is similar. Indeed, by Lemma \ref{L:Fnab} and Lemma \ref{L:Bla-distribution} for any fixed $\la, w$ we have that
\begin{align*}
\limsup_{n \to \infty} &\P(\sup_{r \le -1 - \ep} \pi^n_n(r) - \al r > 0) \\
&\le \P(\sup_{r \le - 1 - \ep} \pi_{\la, w}(r) - \al r > 0) + \P(\fB^\la(w; \cL \mid S^-_{0, -1}) \ne \fB^\la(w)).
\end{align*}
Taking $\la = -k, w = k + \log k$, the right-hand side above converges to $0$ as $k \to \infty$. The final bullet requires a slightly different idea. For each $n$, let 
$$
R_n = \{(y, t) \in \R \X (-\infty, 0] : y > \pi^n_n(t)\}.
$$
By the first two bullet points in Proposition \ref{P:slit-existence}, the law of the set $R_n^c$ is tight with respect to the local Hausdorff topology on closed subsets of $[-\infty, \infty] \X (-\infty, 0]$.
 Moreover, any subsequential limit of this law is the law of a random ray $[-\infty, X] \times \{-1\}$, where $X$ is a $[-\infty, \infty]$-valued random variable. Note that the proof of the first bullet actually shows that $X < \infty$ almost surely, since in that proof we showed that
 $$
 \lim_{k \to \infty} \lim_{n \to \infty} \P(\pi_{-k, k + \log k}, \pi^n_n \text{ are disjoint}) = 1.
 $$
 To prove the final bullet, we must show that $X = 0$ almost surely. By Lemma \ref{L:strong-quadrangle} we have:
\begin{equation}
\label{E:cBacLR}
\fB^0(b; \cL \mid R_n) - \fB^{0}(2b; \cL \mid \bar{R}_n) \le F^n(0, b) - F^n(0, 2b) \le \fB^0(b; \cL \mid \bar{R}_n) - \fB^{0}(2b; \cL \mid R_n).
\end{equation}
Now, Lemma \ref{L:Bla-distribution} implies that for every fixed $w \in \R$ we have that
$$
F^n(0, w) - F^n(0, 2w) \cvgd [\sup_{x \ge 0} B(x) + \cL(x, -1; w, 0)] - [\sup_{x \ge 0} B(x) + \cL(x, -1; 2w, 0)],
$$
where $B$ is an independent Brownian motion. The symmetries of $\cL$ imply that the right-hand side above equals
$
Z(w) + 3w^2,$ where $Z(w), w \le 0$ is a tight collection of random variables. On the other hand, consider a subsequence along which all the random variables 
$$
R_n^c, \cL, \fB^0(w; \cL \mid R_n) - \fB^{0}(2w; \cL \mid \bar{R}_n), \fB^0(w; \cL \mid \bar{R}_n) - \fB^{0}(2w; \cL \mid R_n), w \in \Z
$$ jointly converge in distribution to limits $[-\infty, X] \times \{-1\}, \tilde \cL, Y_w^-, Y_w^+, w \in \Z$.

 Here we allow $Y_w^-, Y_w^+$ to possibly take on the values $\pm \infty$ so that we do not need to address tightness of the prelimiting sequence. As before, we allow $X$ to take on the value $-\infty$. For every $w \in \Z$ we have that
\begin{align*}
&[\sup_{x > X} \tilde B(x) + \tilde \cL(x, -1; w, 0)] - [\sup_{x \ge X} \tilde B(x) + \tilde \cL(x, -1; 2w, 0)] \le Y_w^- \le Y_w^+ \\
&\le [\sup_{x \ge X} \tilde B(x) + \tilde \cL(x, -1; w, 0)] - [\sup_{x > X} \tilde B(x) + \tilde \cL(x, -1; 2w, 0)],
\end{align*}
where $\tilde B(x) = \fB^0(x, -1 ; \cL)$ is another Brownian motion of variance $2$. Now, if $X = -\infty$ with positive probability, then on this event the metric composition law for Busemann functions implies that $Y_w^+ = Y_w^- = B(w) - B(2w),$ where $B$ is a Brownian motion. This contradicts the asymptotic growth of the random variables $Z(w) + 3w^2$. On the other hand, if $X > -\infty$, then the shape theorem for $\cL$ (Lemma \ref{L:shape-land}) then implies that as $w \to -\infty$ we have that $Y_w^-, Y_w^+$ are equal to $3w^2 - 3w X$ plus lower order terms. This contradicts the asymptotic growth of the random variables $Z(w) + 3w^2$ unless $X = 0$ almost surely.
\end{proof}

\bibliographystyle{alpha}
\bibliography{bibliography}

\end{document}